\documentclass[a4paper,11pt,reqno]{article}

\usepackage[utf8]{inputenc}
\usepackage[T1]{fontenc}
\usepackage{lmodern}
\usepackage[english]{babel}
\usepackage{microtype}

\usepackage{amsmath,amssymb,amsfonts,amsthm}
\usepackage{mathtools,accents}
\usepackage{mathrsfs}
\usepackage{aliascnt}
\usepackage{braket}
\usepackage{bm}
\usepackage{esint}

\usepackage[a4paper,margin=3cm]{geometry}
\usepackage[citecolor=blue,colorlinks]{hyperref}

\usepackage{enumerate}
\usepackage{xcolor}

\makeatletter
\g@addto@macro\@floatboxreset\centering
\makeatother

\makeatletter
\def\newaliasedtheorem#1[#2]#3{
  \newaliascnt{#1@alt}{#2}
  \newtheorem{#1}[#1@alt]{#3}
  \expandafter\newcommand\csname #1@altname\endcsname{#3}
}
\makeatother

\numberwithin{equation}{section}

\newtheoremstyle{slanted}{\topsep}{\topsep}{\slshape}{}{\bfseries}{.}{.5em}{}

\theoremstyle{plain}
\newtheorem{theorem}{Theorem}[section]
\newaliasedtheorem{proposition}[theorem]{Proposition}
\newaliasedtheorem{lemma}[theorem]{Lemma}
\newaliasedtheorem{corollary}[theorem]{Corollary}
\newaliasedtheorem{counterexample}[theorem]{Counterexample}

\theoremstyle{definition}
\newaliasedtheorem{definition}[theorem]{Definition}
\newaliasedtheorem{question}[theorem]{Question}
\newaliasedtheorem{example}[theorem]{Example}
\newaliasedtheorem{conjecture}[theorem]{Conjecture}

\theoremstyle{remark}
\newaliasedtheorem{remark}[theorem]{Remark}
\newaliasedtheorem{comment}[theorem]{Comment}

\newcommand{\setN}{\mathbb{N}}

\newcommand{\setR}{\mathbb{R}}

\newcommand{\eps}{\varepsilon}

\let\phi\varphi

\newcommand{\weakto}{\rightharpoonup}

\newcommand{\Id}{\mathrm{Id}}

\newcommand{\di}{\mathop{}\!\mathrm{d}}

\newcommand{\loc}{{\rm loc}}

\newcommand{\res}{\mathop{\hbox{\vrule height 7pt width .5pt depth 0pt
\vrule height .5pt width 6pt depth 0pt}}\nolimits}

\DeclareMathOperator{\supp}{supp}

\newcommand{\Ch}{{\sf Ch}}
\newcommand{\scal}[2]{\ensuremath{\langle #1 , #2 \rangle}} 
\newcommand{\Leb}{\mathrm{Leb}}

\DeclareMathOperator{\Lip}{Lip}

\newcommand{\haus}{\mathscr{H}}

\newcommand{\dist}{\mathsf{d}}

\newcommand{\meas}{\mathfrak{m}}

\DeclareMathOperator{\CD}{CD}
\DeclareMathOperator{\RCD}{RCD}
\DeclareMathOperator{\Span}{Span}
\DeclareMathOperator{\vol}{\mathrm{vol}}

\newfont{\tmpf}{cmsy10 scaled 2500}

\begin{document}
\title{Embedding of $\RCD^*(K,N)$ spaces in $L^2$ via eigenfunctions}
\author{Luigi Ambrosio
\thanks{Scuola Normale Superiore, \url{luigi.ambrosio@sns.it}} \and
Shouhei Honda
\thanks{Tohoku University, \url{shouhei.honda.e4@tohoku.ac.jp}} \and
 Jacobus W. Portegies
\thanks{Eindhoven University of Technology, \url{j.w.portegies@tue.nl}} \and
David Tewodrose
\thanks{CY Cergy Paris University,  \url{david.tewodrose@cyu.fr}}} \maketitle
\begin{center}
\textit{Dedicated to the memory of Professor Kazumasa Kuwada.}
\end{center}

\begin{abstract} In this paper we study the family of embeddings $\Phi_t$ of a compact $\RCD^*(K,N)$ space $(X,\dist,\meas)$ into $L^2(X,\meas)$
via eigenmaps. Extending part of the classical results \cite{Berard85,BerardBessonGallot} known for closed Riemannian manifolds, we prove 
convergence as $t\downarrow 0$ of the rescaled pull-back metrics $\Phi_t^*g_{L^2}$ in $L^2(X,\meas)$ induced by $\Phi_t$.
Moreover we discuss the behavior of $\Phi_t^*g_{L^2}$ with respect to measured Gromov-Hausdorff convergence and $t$.
Applications include the quantitative $L^p$-convergence in the noncollapsed setting for all $p<\infty$, a result new even for closed 
Riemannian manifolds and Alexandrov spaces.
\end{abstract}

\tableofcontents
\section{Introduction}

General Riemannian manifolds, defined through charts, could a priori have been much more complex than submanifolds of Euclidean space, but Nash's embedding theorem tells us that this is not the case: a general closed Riemannian manifold can always be isometrically embedded into a Euclidean space. This reduction in complexity is useful for many reasons, ranging from making it easier to think about Brownian motion on a Riemannian manifold, to opening up analytical tools when studying harmonic maps with a Riemannian manifold as a target. 

Approximately a decade ago, Sturm \cite{Sturm06}, Lott and Villani \cite{LottVillani} independently gave a meaning to a Ricci curvature lower bound and a dimension upper bound on metric measure spaces. It was at that time already well-known that lower bounds on the Ricci curvature ensure many key estimates in the analysis of geometric inequalities and partial differential equations on Riemannian manifolds. Sturm, Lott and Villani moved away from Riemannian manifolds and considered the general class of metric measure spaces, which includes weighted Riemannian manifolds.

The theory of metric measure spaces with generalized lower Ricci curvature bounds has been under rapid development, both in the general
classes singled out by Lott-Villani and Sturm, and in the class of $\RCD^*(K,N)$ spaces. Particularly in the latter, many classical results from Riemannian geometry have been carried over. We recall the precise definition of $\RCD^*(K,N)$ spaces in Subsection~\ref{se:rcdstarkn}.

A priori, $\RCD^*(K,N)$ spaces could be very complex. Certainly, they are more general than Riemannian manifolds, as they contain both the Gromov-Hausdorff limits of $N$-dimensional Riemannian manifolds with uniform Ricci curvature lower bounds and Alexandrov spaces. 
Currently, we do not know whether isometric embeddings of $\RCD^*(K,N)$ spaces into Euclidean spaces always exist, but in this paper
we study a relaxed version of this question. We will seek an embedding into a Hilbert space rather than a Euclidean space, and look for an 
embedding which is only approximately rather than precisely isometric. 

\subsection*{Importance in data analysis}

Another motivation for studying embeddings of $\RCD^*(K,N)$ spaces comes from data analysis.
Indeed, embeddings of \emph{data} into Euclidean space are an important tool in manifold-learning or non-linear dimensionality reduction \cite{LeeV07}. This is a branch unsupervised machine-learning tasked with finding a small set of relevant latent variables in a priori high-dimensional data. Eigenmaps \cite{BelkinN03} and Diffusion Maps \cite{CoifmanL06} are examples of manifold-learning algorithms that are closely related to the embeddings considered in this article.

While the merit of such embeddings is of course application-dependent, it often hinges on how well the embedding preserves distances.

To analyze the quality of embeddings of data, there are at least two good reasons to look at \emph{continuous} spaces. Not only do continuous metric measure spaces often provide a good model to approximate large amounts of data, also in many situations the data is sampled from a ``ground truth'' distribution which in fact forms a continuous space itself. 

For smooth Riemannian manifolds, classical theorems can be used to produce embeddings, but as a side-effect the quality of embeddings depends on high regularity of the manifold. 
Yet embeddings are also desired in situations where for instance bounds on high derivatives of the metric are not available, or worse, when the ground truth has singularities. To produce embeddings and guarantee their quality for Riemannian manifolds that depend only on relatively low-level geometric information such as curvature and dimension, or to construct embeddings for nonsmooth spaces, it is essential to understand whether and how certain maps embed metric measure spaces into Euclidean spaces. Our goal is to provide convergence results that depend only on lower bounds on curvature, volume and upper bounds 
on dimension and diameter, not using bounds on injectivity radius or derivatives of the metric; of course the price we pay is that convergence is understood
in weaker topologies. In order to obtain quantitative estimates, in the $\epsilon-\delta$ form, we argue by contradiction and, for this reason,
it is necessary to work in the compact category of $\RCD^*(K,N)$ spaces.

\subsection*{Embedding a manifold $M^n$ into $L^2$}

For a closed $n$-dimensional Riemannian manifold $(M^n, g)$ and a positive time $t > 0$, the map $\Phi_t: M^n \to L^2(M^n, \mathrm{vol}_g)$ is given by
\begin{equation}\label{eq:introPhi}
\Phi_t(x) = p(x, \cdot,t)
\end{equation}
where $p: M^n \times M^n \times (0,\infty) \to (0,\infty)$ is the heat kernel on $M^n$.
B\'erard, Besson and Gallot showed that the map $\Phi_t$ provides a smooth embedding of the Riemannian manifold $(M^n,g)$ \cite{Berard85,BerardBessonGallot}. In addition, they showed that it almost preserves distances. Their original result was phrased in terms of asymptotics
for the pullback metric of the metric on $L^2(M^n, \mathrm{vol}_g)$ as $t$ converges to $0$, namely
\begin{equation*}
c(n)t^{(n+2)/2} \Phi_t^*g_{L^2} =  g - \frac{2t}{3}\left(\mathrm{Ric}_g - \frac{1}{2}\mathrm{Scal}_g\,  g\right) + O(t^2), \quad t \downarrow 0.
\end{equation*}
This asymptotic expansion contains explicit curvature tensors on the right-hand side, which are not available in a nonsmooth context, or which cannot be bounded if only a lower-bound on the Ricci curvature is known. 

A slightly different approach is more robust, thus better suited in nonsmooth settings, and was used by the third author to obtain convergence results for the diffusion maps algorithm \cite{Portegies16}. 
Omitting for notational simplicity the $x$ dependence of $g$, the differential of the map $\Phi_t$ in the direction of the tangent vector $v \in T_x M^n$ is
\[
(d\Phi_t)_x(v) = g(v, \nabla_x p(x, \cdot, t)).
\]
Its length is therefore
\begin{equation}\label{eq:goodlimit1}
c(n) t^{(n+2)/2} \| (d\Phi_t)_x(v) \|^2_{L^2} = c(n) t^{(n+2)/2} \int_{M^n}  g(v , \nabla_x p(x,y,t) )^2 \di \mathrm{vol}_g(y).
\end{equation}
Now, every Riemannian manifold is locally Euclidean. For small enough $t$, the heat kernel localizes so strongly, that only a small neighborhood is probed in the integral at the right-hand side. Hence, this integral converges to its value in Euclidean space.

\subsection*{Embedding an $\RCD^*(K,N)$ space and our convergence results}

The analogous map $\Phi_t: X \to L^2(X, \meas)$ can also be constructed for a compact $\RCD^*(K,N)$ space $(X, \dist, \meas)$ and is 
still given by \eqref{eq:introPhi}.
The heat kernel $p$ exists and satisfies natural estimates: this follows from the theory of linear heat flow on an $\RCD^*(K,N)$ space 
$(X,\dist, \meas)$ developed by Gigli, Savar\'e and the first author \cite{AmbrosioGigliSavare13}, and from decay estimates on the heat 
kernel obtained by Jiang, Li and Zhang \cite{JiangLiZhang}.  

We show that the map $\Phi_t: X \to L^2(X, \meas)$ is a \emph{continuous} embedding of the compact $\RCD^*(K,N)$ space $(X, \dist, \meas)$ into $L^2(X,\meas)$, in other words, the map $\Phi_t$ is a homeomorphism onto its image.
Note that our proof actually shows that $\Phi_t$ is Lipschitz, but in general $(\Phi_t)^{-1}$ is not (see Remark \ref{counterex}). 
The next step is to define the pull-back metric $g_t$, formally given at $x\in X$ by 
$$
g_t(v,v)= \int_X g(\nabla_x p(x,y,t),v)^2\di\meas (y)
$$
for a tangent vector $v$ at $x$, where $g$ is the Riemannian metric of the space $(X,\dist,\meas)$, canonically
derived from Cheeger's energy (see Proposition~\ref{prop:gcanonical}).

However, since in calculus in metric measure spaces 
many objects (vector fields, gradients, Hessians, etc.) are only defined up to $\meas$-negligible sets, we shall
rather work with the integral formula 
$$
\int_X g_t(V,V)\di\meas:=\int_X \int_X g(\nabla_x p(x,y,t),V(x))^2\di\meas (y)\di\meas(x)
$$
for any square integrable tangent vector field $V$ and we 
prove convergence as $t\downarrow 0$, after a suitable rescaling, to $\int_X g(V,V)\di\meas$.
In the $\RCD^*(K,N)$ theory we know, thanks to the very recent work \cite{BrueSemola}
(which extends a part of \cite{ColdingNaber} from Ricci limit spaces to general spaces)
that $\RCD^*(K,N)$ spaces have a unique ``essential dimension'' $n$ (see Theorem~\ref{th: RCD decomposition} for the precise statement)
related to $\meas$ by the identity $\meas=\theta\mathcal{H}^n\res\mathcal R_n$, where $\mathcal R_n$ is the $n$-dimensional regular set of $(X,\dist,\meas)$ according to \cite{MondinoNaber}. Because of the weight $\theta$, 
it is natural to replace the scaling $c(n)t^{(n+2)/2}$ in \eqref{eq:goodlimit1} by the local and dimension-free scaling
function $t\meas(B_{\sqrt{t}}(x))$. We prove in Theorem~\ref{th:convergence1} that $\hat{g}_t:=t\meas(B_{\sqrt{t}}(\cdot))g_t$ converge as
$t\downarrow 0$, in a strong sense (which involves also the Hilbert-Schmidt norms of the metrics), to 
$$
\hat{g}:=c_n g
$$
where $c_n$ is a suitable dimensional constant (see \eqref{eq:defck}). Under an additional technical assumption, see \eqref{eq:goodlimit22} (satisfied for instance in Alexandrov spaces, weighted Riemannian manifolds and 
Ahlfors regular spaces), we can also consider the rescalings
$$
\tilde{g}_t:=t^{(n+2)/2}g_t
$$
and prove, in Theorem~\ref{th:convergence2}, their convergence to 
$$
\tilde{g}:=\frac{c_n}{\theta\omega_n}1_{\mathcal{R}_n^*}g.
$$
where $\mathcal{R}_n^*$ is the ``reduced'' regular set introduced in (\ref{eq:defRkstar}) (in particular $c_n$ is related to the
constant $c(n)$ in \eqref{eq:goodlimit1} by $c(n)=\omega_n/c_n$).

It would be desirable to have a counterpart of these convergence results
involving also the global (or, better, non-infinitesimal) point of view, i.e. distances instead of  metrics. Unfortunately, in the nonsmooth setting, the
process that allows to recover distances out of metrics is not straightforward, since the latter are only defined up
to $\meas$-negligible sets. We will tackle this problem in a forthcoming paper.

However, in this paper we prove two results that go in this direction. Specifically, let us endow
the class of $\RCD^*(K,N)$ metric measure spaces with the topology of measured Gromov-Hausdorff convergence.
We prove in Theorem~\ref{th:convergence3} that, one has: 
\begin{itemize}
\item[(1)] the map $(X,\dist,\meas)\mapsto g_t(X,\dist,\meas)$ is continuous, with respect to $t>0$ and the convergence of metrics on 
different metric measure spaces of Definition~\ref{def:conriemaspaces};
\item[(2)] the map $(X,\dist,\meas)\mapsto (\Phi_t(X),\dist_t)$, where $\dist_t$ is the \emph{restriction of the ambient $L^2(X, \meas)$-distance},
is continuous, endowing the target space with the Gromov-Hausdorff topology.
\end{itemize}
Moreover, in the noncollapsed setting, (1) can be improved to the case up to continuity at $t=0$, which allows us to
show the sharp quantitative convergence of $\hat{g}_t$ as $t \downarrow 0$ (see Theorems \ref{thm:quantitativ} and \ref{label}). These 
results are new, as far as we know, even for Riemannian manifolds and Alexandrov spaces.

\subsection*{Plan of the paper}

The paper is organized as follows: Section~\ref{sec:prem} collects all notation, preliminary results and terminology on $\RCD^*(K,N)$ spaces. In particular 
we focus on convergence results for Sobolev functions and heat flows, also in the local form that is sometimes needed in
the paper, when proving results by a blow-up argument. Section~\ref{sec:tan} provides a description of the tangent bundle, where we follow
closely Gigli's axiomatization in \cite{Gigli}. In particular, on the basis of this axiomatization and of \cite{AmbrosioGigliSavare14}, we are able to define the notion of Riemannian metric on an infinitesimally Hilbertian metric measure space $(X,\dist,\meas)$: in this family, the canonical Riemannian metric is the one induced by Cheeger's energy, since Cheeger's energy can be canonically built out of distance $\dist$ and measure $\meas$. In Section~\ref{sec:embedding} we introduce
the embedding map $\Phi_t$, first in the smooth case (on the basis of \cite{Berard85,BerardBessonGallot}) and then in the nonsmooth case. Section~\ref{sec:5} provides the proof of all convergence results except for quantitative ones, to which Section~\ref{sec:6} is dedicated. Finally, Appendix is devoted to asymptotic bounds on the eigenvalues
in the $\RCD^*(K,N)$ setting and to the expansion as a power series of the heat kernel.

\smallskip\noindent
\textbf{Acknowledgement.}
The first and fourth authors acknowledge the support of the MIUR PRIN 2015 project ``Calculus of Variations''.
The second author acknowledges supports of the JSPS Program for Advancing Strategic
International Networks to Accelerate the Circulation of Talented Researchers, of the Grantin-Aid
for Young Scientists (B) 16K17585, Grant-in-Aid for Scientific Research (B) of 18H01118 and of 20H01799.
The authors warmly thank the referee for the detailed reading of the paper and for the constructive comments. 

\section{Preliminary notions}\label{sec:prem}

Throughout this paper, by metric measure space we mean a triple $(X,\dist,\meas)$ where $(X,\dist)$ is a complete and separable
metric space and $\meas$ is a nonnegative measure on the Borel $\sigma$-algebra, finite on bounded sets. We use the notation
$L^p(X,\meas)$ for the space of $p$-integrable functions, where $1\le p \le \infty$, and $L^0(X,\meas)$ for $\meas$-measurable functions. Similarly we define $L^p(A, \meas)$ for all $A \subset X$ Borel, and $L^p_{\mathrm{loc}}(X, \meas)$ as the set of all $f \in L^0(X, \meas)$ with $f1_A \in L^p(X, \meas)$ for all bounded Borel subset $A$ of $X$, where $1_A$ denotes the characteristic function of a set $A$, with values in $\{0,1\}$. 

We adopt standard metric space notation, as $B_r(x)$ ($\overline{B}_r(x)$, resp.) for open (closed, resp.) balls, $C(X)$ ($C_c(X)$, resp.) for continuous (compactly supported Lipschitz, resp.) functions, $\Lip(X,\dist)$ ($\Lip_b$, $\Lip_c$, resp.) for Lipschitz (bounded
Lipschitz, compactly supported Lipschitz, resp.) functions, etc.

\subsection{Cheeger energy and Laplacian}\label{cheelap}

The Cheeger energy $\Ch:L^2(X,\meas)\to [0,\infty]$ associated to the metric measure structure $(X,\dist,\meas)$
is the convex and $L^2(X,\meas)$-lower semicontinuous functional defined by
\begin{equation}\label{eq:defchee}
\Ch(f):=\inf\left\{\liminf_{i\to\infty}\int_X{\rm lip}^2f_i\di\meas:\ f_i\in\Lip_b(X,\dist)\cap L^2(X,\meas),\,\,\,\|f_i-f\|_{L^2(X, \meas)}\to 0
\right\},
\end{equation}
where
$$
{\rm lip}f(x) :=
\begin{cases}
\limsup\limits_{y\to x}\frac{|f(y)-f(x)|}{\dist(y,x)} &  \text{if $x \in X$ is not isolated},\\
0 & \text{otherwise}.
\end{cases}
$$
denotes the local Lipschitz constant.
Accordingly, the Sobolev space $H^{1,2}(X,\dist,\meas)$ is defined as the finiteness domain of $\Ch$.

By looking at the optimal sequence in \eqref{eq:defchee} one can identify a canonical object $|\nabla f|$, called the minimal relaxed slope,
which is local on Borel sets (i.e. $|\nabla f_1|=|\nabla f_2|$ $\meas$-a.e. on $\{f_1=f_2\}$) and provides integral representation to $\Ch$, namely
$$\Ch (f)=\int_X|\nabla f|^2\di\meas\qquad\forall f\in H^{1,2}(X,\dist,\meas).$$

In this paper we shall only deal with \textit{infinitesimally Hilbertian} metric measure spaces, i.e. the metric measure spaces
such that $\Ch$ is a quadratic form. The following result, borrowed from \cite{AmbrosioGigliSavare14} (see also \cite{Gigli1} for the first part), 
plays an important role in our discussion:

\begin{theorem} \label{thm:duke} If $\Ch$ is quadratic, the function 
$$
\langle\nabla f_1,\nabla f_2\rangle:=\lim_{\epsilon\to 0}\frac{|\nabla (f_1+\epsilon f_2)|^2-|\nabla f_1|^2}{2\epsilon}
$$
provides a symmetric bilinear form on $H^{1,2}(X,\dist,\meas)\times H^{1,2}(X,\dist,\meas)$ with values in $L^1(X,\meas)$, and
$$
\mathcal{E} (f_1,f_2) := \int_X \langle \nabla f_1, \nabla f_2 \rangle \di \meas, \qquad \forall f_1, f_2 \in H^{1,2}(X,\dist,\meas) 
$$
defines a strongly local Dirichlet form.
\end{theorem}

Still assuming that $\Ch$ is a quadratic form, we can adopt the standard definition of Laplacian, namely
\begin{align*}
\mathcal{D}(\Delta) := \{ f \in H^{1,2}(X,\dist,\meas) \, : \, \, & \text{there exists} \, \, h \in L^2(X,\meas) \, \, \text{such that} \\
&  \mathcal{E}(f,g)= - \int_X h g \di \meas \, \, \, \text{for all} \, \,  g\in H^{1,2}(X,\dist,\meas) \, \}
\end{align*}
and $\Delta f := h$ for any $f \in \mathcal{D}(\Delta)$.

Besides the construction of $\Ch$, we need the following results from the seminal paper \cite{Cheeger}. They
hold in the class of so-called PI spaces, namely metric measure spaces $(X,\dist,\meas)$ satisfying 
the local doubling condition and a local $(1, 2)$-Poincar\'e inequality.

\begin{theorem}\label{thm:Cheeger}
Let $(X,\dist,\meas)$ be a PI space. Then:
\begin{itemize}
\item[(1)] For all $f\in H^{1,2}(X,\dist,\meas)$, for $\meas$-a.e. $x\in X$ one has
${\rm Dev}(f,B_r(x))=o(\meas(B_r(x)))$ as $r\downarrow 0$, where
\begin{equation}\label{eq:odev}
{\rm Dev}(f,B_r(x))=\int_{B_r(x)}|\nabla f|^2\di\meas-\inf\left\{\int_{B_r(x)}|\nabla h|^2\di\meas:\ f-h\in\Lip_c(B_r(x))\right\}.
\end{equation}
\item[(2)] There exist
$M\geq 0$ and a sequence of Borel subsets of positive measure $A_j$ such that $\meas \left(X \setminus \bigcup_jA_j\right)=0$ holds and for any $j$, there exist an integer $1\le k=k(j)\le M$ and Lipschitz functions $F_i\in \mathrm{Lip}(X,\dist)$, $1\leq i\leq k$, 
such that, for all $f\in \Lip(X, \dist)$, one has
\begin{equation}\label{eq:Cheexp}
{\rm lip}\bigl( f(\cdot)-\sum_{i=1}^k\chi_i(x_0)F_i(\cdot)\bigr)(x_0)=0
\qquad\text{for $\meas$-a.e. $x_0\in A_j$}
\end{equation}
for suitable $\chi_i\in L^2(A_j, \meas)$ with $\sum_i\chi_i^2\leq M|\nabla f|^2$ $\meas$-a.e. in $A_j$.
\end{itemize}
\end{theorem}
\begin{proof}
Let us prove \textit{(1)}. This statement holds if $f \in \mathrm{Lip}(X, \dist)$, see \cite[Th.~3.7]{Cheeger}. For $f \in H^{1, 2}(X, \dist, \meas)$, a telescoping argument (for instance \cite[Th.~4.14]{Cheeger}) allows us to prove that there exist Borel subsets $A_i \subset X (i=1, 2, \ldots)$ such that $\meas \left(X \setminus \bigcup_iA_i\right)=0$ holds and that $f|_{A_i}$ is Lipschitz. Take a Lipschitz function $F_i:X \to \mathbb{R}$ with $F_i|_{A_i}=f|_{A_i}$. Note that ${\rm Dev}(f,B_r(x))=o(\meas(B_r(x)))$ holds whenever ${\rm Dev}(F,B_r(x))=o(\meas(B_r(x)))$ and 
\begin{equation}\label{eq:added}
\int_{B_r(x)}|\nabla (F_i-f)|^2\di \meas=o(\meas(B_r(x))).
\end{equation}
Since (\ref{eq:added}) holds for $\meas$-a.e. $x \in A_i$, we have \textit{(1)}.

For the proof of \textit{(2)}, see \cite[Th.~4.38]{Cheeger}. 
\end{proof}

\subsection{$\RCD^*(K,N)$ spaces: definition and main properties}\label{se:rcdstarkn}

Throughout this paper the parameters $K\in\mathbb{R}$ (lower bound on Ricci curvature) and $1<N<\infty$ (upper bound on dimension) will
be kept fixed. The class of $\RCD^*(K,N)$ metric measure spaces, introduced in \cite{Gigli1} (after the case $N=\infty$ studied in \cite{AmbrosioGigliSavare14})
can now be characterized in many ways, (under suitable assumptions) via entropic convexity inequalities along Wasserstein geodesics
or evolution variational inequalities satisfied by the heat flow or nonlinear diffusion equations
(see \cite{ErbarKuwadaSturm}, \cite{AmbrosioMondinoSavare}). For the language adopted in this paper, where optimal transport does not play
a dominant role, the most appropriate characterization is the one based on the quadraticity of $\Ch$, the growth condition
$\meas(B_r(\bar x))\leq c_1\exp(c_2r^2)$ (for some, and thus any, $\bar x\in X$) on the measure of balls, the 
Sobolev-to-Lipschitz property (namely that any $f\in H^{1,2}(X,\dist,\meas)$ with $|\nabla f|\in L^\infty(X,\meas)$ has a Lipschitz
representative, with Lipschitz constant smaller than $\||\nabla f|\|_\infty$) and the validity of Bochner's
inequality
$$
\frac 12 \Delta |\nabla f|^2-\langle\nabla f,\nabla\Delta f\rangle\geq \frac{(\Delta f)^2}{N}+K|\nabla f|^2 
$$
in the class of functions
\begin{equation}\label{eq:deftest}
\mathrm{Test}F(X,\dist,\meas):=\left\{f\in\Lip_b(X,\dist )\cap H^{1,2}(X,\dist,\meas):\ \Delta f\in H^{1,2}(X,\dist,\meas)
\right\}
\end{equation}
which, a posteriori, turns out to be an algebra thanks to Bochner's inequality \cite{Savare}.

True for the larger class of weak $\CD(K,N)$ spaces \cite[Th.~30.11]{Villani}, the Bishop-Gromov theorem holds for any $\RCD^*(K,N)$ space $(X,\dist,\meas)$:
\begin{equation}\label{eq:BishopGromov}
\frac{\meas(B_R(x))}{\meas(B_r(x))} \le \frac{\mathrm{Vol}_{K,N}(R)}{\mathrm{Vol}_{K,N}(r)}
\le c_0 e^{c_1 R/r}
\end{equation}
for any $x \in\supp\meas$ and $0<r\leq R$, where $\mathrm{Vol}_{K,N}(r)$ denotes the volume of a ball of radius $r$ in the $N$-dimensional model space with Ricci curvature $K$ and $c_0, \,c_1 >0$ depend only on $K^-$ and $N$. A first trivial consequence is that $(X,\dist,\meas)$ is locally doubling, meaning that for any $R>0$, there exists $C_D>0$ depending only on $K^-$, $N$ and $R$, such that
\begin{equation}\label{eq:locdoubling}
\meas(B_{2r}(x)) \le C_D \meas(B_r(x))\qquad\forall x\in\supp\meas,\,\forall r\le R.
\end{equation}

Because of \eqref{eq:BishopGromov}, the following lemma, whose proof is omitted for brevity, applies to the whole class of $\RCD^{*}(K,N)$ spaces. It is a simple consequence of Cavalieri's formula together with \eqref{eq:BishopGromov} and its useful corollary:
\begin{equation}\label{doubspec}
\frac{\meas (B_1(y))}{\meas (B_1(x))} \le c_2 \exp \left( c_1\dist (x, y)\right)
\qquad\forall x,\, y\in\supp\meas
\end{equation}
with $c_2=c_0 e^{c_1}$, thanks to the inclusion $B_1(x)\subset B_{1+\dist(x,y)}(y)$.

\begin{lemma}\label{lem:volume}
Let $(Y,\dist_Y,\meas_Y)$ be a metric measure space and let $x\in\supp\meas_Y$ be satisfying 
\begin{equation}\label{eq:exp}
\frac{\meas(B_R(x))}{\meas(B_1(x))}\leq c_0 e^{c_1 R}\qquad \quad\forall R\geq 1
\end{equation}
for some constants $c_0,\,c_1>0$. Then:
\begin{itemize}
\item[(1)] for any $\delta>0$ there exists $L_0=L_0(\delta,c_0,c_1)>1$ such that
\begin{equation}\label{eq:11}
\int_{Y \setminus B_{L_0}(x)}\meas_Y (B_1(y))\exp \left(-\frac{2\dist_Y^2 (x, y)}{5}\right)\di\meas_Y(y)\le \delta(\meas_Y(B_1(x)))^2;
\end{equation}
\item[(2)] for any $\ell \in \mathbb{Z}$ there exists $C=C(\ell,c_0,c_1)\in [0,\infty)$ such that
\begin{equation}\label{lem:bound}
\int_Y \meas_Y (B_1(y))^\ell \exp \left(-\frac{2\dist_Y^2 (x, y)}{5}\right)\di\meas_Y(y) \le C(\meas_Y(B_1(x)))^{\ell+1}.
\end{equation}
\end{itemize}
\end{lemma}

Besides the doubling condition, Rajala proved \cite[Th.~1]{Rajala} that a local $(1,1)$-Poincar\'e inequality holds on the larger class of $\CD(K,\infty)$ spaces, and thus on $\RCD^{*}(K,N)$ spaces:
\begin{equation}\label{eq:locPoincaré}
\int_{B_r(x)} |f - \fint_{B_r(x)}f\di \meas| \di \meas \le 4 r e^{|K|r^2} \int_{B_{2r}(x)} |\nabla f| \di \meas,
\end{equation}
for any $f \in H^{1,2}(X,\dist,\meas)$ and any ball $B_r(x)$ with $x\in\supp\meas$. Here $\fint_{B_r(x)}f\di \meas$ denotes the mean value $\meas(B_r(x))^{-1} \int_{B_r(x)} f \di \meas$. It is also worth pointing out that also a local $(2, 2)$-Poincar\'e inequality holds if $N<\infty$, as a direct consequence of \cite[Th. 5.1]{HajlaszKoskela} with \eqref{eq:BishopGromov} and \eqref{eq:locPoincaré}.

Furthermore, it follows from the Sobolev-to-Lipschitz property 
(see \cite[Th.~6.2]{AmbrosioGigliSavare14}, \cite[Th.~12.8]{AmbrosioErbarSavare} for details)
that, on any $\RCD^{*}(K,N)$ space $(X,\dist,\meas)$, the intrinsic distance
$$
\dist_{\mathcal{E}}(x,y):= \sup \{ |f(x)-f(y)| \, : \, f \in H^{1,2}(X,\dist,\meas) \cap C_b(X), \, |\nabla f| \le 1 \}
$$
associated to $\mathcal{E}$ coincides with the original distance $\dist$. Consequently, Sturm's works on the general theory of Dirichlet forms on PI spaces provide existence of a locally H\"older continuous representative $p$ on $\supp\meas\times\supp\meas\times (0,\infty)$ for the heat kernel of $(X,\dist,\meas)$: see \cite[Prop.~2.3]{Sturm95} and \cite[Cor.~3.3]{Sturm96}. The sharp Gaussian estimates on this heat kernel have been proved later on in the $\RCD$ context by Jiang, Li and Zhang \cite[Th.~1.2]{JiangLiZhang}:
for any $\epsilon>0$, there exist $C_i:=C_i(\epsilon, K, N)>1$ for $i=1,\,2$, depending only on $K$, $N$ and $\epsilon$, such that 
\begin{equation}\label{eq:gaussian}
\frac{C_1^{-1}}{\meas (B_{\sqrt{t}}(x))}\exp \left(-\frac{\dist^2 (x, y)}{(4-\epsilon)t}-C_2t \right) \le p(x, y, t) \le \frac{C_1}{\meas (B_{\sqrt{t}}(x))}\exp \left( -\frac{\dist^2 (x, y)}{(4+\epsilon)t}+C_2t \right)
\end{equation}
for all $x,\, y \in\supp\meas$ and any $t>0$, where from now on we state our inequalities with the H\"older continuous representative. Combined with the Li-Yau inequality \cite{GarofaloMondino, Jiang15}, 
\eqref{eq:gaussian} implies a gradient estimate \cite[Cor.~1.2]{JiangLiZhang}:
\begin{equation}\label{eq:equi lip}
|\nabla_x p(x, y, t)|\le \frac{C_3}{\sqrt{t}\meas (B_{\sqrt{t}}(x))}\exp \left(-\frac{\dist^2(x, y)}{(4+\epsilon) t}+C_4t\right)
\qquad\text{for $\meas$-a.e. $x\in X$}
\end{equation}
for any $t>0$, $y\in\supp\meas$.  
Moreover by \cite[Th.~4]{Davies} with (\ref{eq:gaussian}) the inequality
\begin{equation}\label{eq:lapheatbd}
\left|\frac{\di}{\di t}p(x, y, t)\right|=|\Delta_x p(x, y, t)|\le \frac{C_5}{t\meas (B_{\sqrt{t}}(x))}\exp \left( -\frac{\dist (x, y)^2}{(4+\epsilon)t}+C_6t\right)
\end{equation}
holds for all $t>0$ and $\meas \times \meas$-a.e. $(x, y) \in X \times X$, 
where $C_i:=C_i(\epsilon, K, N)>1$ ($i=3,\,4,\,5,\,6$). (see also \cite[(3.11)]{JiangLiZhang}).
Note that in this article, we will always work with \eqref{eq:gaussian}, \eqref{eq:equi lip} and \eqref{eq:lapheatbd} in the case $\epsilon = 1$.

In the sequel we shall denote by $p_{k}$ the Euclidean heat kernel in $\setR^k$, given by
\begin{equation}\label{eq:pEuclidean}
p_{k}(x,y,t):=\frac{1}{\sqrt{4\pi t}^{k}}\exp\left(-\frac{|x-y|^2}{4t}\right)
\end{equation}
and recall the classical identity 
\begin{equation}\label{eq:Gaussian_variance}
\frac{1}{\sqrt{2\pi t}}\int_{\setR}x^2\exp\left(-\frac{x^2}{2t}\right)\di x=t.
\end{equation}
Furthermore, we shall often use the scaling formula
\begin{equation}\label{eq:1001}
\tilde{p}(x,y,s)=b^{-1} p(x,y,a^{-2} s)  \qquad \forall x,\,y \in\supp\meas, \, \, \, \forall s > 0 
\end{equation}
relating for any $a,\,b >0$, the heat kernel $\tilde{p}$ of the rescaled space $(X,a\dist,b\meas)$ to the heat
 kernel $p$ of $(X,\dist,\meas)$.

Let us spend some words concerning spectral theory on compact $\RCD^{*}(K,N)$ spaces. It follows from a standard argument on Dirichlet forms \cite{FukushimaOshimaTakeda} that the resolvent operators $R_\alpha:=(\alpha \Id - \Delta)^{-1} : L^2(X,\meas) \to H^{1,2}(X,\dist,\meas)$, $\alpha >0$, are well-defined injective bounded linear operators and that $R_\alpha (L^2(X,\meas))$ is a dense subset of $L^2(X,\meas)$, independent of $\alpha$, which coincides with $D(\Delta)$. By the Rellich-Kondrachov theorem \cite[Th.~8.1]{HajlaszKoskela}, all the $R_{\alpha}$ are compact operators sharing the same discrete positive spectrum $\mu_1 \ge \mu_2 \ge \cdots \to 0$, implying that (minus) the Laplacian operator $-\Delta$ admits a discrete positive spectrum $0=\lambda_0 < \lambda_1 \le \lambda_2 \le \cdots \to + \infty$. This provides the following expansions for the heat kernel $p$:
\begin{equation}\label{eq:expansion1}
p(x,y,t) = \sum_{i \ge 0} e^{- \lambda_i t} \phi_i(x) \phi_i (y) \qquad \text{in $C(\supp\meas\times\supp\meas)$}
\end{equation}
for any $t>0$ and
\begin{equation}\label{eq:expansion2}
p(\cdot,y,t) = \sum_{i \ge 0} e^{- \lambda_i t} \phi_i(y) \phi_i \qquad \text{in $H^{1,2}(X,\dist,\meas)$}
\end{equation}
for any $y\in\supp\meas$ and $t>0$. We refer to the Appendix for a detailed proof of these expansions.

Let us conclude this overview by mentioning the main structural properties of $\RCD^{*}(K,N)$ spaces. Before that, we need to recall the definitions of rectifiable sets and of tangent spaces to a metric measure space $(X,\dist,\meas)$ at a point $x$.

\begin{definition}[Rectifiable sets]
Let $(Y,\dist_Y)$ be a metric space and let $k\geq 1$ be an integer.
\begin{itemize}
\item[(1)] We say that $S\subset Y$ is
countably $k$-rectifiable if there exist at most countably many bounded Borel sets $B_i\subset\setR^k$ 
and Lipschitz maps $f_i:B_i\to Y$ such that $S\subset\cup_i f_i(B_i)$.
\item[(2)] For a nonnegative Borel measure $\mu$ in $Y$ (not necessarily $\sigma$-finite), we say
that $S$ is $(\mu,k)$-rectifiable if there exists a countably $k$-rectifiable set $S'\subset S$
such that $\mu^*(S\setminus S')=0$, i.e. $S\setminus S'$ is contained in a
$\mu$-negligible Borel set.
\end{itemize}
\end{definition}

\begin{definition}[Tangent metric measure spaces]
For $x \in\supp\meas$, we denote by $\mathrm{Tan}(X, \dist,\meas,x)$ the set of tangent spaces to $(X,\dist,\meas)$ at $x$: the collection of all
pointed metric measure spaces $(Y, \dist_Y,\meas_Y,y)$ such that, as $i\to\infty$, one has
\begin{equation}\label{eq:ulla_ulla}
\left(X, r_i^{-1}\dist, \meas (B_{r_i}(x))^{-1}\meas, x\right) \stackrel{mGH}{\to} (Y, \dist_Y,\meas_Y,y) 
\end{equation}
for some $r_i\to 0^+$, where mGH denotes the measured pointed Gromov-Hausdorff convergence. 
\end{definition}

If $\meas$ is doubling, it is not hard to prove, by rescaling the $r_i$ in \eqref{eq:ulla_ulla} by a constant factor,
 that $\mathrm{Tan}(X, \dist,\meas,x)$ is not empty for all
$x\in\supp\meas$ and that it is a cone in the following very weak sense: for all $t>0$ and all
$(Y, \dist_Y, \meas_Y,y)\in \mathrm{Tan}(X, \dist,\meas,x)$,
$$
(Y, t^{-1}\dist_Y, \meas_Y (B_t(y))^{-1}\meas_Y,y)\in \mathrm{Tan}(X, \dist,\meas,x).
$$

\begin{definition}[Regular set $\mathcal{R}_k$]
For any $k \geq 1$, we denote by $\mathcal{R}_k$ the $k$-dimensional regular set  of $(X, \dist, \meas)$, 
namely the set of points $x \in\supp\meas$ such that
$$
\mathrm{Tan}(X, \dist,\meas,x) =\left\{ \left(\mathbb{R}^k, \dist_{\mathbb{R}^k},(\omega_k)^{-1}\mathcal{H}^k,0_k\right) \right\},
$$
where $\omega_k$ is the $k$-dimensional volume of the unit ball in $\mathbb{R}^k$ with respect to the $k$-dimensional Hausdorff measure $\mathcal{H}^k$. 
\end{definition}

We are now in a position to introduce the latest structural result for $\RCD^*(K,N)$ spaces.

\begin{theorem}[Essential dimension of $\RCD^*(K,N)$ spaces]\label{th: RCD decomposition}
Let $(X,\dist,\meas)$ be a $\RCD^*(K,N)$ space. Then, there exists a unique integer $n\in [1,N]$ such that
 \begin{equation}\label{eq:regular set is full}
\meas(X\setminus \mathcal{R}_n\bigr)=0.
\end{equation}
In addition, the set $\mathcal{R}_n$ is $(\meas,n)$-rectifiable and $\meas$ is representable as
$\theta\mathcal{H}^n\res\mathcal{R}_n$.
\end{theorem}

We denote by $\dim_{\dist,\meas}(X)$ the ``essential dimension'' of $(X,\dist,\meas)$, namely the integer $n$ such that $\meas(\mathcal{R}_n) >0$. Note that the rectifiability of all sets $\mathcal{R}_k$ was inspired by \cite{CheegerColding1} and proved in \cite{MondinoNaber}, together with the concentration property
$\meas(X\setminus\cup_k\mathcal{R}_k)=0$, with the crucial uses of \cite{GigliMondinoRajala} and of \cite{Gigli13} ; the absolute continuity of $\meas$ on regular sets with respect to the corresponding Hausdorff measure was proved afterwards and is a consequence of \cite{KellMondino}, \cite{DePhillippisMarcheseRindler} and \cite{GigliPasqualetto}. Finally, in the very recent work \cite{BrueSemola} it is proved that only one set $\mathcal{R}_k$ has positive
$\meas$-measure, leading to \eqref{eq:regular set is full} and to the representation $\meas=\theta\mathcal{H}^n\res\mathcal{R}_n$.

By slightly refining the definition of $n$-regular set, passing to a reduced set $\mathcal{R}_n^*$, general results of measure differentiation provide also  
the converse absolutely continuity property $\haus^n \ll \meas$ on $\mathcal{R}_n^*$. We summarize here the results obtained in this direction in
\cite{AmbrosioHondaTewodrose}:

\begin{theorem}[Weak Ahlfors regularity]\label{thm:RN}
Let $(X, \dist, \meas)$ be a $\RCD^* (K, N)$-space, $n=\dim_{\dist,\meas}(X)$, $\meas=\theta\mathcal{H}^n\res\mathcal{R}_n$ and set
\begin{equation}\label{eq:defRkstar}
{\mathcal R}_n^*:=\left\{x\in\mathcal{R}_n:\
\exists\lim_{r\to 0^+}\frac{\meas(B_r(x))}{\omega_nr^n}\in (0,\infty)\right\}.
\end{equation}
Then $\meas(\mathcal{R}_n\setminus\mathcal{R}_n^*)=0$, $\meas\res\mathcal{R}_n^*$ and 
$\mathcal{H}^n\res\mathcal{R}_n^*$ are mutually absolutely continuous and
\begin{equation}\label{eq:gooddensity}
\lim_{r\to 0^+}\frac{\meas(B_r(x))}{\omega_nr^n}=\theta(x)
\qquad\text{for $\meas$-a.e. $x\in\mathcal{R}_n^*$,}
\end{equation}
\begin{equation}\label{eq:goodlimitRN}
\lim_{r\to 0^+} \frac{\omega_nr^n}{\meas(B_{r}(x))}=1_{\mathcal{R}^*_{n}}(x)
\frac{1}{\theta(x)}
\qquad\text{for $\meas$-a.e. $x\in X$.}
\end{equation}
Moreover  $\mathcal{H}^n(\mathcal{R}_n\setminus\mathcal{R}_n^*)=0$ 
if $n=N$.
\end{theorem}

\subsection{Sobolev spaces and Laplacians on open sets}

Following a standard approach, let us localize some of the concepts introduced in Section~\ref{cheelap}. First of all, let us introduce the
Sobolev space $H^{1, 2}(U, \dist, \meas)$ for an open subset $U$ of a $\RCD^*(K, N)$-space $(X, \dist, \meas)$. See also \cite{Cheeger,Shanmugalingam} for the definition of Sobolev space $H^{1, p}(U, \dist, \meas)$ for any $p \in [1, \infty)$. Our working definition is the following.

\begin{definition}\label{def:reddu}
Let $U\subset X$ be open. 
\begin{enumerate}
\item{($H^{1,2}_0$-Sobolev space)} We denote by $H^{1,2}_0(U, \dist, \meas )$ the 
$H^{1,2}$-closure of $\Lip_c(U, \dist)$.
\item{(Sobolev space on an open set $U$)} We say that $f\in L^2_{\mathrm{loc}}(U,\meas)$ belongs to $H^{1,2}_{\rm loc}(U,\dist,\meas)$ if
$\phi f \in H^{1, 2}(X, \dist, \meas)$ for any $\phi \in\Lip_c(U, \dist)$. If, in addition, $f,  |\nabla f|\in L^2(U,\meas)$,
we say that $f\in H^{1,2}(U,\dist,\meas)$. 
\end{enumerate}
\end{definition}

Notice that $f\in H^{1,2}_{\rm loc}(U,\dist,\meas)$  if and only if  for any $V\Subset U$ there exists 
$\tilde f\in H^{1,2}(X,\dist,\meas)$ with $\tilde f\equiv f$ on $V$.
The global condition $f, |\nabla f|\in L^2(U,\meas)$ in the definition of $H^{1,2}(U,\dist,\meas)$
is meaningful, since the locality properties of the minimal relaxed slope ensure that $|\nabla f|$ makes sense 
$\meas$-a.e. in $X$ for all functions $f\in H^{1,2}_{\rm loc}(U,\dist,\meas)$. Indeed, choosing
$\phi_n\in\Lip_c(U, \dist)$ with $\{\phi_n=1\}\uparrow U$ and defining
$$
|\nabla f|:=|\nabla(f\phi_n)|\qquad\text{$\meas$-a.e. in $\{\phi_n=1\}$}
$$
we obtain an extension of the minimal relaxed gradient to $H^{1,2}(U,\dist,\meas)$
(for which we keep the same notation, being also $\meas$-a.e. independent of the choice of $\phi_n$) 
which retains all bilinearity and locality properties. 


We introduce the Dirichlet Laplacian acting only on $H^{1, 2}_0$-functions as follows:

\begin{definition}[Dirichlet Laplacian on an open set $U$]\label{def:dlapballs}
Let $D_0(\Delta, U)$ denote the set of all $f \in H^{1, 2}_0(U, \dist, \meas)$ such that 
there exists $h:=\Delta_{U} f\in L^2(U,\meas)$ satisfying
$$\int_U hg\di\meas=-\int_U \langle\nabla f,\nabla g\rangle\di\meas\qquad\forall g\in H^{1,2}_0(U,\dist,\meas).
$$
We also set $\Delta_{x, R}:=\Delta_{B_R(x)}$ when $U=B_R(x)$ for some $x \in X$ and $R >0$.
\end{definition}

Strictly speaking, the Dirichlet Laplacian $\Delta_{U}$ should not be confused with the operator $\Delta$, even if the two operators
agree on functions compactly supported on $U$; for this reason we adopted a distinguished symbol. 
Notice that $\lambda_1^D(B_R(x))>0$ whenever $\meas(X\setminus B_{R}(x))>0$, as
a direct consequence of the local Poincar\'e inequality.

\begin{definition}[Laplacian on an open set $U$]
For $f \in H^{1, 2}(U, \dist, \meas)$, we write $f\in D(\Delta, U)$ if there exists 
$h:=\Delta_{U}f\in L^2(U,\meas)$  satisfying
$$
\int_U hg\di\meas=-\int_U \langle\nabla f, \nabla g\rangle \di\meas \qquad
\forall g\in H^{1,2}_0(U,\dist,\meas).
$$
\end{definition}

Since for $f \in H^{1, 2}_0(U, \dist, \meas)$ one has $f\in D(\Delta, U)$ iff
$f\in D_0(\Delta,U)$ and the Laplacians are the same, we retain the same notation $\Delta_{U}$
of Definition~\ref{def:dlapballs}.
It is easy to check that for any $f \in D(\Delta, U)$ and any 
$\phi \in D(\Delta ) \cap \Lip_c(U, \dist)$ with $\Delta \phi \in L^{\infty}(X, \meas)$
one has (understanding $\phi\Delta_{U}f$ to be null out of $U$)
$\phi f \in D(\Delta )$ with
\begin{equation}\label{eq:local to global}
\Delta (\phi f)=f\Delta \phi +2\langle\nabla\phi,\nabla f\rangle +\phi \Delta_{U} f
\qquad\text{$\meas$-a.e. in $X$.}
\end{equation}

Such notions allow to define harmonic functions on an open set $U$ as follows.

\begin{definition}
Let $U\subset X$ be open. We say that $f\in H^{1,2}_{\rm loc}(U,\dist,\meas)$ is harmonic in $U$ if 
$f \in \mathcal{D}(\Delta, V)$ with $\Delta_V f=0$ for any open set $V \Subset U$, namely
$$
\int_U \langle\nabla f, \nabla g\rangle \di\meas=0 \qquad\forall g\in\Lip_c(U,\dist).
$$
Let us denote by $\mathrm{Harm}(U, \dist, \meas)$ the set of harmonic functions on $U$.
\end{definition}

In this article, we will consider mainly globally defined harmonic functions. It is worth pointing out that, in general, these functions
do not belong to $H^{1,2}(X,\dist,\meas)$ but, by definition, they belong to $H^{1,2}_{\rm loc}(X,\dist,\meas)$.

\subsection{Convergence of global/local Sobolev functions}\label{secopt}

Let us fix a pointed measured Gromov-Hausdorff (mGH) convergent sequence
\begin{equation}\label{eq:05}
(X_i,\dist_i,\meas_i,x_i) \stackrel{mGH}{\to} (X,\dist,\meas,x)
\end{equation}
of $\RCD^{*}(K,N)$ spaces.  From now on we denote by $\Ch^i=\Ch_{\meas_i}$, $\langle\cdot,\cdot\rangle_i$, $\Delta_i$, etc. the various objects associated to the $i$-th
metric measure structure.

We adopt here the so-called
extrinsic approach from \cite[Def.~3.9]{GigliMondinoSavare13} to deal with the convergence \eqref{eq:05}, that is to say that we assume $X_i=\supp\meas_i$, $X=\supp\meas$ and all the sets $X_i$, as well as $X$, are isometrically embedded into a common complete and separable metric space $(Y,\underline{\dist})$, and if we identify $x_i$, $x$ with their image through  the isometric embeddings into $Y$ and $\meas_i$, $\meas$ with their pushforward through these embeddings, then $x_i\to x$ in $Y$ and $\meas_i \weakto \meas$ in duality with $C_{bs}(Y)$. When all the spaces involved are uniformly locally doubling (what we do have in \eqref{eq:05} thanks to \eqref{eq:locdoubling}), this approach, also called pointed measured Gromov convergence (pmG for short), is equivalent to the classical mGH convergence. See \cite[Th.~3.15, Prop. 3.30 and 3.33]{GigliMondinoSavare13} for a proof of this equivalence. Moreover, there is no loss of generality in assuming $Y$ proper (i.e.~bounded sets are compact), see \cite[Rk.~3.27]{GigliMondinoSavare13}.

The extrinsic approach is convenient to formulate various notions of convergence and to avoid the use of
$\epsilon$-isometries. However, it should be handled with care: for instance, if $f\in\Lip_b(Y,\dist)$ 
is viewed as a sequence of bounded Lipschitz functions in the spaces $(X_i,\dist_i,\meas_i)$, then the sequence need not be strongly
convergent in $H^{1,2}$ (see \cite{AmbrosioStraTrevisan} for a simple example).
Unlike $X_i$, the ambient space $(Y,\dist)$ will not appear often in our notation, since the measures $\meas_i$ are concentrated 
on $X_i$; however $Y$ plays an important role to define weak convergence of functions $f_i\in L^p(X_i,\meas_i)$, since the test functions 
are continuous and compactly supported in the ambient space. Notice also that any continuous (compactly
supported, resp.) function $\varphi:B_R^{X_i}(x)\to\mathbb R$ can be thought as the restriction of a continuous (compactly supported, resp.)
function $\tilde\varphi:B_R^Y(x)\to\mathbb R$.

In this setting, let us recall the definition of $L^2$-strong/weak convergence of functions with respect to the mGH-convergence.
The following formulation is due to \cite{GigliMondinoSavare13} and \cite{AmbrosioStraTrevisan}, which fits the pmG-convergence well.
Other equivalent
formulations of $L^2$-convergence, in connection with mGH-convergence, can be found in \cite{KuwaeShioya03, Honda2}.
See also their references and \cite{AmbrosioHonda} for the definition of $L^p$-convergence for all $p \in [1, \infty)$.

\begin{definition}[$L^2$-convergence of functions defined on varying spaces]\label{def:l2}
We say that $f_i \in L^2(X_i, \meas_i)$ \textit{$L^2$-weakly converge to $f \in L^2(X, \meas )$} 
if $\sup_i\|f_i\|_{L^2}<\infty$ and $\int_Yhf_i\di\meas_i \to \int_Yhf\di \meas$ for all $h \in C_c(Y)$.
Moreover, we say that $f_i\in L^2(X_i,\meas_i)$ \textit{$L^2$-strongly converge to $f\in L^2(X,\meas)$} if 
$f_i$ $L^2$-weakly converge to $f$ with $\limsup_i\|f_i\|_{L^2}\le \|f\|_{L^2}$. 
\end{definition}

Note that it was proven in \cite{GigliMondinoSavare13} (see also \cite{AmbrosioStraTrevisan}, \cite{AmbrosioHonda}) 
that any $L^2$-bounded sequence has an $L^2$-weak convergent subsequence in the above sense. 

Following \cite{GigliMondinoSavare13}, let us now define weak and strong convergence of 
Sobolev functions defined on varying metric measure spaces.

\begin{definition}[$H^{1,2}$-convergence of functions defined on varying spaces]
We say that $f_i\in H^{1,2}(X_i,\dist_i,\meas_i)$ are weakly convergent in $H^{1,2}$ to 
$f\in H^{1,2}(X_i,\dist_i,\meas_i)$ if $f_i$ are $L^2$-weakly convergent to $f$ 
and $\sup_i\Ch^i(f_i)$ is finite. Strong convergence
in $H^{1,2}$ is defined by requiring $L^2$-strong convergence of the functions, and $\Ch(f)=\lim_i\Ch^i(f_i)$. 
\end{definition}

We can now introduce the  local counterpart of these concepts.

\begin{definition}[Local $L^2$-convergence on varying spaces]\label{def:l2loc}
We say that $f_i \in L^2(B_R(x_i), \meas_i)$ are $L^2$-weakly (or strongly, resp.) convergent to $f \in L^2(B_R(x), \meas)$ on $B_R(x)$ if $f_i1_{B_R(x_i)} \in L^2(X_i, \meas_i)$ $L^2$-weakly (or strongly, resp.) convergent to $f1_{B_R(x)}$ according to Definition~\ref{def:l2}. 

We say that $g_i \in L^2_{\mathrm{loc}}(X_i, \meas_i)$ are $L^2_{\mathrm{loc}}$-weakly (or strongly, resp.) convergent to $g \in L^2_{\mathrm{loc}}(X, \meas)$ if $g_i$ $L^2$-weakly (or strongly, resp.) convergent to $g$ on $B_R(x)$ for all $R>0$.
\end{definition}

Similarly, let us define local $H^{1, 2}$-convergence as follows. 

\begin{definition}[Local $H^{1, 2}$-convergence on varying spaces]
We say that the functions $f_i \in H^{1, 2}(B_R(x_i), \dist_i, \meas_i)$ are weakly convergent in $H^{1, 2}$ to 
$f \in H^{1, 2}(B_R(x), \dist, \meas)$ on $B_R(x)$ if $f_i$ are $L^2$-weakly 
convergent to $f$ on $B_R(x)$ with $\sup_i\|f_i\|_{H^{1, 2}(B_R(x_i))}<\infty$.
Strong convergence in $H^{1, 2}$ on $B_R(x)$ is defined by requiring strong $L^2$ convergence and
$\lim_i\||\nabla f_i|_i\|_{L^2(B_R(x_i))}=\||\nabla f|\|_{L^2(B_R(x))}$.

We say that $g_i \in H^{1, 2}_{\mathrm{loc}}(X_i, \dist_i, \meas_i)$ $H^{1, 2}_{\mathrm{loc}}$-weakly (or strongly, resp.) convergent to $g \in H^{1, 2}_{\mathrm{loc}}(X, \dist, \meas)$ if $g_i|_{B_R(x_i)}$ $H^{1, 2}$-weakly (or strongly, resp.) convergent to $g|_{B_R(x)}$ for all $R>0$. 
\end{definition}

The following fundamental properties of local convergence of functions have been established in \cite{AmbrosioHonda2}. 
They imply, among other things, that in the definition of local $H^{1,2}$-weak convergence one may equivalently
require $L^2$-weak or $L^2$-strong convergence of the functions. 

\begin{theorem}[Compactness of local Sobolev functions]\label{thm:compact loc sob}
Let $R>0$ and let $f_i \in H^{1, 2}(B_{R}(x_i), \dist_i, \meas_i)$ with $\sup_i\|f_i\|_{H^{1, 2}}<\infty$.
Then there exist $f \in H^{1, 2}(B_{R}(x), \dist, \meas)$ and a subsequence $f_{i(j)}$ such that $f_{i(j)}$ $L^2$-strongly converge to $f$ on $B_R(x)$
and  
$$
\liminf_{j \to \infty}\int_{B_R(x_{i(j)})}|\nabla f_{i(j)}|^2_{i(j)}\di\meas_{i(j)} \ge \int_{B_R(x)}|\nabla f|^2\di\meas.
$$  
\end{theorem}

\begin{theorem}[Stability of Laplacian on balls]\label{thm:stability lap}
Let $f_i \in D(\Delta, B_R(x_i))$ with 
$$
\sup_i(\|f_i\|_{H^{1, 2}(B_R(x_i), \dist_i, \meas_i)}+\|\Delta_{x_i,R} f_i\|_{L^2(B_R(x_i), \meas_i)})<\infty,
$$
and with $f_i$ $L^2$-strongly convergent to $f$ on $B_R(x)$ (so that, by Theorem~\ref{thm:compact loc sob}, 
$f \in H^{1, 2}(B_R(x), \dist, \meas)$). Then: 
\begin{enumerate}
\item[(1)] $f \in D(\Delta, B_R(x))$;
\item[(2)] $\Delta_{x_i,R} f_i$ $L^2$-weakly converge to $\Delta_{x,R} f$ on $B_R(x)$; 
\item[(3)] $|\nabla f_i|_i$ $L^2$-strongly converge to $|\nabla f|$ on $B_r(x)$ 
for any $r<R$. 
\end{enumerate}
\end{theorem}

The pointwise convergence of heat kernels for a convergent sequence of $\RCD^*(K,N)$ spaces has been proved 
in \cite[Th.~3.3]{AmbrosioHondaTewodrose}; building on this, and using the ``concentration'' estimate \eqref{eq:gaussian_ep} below, 
one can actually prove the global $H^{1, 2}$-strong convergence.

\begin{theorem}[$H^{1, 2}$-strong convergence of heat kernels]\label{thm:local conv heat kernel}
For all convergent sequences $t_i \to t$ in $(0,\infty)$ and $y_i \in X_i \to y \in X$, 
$p_i(\cdot, y_i, t_i) \in H^{1,2}(X_i,\dist_{i},\meas_i)$ $H^{1,2}$-strongly converge to $p(\cdot, y, t) \in H^{1,2}(X,\dist,\meas)$.
\end{theorem}
\begin{proof} By a rescaling argument we can assume $t_i=t=1$.
Applying Theorem~\ref{thm:stability lap} for $p_i$ with (\ref{eq:lapheatbd}) yields that $p_i( \cdot, y_i, 1)$ $H^{1, 2}_{\mathrm{loc}}$-strongly converge to $p( \cdot, y, 1)$.
We claim that for any $\delta>0$ there exists $L:=L(K^-,N,\delta)>1$ 
such that for any $\RCD^*(K, N)$ space $(Z,\dist,\nu)$ and any $y\in\supp\nu$ one has
($q$ denoting its heat kernel)
\begin{equation}\label{eq:gaussian_ep}
\int_{Z\setminus B_L(y)}q^2(z, y, 1) + |\nabla_z q(z, y, 1)|^2 \di\nu (z) \le \frac \delta {\nu^2(B_1(y))}. 
\end{equation}
Indeed, let us prove the estimate for $q$, the proof of the estimate for $|\nabla_z q|$ (based on \eqref{eq:equi lip}) being similar.
Combining \eqref{doubspec} with the Gaussian estimate \eqref{eq:gaussian} with 
$\epsilon=1$, one obtains
$$
\int_{Z\setminus B_L(y)}q^2(z, y, 1) \di\nu(z)\leq\frac{c_2^2C_1^2e^{2C_2}}{\nu^2(B_1(y))}
\int_{Z\setminus B_L(y)}\exp\biggl(-\frac{2}{5}\dist^2(z,y)+2c_1\dist(z,y)\biggr)\di\nu(z)
$$
and then one can use the exponential growth condition on $\nu(B_R(y))$, coming from \eqref{eq:BishopGromov}, 
to obtain that the right hand side is smaller than $\delta/\nu^2(B_1(y))$ for $L=L(K^-,N,\delta)$ sufficiently large.

Combining \eqref{eq:gaussian_ep} with the $H^{1, 2}_{\mathrm{loc}}$-strong convergence of $p_i$ shows that
\begin{equation}\label{eq:global conv}
\lim_{i \to \infty}\|p_i(\cdot, y_i, 1)\|_{H^{1, 2}(X_i,\dist_i,\meas_i)}=\|p (\cdot, y, 1)\|_{H^{1, 2}(X,\dist,\meas)},
\end{equation}
which completes the proof.
\end{proof}

We shall also use the following local compactness theorem under $BV$ bounds, applied to sequences of Sobolev
functions. 

\begin{theorem}\label{thm:last_compactness}
Assume that a sequence $(f_i)\subset H^{1,2}(B_2(x_i), \dist_i, \meas_i)$ satisfy
$$
\sup_i\|f_i\|_{L^\infty(B_2(x_i), \meas_i)}+\int_{B_2(x_i)} |\nabla f_i|_i\di\meas_i<\infty.
$$
Then $(f_i)$ has a subsequence $L^p$-strong convergent on $B_1(x)$ for all $p\in [1,\infty)$.
\end{theorem}
\begin{proof} The proof of the compactness w.r.t. $L^1$-strong convergence can be obtained arguing as in 
\cite[Prop.~7.5]{AmbrosioHonda} (where the result
is stated in global form, for normalized metric measure spaces, even in the $BV$ setting), 
using good cut-off functions, see also \cite[Prop.~3.39]{Honda2} where a
uniform $L^p$ bound on gradients, for some $p>1$ is assumed. Then, because of the uniform
$L^\infty$ bound, the convergence is $L^p$-strong for any $p\in [1,\infty)$, see \cite[Prop.~3.3(e)]{AmbrosioHonda}.
\end{proof}

Let us conclude this section by introducing the notion of harmonic replacement which will play key roles in Sections \ref{sec:embedding} and \ref{sec:5}.
As we already remarked, the assumption that the first Dirichlet eigenvalue $\lambda_1^D(B_R(x))$ for the ball $B_R(x)$ is
strictly positive is valid for sufficently small balls, indeed it holds as soon as $\meas(X\setminus B_R(x))>0$. See \cite[Lem.~4.2]{AmbrosioHonda2} for the proof of the following proposition.

\begin{proposition}\label{prop: harmonic replacement}
Assume $\lambda_{1}^{D}(B_R (x))>0$. Then for any 
$f \in H^{1,2}(B_R (x),\dist,\meas)$, there exists a unique $\hat{f} \in D(\Delta, B_R(x))$, called \textit{harmonic replacement} of $f$, such that 
\begin{equation}\label{eq: harmonic replacement}
\begin{cases}
\Delta_{x,R}\hat{f}=0\\\\
f-\hat{f} \in H^{1, 2}_0(B_R(x), \dist, \meas).
\end{cases} 
\end{equation}
Moreover,
\begin{equation}
\label{eq:sol bound}
\||\nabla\hat{f}|\|_{L^2(B_R(x), \meas)}\leq 2\||\nabla f|\|_{L^2(B_R(x), \meas)},
\end{equation}
\begin{equation}\label{eq:l2 est}
\|\hat f\|_{L^2(B_R(x), \meas)}\leq \|f\|_{L^2(B_R(x), \meas)}+\frac{1}{\lambda_1^D(B_R(x))}\||\nabla f|\|_{L^2(B_R(x), \meas)}.
\end{equation}
Finally, $\hat{f}-f$ is the unique minimizer of the functional
$$ \psi\in H^{1,2}_0(B_R(x),\dist,\meas)\mapsto\int_X|\nabla (f+\psi)|^2\di\meas.$$
\end{proposition}

Next proposition, which is crucial for Section~\ref{sec:5}, gives some conditions under which harmonic replacements are continuous with 
respect to measured Gromov-Hausdorff convergence. It is a consequence of \cite[Th.~3.4]{AmbrosioHonda2}.

\begin{proposition}[Continuity of harmonic replacements]\label{prop:harmconti}
Assume $\lambda_1^D(B_R(x))>0$ with
\begin{equation}\label{eq:countably}
H^{1, 2}_0(B_R(x), \dist, \meas)=\bigcap_{\epsilon >0}H^{1, 2}_0(B_{R+\epsilon}(x), \dist, \meas).
\end{equation}
Let $f_i \in H^{1, 2}(B_R(x_i), \dist_i, \meas_i)$ be a weakly $H^{1, 2}$-convergent sequence to $f \in H^{1, 2}(B_R(x), \dist, \meas)$ on $B_R(x)$. 
Then the harmonic replacements $\hat{f_i}$ of $f_i$ on $B_R(x_i)$ exist for $i$ large enough and $L^2$-strongly converge to the
harmonic replacement $\hat{f}$ of $f$ on $B_R(x)$.
\end{proposition}

Notice that a simple separability argument shows that, given $x\in X$, the condition \eqref{eq:countably} is satisfied for
all $R>0$ with at most countably many exceptions (see \cite[Lem.~2.12]{AmbrosioHonda2}).

\section{Tangent bundle}\label{sec:tan}

In this section we introduce the tangent bundle $T(X,\dist,\meas)$ on an infinitesimally Hilbertian space $(X, \dist, \meas)$. More precisely, in the smooth
setting, the construction we give provides $L^2(T(X,\dist,\meas))$, namely all $L^2$ sections of the tangent bundle; here, according to
\cite{Gigli,Weaver} we describe the tangent bundle implicity, through the collection of its sections.
We follow closely the construction from \cite{Gigli}, with minor simplifications deriving from the Hilbertian assumption, 
since the original construction therein starts from $L^2$ sections of the cotangent bundle $L^2(T^*(X,\dist,\meas))$ 
and then recovers $L^2(T(X,\dist,\meas))$ by duality. 

Recall that, according to \cite{Gigli}, a (real) $L^\infty(X,\meas)$-module is a Banach real vector space $(M,\|\cdot\|_B)$ with an additional structure of bilinear multiplication
by $L^\infty(X,\meas)$ functions $m\in M\mapsto\chi m\in M$, $\chi \in L^\infty(X,\meas)$, such that $\|\chi m\|_B \le \|\chi\|_\infty \|m\|_B$ and the associative property $\chi(\chi' m)=(\chi\chi') m$ hold, satisfying
also the locality and gluing axioms (see (1.2.1) and (1.2.2) in \cite{Gigli});
in addition, multiplication by $\lambda\in\setR$ corresponds to multiplication by the $L^\infty(X,\meas)$ function equal $\meas$-a.e.
to $\lambda$. We say that a $L^\infty(X,\meas)$-module $M$ is a $L^2(X,\meas)$-normed module if there exists a ``local norm''
$|\cdot|:M\to \{f\in L^2(X,\meas):\ f\geq 0\}$ satisfying:
\begin{itemize}
\item[(a)] $|m+m'|\leq |m|+|m'|$ $\meas$-a.e.~in $X$ for all $m,\,m'\in M$;
\item[(b)] $|\chi m|=|\chi||m|$ $\meas$-a.e.~in $X$ for all $m\in M$, $\chi\in L^\infty(X,\meas)$;
\item[(c)] the function
\begin{equation}\label{eq:natural_norm}
\|m\| := \biggl(\int_X|m(x)|^2\di\meas(x)\biggr)^{1/2}
\end{equation}
is a norm in $M$ which coincides with $\|\cdot\|_B$.
\end{itemize}
Notice that homogeneity and subadditivity of $\|\cdot\|$ are obvious consequences of (a), (b).

The starting point of Gigli's construction is provided by the formal expressions $\{(A_i,\nabla f_i)\}_{i\in I}$, where $I$ is a finite
index set, $\{A_i\}_{i\in I}$ 
is a $\meas$-measurable partition of $X$ and $f_i\in H^{1,2}(X,\dist,\meas)$. The sum of two families $\{(A_i,\nabla f_i)\}_{i\in I}$, 
$\{(B_j,\nabla g_i)\}_{j\in J}$ is
$\{A_i\cap B_j,\nabla (f_i+g_j)\}_{(i,j)\in I\times J}$ and multiplication by $\meas$-measurable functions $\chi$ taking finitely many values 
is defined by
$$
\chi\{(E_i,\nabla f_i)\}_{i\in I}=\{(E_i\cap F_j,\nabla (z_jf_i))\}_{(i,j)\in I\times J}\qquad\text{with $\chi=\sum_{j=1}^Nz_j1_{F_j}$.}
$$
Two families $\{(A_i,\nabla f_i)\}_{i\in I}$, $\{(B_j,\nabla g_i)\}_{j\in J}$ are said to be equivalent if
$f_i=g_j$ $\meas$-a.e. on $A_i\cap B_j$ for all $(i,j)\in I\times J$ and one works with the vector space $M$
of these equivalence classes, since the above defined operations are compatible with the equivalence relation. 

The local norm $|\{(A_i,\nabla f_i)\}|\in L^2(X,\meas)$ of $\{(A_i,\nabla f_i)\}$ is defined by 
$$
|\{(A_i,\nabla f_i)\}|(x):=|\nabla f_i|(x)\qquad\text{$\meas$-a.e. on $A_i$.}
$$
Thanks to the locality properties of the minimal relaxed slope, this definition does not depend on the choice of
the representative and satisfies $|\chi \{(A_i,\nabla f_i)\}|=|\chi||\{(A_i,\nabla f_i)\}|$ whenever $\chi$ takes finitely many values. 

This way, all properties of $L^2(X,\meas)$ normed modules are satisfied, with the only difference that multiplication is
defined only for functions $\chi\in L^\infty(X,\meas)$ having finitely many values. By completion of $M$ with respect to the norm 
$\bigl(\int_X |\{A_i,f_i\}|^2\di\meas\bigr)^{1/2}$ we obtain the normed module $L^2(T(X,\dist,\meas))$.   

In the sequel we shall denote by $V,\,W$, etc. the typical elements of $L^2(T(X,\dist,\meas))$ and by $|V|$ the local norm.
As in other papers on this topic we start using a more intuitive notation, using $\nabla f$ for (the equivalence class of)
$\{(X,\nabla f)\}$ and expressions like finite sums $\sum_i\chi_i f_i$.

The following result is a simple consequence of the definition of $L^2(T(X,\dist,\meas))$.

\begin{theorem} \label{thm:Gigli_density} The vector space
$$
\left\{\sum_{i=1}^n\chi_i\nabla f_i:\ \chi_i\in L^\infty(X,\meas),\,\,f_i\in H^{1,2}(X,\dist,\meas),\,\,n\geq 1\right\}
$$
is dense in $L^2(T(X,\dist,\meas))$.
\end{theorem}

More generally, density still holds if the functions $\chi_i$ vary in a set $D\subset L^2\cap L^\infty(X,\meas)$ stable
under truncations and dense in $L^2(X,\meas)$ (such as $\Lip_b(X,\dist)\cap L^2(X,\meas)$).

Since Theorem~\ref{thm:duke} guarantees that the square $|\cdot|^2$ of the local norm
satisfies $\meas$-a.e. the parallelogram rule, the same holds on $L^2(T(X,\dist,\meas))$, therefore
one has the following:

\begin{proposition}[The canonical metric $g$]\label{prop:gcanonical}
There exists a unique symmetric and $L^\infty(X,\meas)$-bilinear form $g$ on $L^2(T(X,\dist,\meas))\times L^2(T(X,\dist,\meas))$ 
satisfying
$$
g(\nabla f_1,\nabla f_2)=\langle\nabla f_1,\nabla f_2\rangle\qquad\text{$\meas$-a.e. on $X$}
$$
for all $f_1,\,f_2\in H^{1,2}(X,\dist,\meas)$.
\end{proposition}

The Riemannian metric $g$ can be canonically viewed not only as a quadratic form or bilinear form, but also as a linear operator,
that we shall denote
$\bf g$, on the symmetric product bundle. Recalling the definition \eqref{eq:deftest} of $\mathrm{Test}F(X,\dist,\meas)$, 
the normed module $L^2(T^{\otimes 2}(X,\dist,\meas))$ of the $L^2$ sections of the symmetric tensor product of tangent bundles as constructed in \cite{Gigli} 
 arises as the $L^2$ completion
of the finite sums $\sum_i \chi_i \nabla f^1_i\otimes\nabla f^2_i$ (with $\chi_i,\,f_i^1,\,f_i^2\in\mathrm{Test}F(X,\dist,\meas)$)
with respect to a canonical Hilbert-Schmidt local norm $|\cdot|_{HS}$ (see also Definition~\ref{def:extralabel} below). 
Given this construction,  we define ${\bf g}$ as follows:
\begin{equation}\label{def:bfg}
\langle{\bf g},\sum_i \chi_i \nabla f_i^1\otimes\nabla f^2_i\rangle:=\sum_i \chi_i g(\nabla f^1_i,\nabla f^2_i).
\end{equation}
Notice that at this stage it is not clear whether the (dual) Hilbert-Schmidt norm of $\bf g$ is finite, so that $\bf g$
might not admit in general an extension to the whole of $L^2(T^{\otimes 2}(X,\dist,\meas))$.

We shall also use the $L^2(X,\meas)$-normed module $L^2(T^*(X,\dist,\meas))$ which is the dual of $L^2(T(X,\dist,\meas))$ according to \cite[Def.~1.2.6]{Gigli}. In particular, we will use the differential operator $\dist:H^{1,2}(X,\dist,\meas) \to L^2(T^*(X,\dist,\meas))$
(acting on gradient vector fields by $\dist f(\nabla g)=\langle\nabla f,\nabla g\rangle$),
which satisfies all reasonable properties like locality, chain and Leibniz rules, see \cite[Sect.~2.2.2]{Gigli}  for details.

Moreover for all Borel subset $A$ of $X$, we define $L^2(T(A, \dist, \meas))$ by the set of all $V \in L^2(T(X, \dist, \meas))$ with $|V|_{HS} =0$ $\meas$-a.e. $x \in X \setminus A$. Similarly we define $L^2(T^*(A, \dist, \meas))$. They will be used in Subsection~\ref{proof}, where it will be more useful
to distinguish the roles of vectors and covectors.

Motivated by \eqref{def:bfg}, we define also
$$
\mathrm{Test}T^{\otimes 2}(X,\dist,\meas) := \left\{ \sum_{i=1}^n 
\chi_i \nabla f_i^1 \otimes\nabla f_i^2 \, \, : \, \, \chi_i,\, f_i^1,\,f_i^2 \in \mathrm{Test}F(X,\dist,\meas) \, \, \forall 1 \le i \le n\right\},
$$
$$
\mathrm{Test}(T^*)^{\otimes 2}(X,\dist,\meas) := \left\{ \sum_{i=1}^n 
\chi_i \dist f_i^1 \otimes \dist f_i^2 \, \, : \, \, \chi_i,\, f_i^1,\,f_i^2 \in \mathrm{Test}F(X,\dist,\meas) \, \, \forall 1 \le i \le n\right\}
$$
and $L^2((T^*)^{\otimes 2}(X, \dist, \meas))$ as the $L^2$ completion of the later one.
Then $L^2((T^*)^{\otimes 2}(X, \dist, \meas))$ is canonically isometric to the dual space of $L^2(T^{\otimes 2}(X, \dist, \meas))$ (see \cite[Sect. 3.2]{Gigli}).

The following result is a consequence of the rectifiability of the set $\mathcal{R}_n$ in Theorem~\ref{th: RCD decomposition}, which
provides a canonical isometry between the tangent bundle as defined in this paper and the tangent bundle defined via measured Gromov-Hausdorff
limits, see \cite[Th.~5.1]{GigliPasqualetto1} for the proof.

\begin{lemma}\label{riemdim} If $(X,\dist,\meas)$ is a $\RCD^*(K,N)$ space, the canonical metric $g$ of Proposition~\ref{prop:gcanonical} satisfies
\begin{equation}\label{eq:righthsk}
|{\bf g}|_{HS}^2=n\qquad\text{$\meas$-a.e. in $X$, with $n=\dim_{\dist,\meas}(X)$.}
\end{equation}
\end{lemma}

In the context of $\RCD(K,\infty)$ spaces, a good local notion of Hessian is available as symmetric bilinear form
on $L^2(T(X,\dist,\meas))$ (see also Subsection~\ref{secorm}). In this paper the Hessian will play a role only in Subsection \ref{proof}. In particular we will only use the
fact that the Hessian is defined for all $f\in D(\Delta)$ with an integral
estimate coming from Bochner's inequality \cite[Cor.~3.3.9]{Gigli}
\begin{equation}\label{eq:Hessian1}
\int_X|\mathrm{Hess}_f|^2\dist \meas\le \int_X\left(|\Delta f|^2 -K|\nabla f|^2 \right)\dist \meas \qquad \forall f \in D(\Delta).
\end{equation}
In addition, we shall use the property (see \cite{Savare}, \cite[Prop.~3.3.22]{Gigli}) that, for $f_i\in D(\Delta)$ with $|\nabla f_i| \in L^\infty(X,\meas)(i=1, 2)$, 
one has $\langle\nabla f_1,\nabla f_2\rangle\in H^{1,2}(X,\dist,\meas)$, with
\begin{equation}\label{eq:Hessian2} 
\di \langle\nabla f_1,\nabla f_2\rangle=
\mathrm{Hess}_{f_1}(\nabla f_2,\cdot)+\mathrm{Hess}_{f_2}(\nabla f_1,\cdot)
\qquad\text{in $L^2(T^*(X, \dist, \meas))$.}
\end{equation}

\section{Embeddings to $L^2$-spaces via heat kernels}\label{sec:embedding}

In this section we study the properties induced by the family of continuous embeddings $(\Phi_t)_{t>0}$ of a compact $\RCD^*(K,N)$ space $(X,\dist,\meas)$ into $L^2 (X,\meas)$. Each map $\Phi_t:X\to L^2(X,\meas)$ is defined as follows:
\begin{equation}\label{eq:mancava}
\Phi_t(x) := p(x,\cdot,t) \qquad \forall x \in X.
\end{equation}
Here $p:X\times X\times (0,\infty)\to (0,\infty)$ 
denotes the H\"older continuous representative of the heat kernel of $(X,\dist,\meas)$ and, in this section, we are assuming that $\meas$
has full support. 

We start with a brief account of the Riemannian picture, in which it is known from \cite{BerardBessonGallot} that the embeddings $\Phi_{t}$ are smooth and provide a family of pull-back metrics $\Phi_t^* g_{L^2}$ which, after rescaling, nicely converge to the original metric as 
$t$ goes to $0$ (see \eqref{eq: BBG expansion} below). We focus afterwards on the (possibly non-smooth) $\RCD^*(K,N)$ setting. 
To treat properly the Riemannian result (\ref{eq: BBG expansion}) in this context, we first introduce a meaningful notion of Riemannian metric on $(X,\dist,\meas)$. Among these Riemannian metrics on $(X,\dist,\meas)$ there is a canonical one $g$, singled out by
Proposition~\ref{prop:gcanonical}, which obviously coincides with the classical metric when $(X,\dist,\meas)$ is a weighted Riemannian manifold. Finally we define a family of well-chosen Riemannian metrics $g_{t}$ serving as pull-back metrics on $(X,\dist,\meas)$. 
The convergence of $\mathrm{sc}_t g_t$ to $g$, where $\mathrm{sc}_t$ is a suitable scaling function, will be treated in Section~\ref{sec:5}.

\subsection{Smooth case}
Let $(M^n, g)$ be an $n$-dimensional closed Riemannian manifold equipped with its canonical Riemannian distance $\dist_g$ and volume measure $\mathrm{vol}_g$. The next proposition is similar to \cite[Th.~5]{BerardBessonGallot}. We give a proof for the reader's convenience.

\begin{proposition}\label{prop:smooth embedding}
For any $t>0$ the map $\Phi_t$ is a smooth embedding. Moreover the differential $d_x\Phi_t: T_xM^n \to L^2(M^n, \mathrm{vol}_g)$ at $x \in M^n$ is given by
\begin{equation}\label{eq:derivative}
d_x \Phi_t(v) =  y\mapsto g_x(\nabla_ x p (x,y,t),v)\qquad\forall v\in T_xM^n.
\end{equation}
In particular 
$$
\|d_x\Phi_t(v) \|_{L^2(M^n,{\rm vol}_g)}^2=\int_{M^n}|g_x(\nabla_xp(x, y, t), v)|^2 \di \mathrm{vol}_g(y)\qquad\forall v\in T_xM^n.
$$
\end{proposition}
\begin{proof} 
We first check that $\Phi_t$ is a continuous embedding. Continuity is obvious. As $(M^n, \dist_g)$ is compact, 
it suffices to show that $\Phi_t$ is injective. 
Recall the expression \eqref{eq:expansion1} of the heat kernel, we see that $\Phi_t(x_1)=\Phi_t(x_2)$ yields
\begin{equation}\label{eq:expan}
\sum_ie^{-\lambda_it}\phi_i(x_1) \phi_i(y)=\sum_ie^{-\lambda_it}\phi_i(x_2) \phi_i(y) \quad \text{for ${\rm vol}_g$-a.e. } y \in M.
\end{equation}
In particular, multiplying both sides of \eqref{eq:expan} by $\phi_j(y)$ and integrating over $M$ shows that 
$\phi_j(x_1)=\phi_j(x_2)$ holds for all $j$. 
Then since $p(x_1, x_1, s)=p(x_1, x_2, s)$ for all $s>0$ by (\ref{eq:expansion1}), the Gaussian bounds (\ref{eq:gaussian}) 
with $\epsilon=1$ yield
\begin{align*}
\frac{1}{C_1\mathrm{vol} (B_{s^{1/2}}(x_1))}\exp \left( -C_2s \right) &\le p(x_1, x_1, s) =p(x_1, x_2, s) \\
&\le \frac{C_1}{\mathrm{vol} (B_{s^{1/2}}(x_1))}\exp \left(-\frac{\dist^2 (x_1, x_2)}{5s}+C_2s \right), \\
\end{align*}
i.e. $\exp \bigl( -C_2s \bigr)  \le C_1^2 \exp \bigl(-\dist^2 (x_1, x_2)/(5s)+C_2s\bigr)$.
Then letting $s \downarrow 0$ yields $x_1=x_2$, which shows that $\Phi_t$ is injective.

Next we prove the smoothness of $\Phi_t$ along with (\ref{eq:derivative}).
Take a smooth curve $c:( -\epsilon, \epsilon) \to M^n$ with $c(0)=x$ and $c'(0)=v$ and estimate
\begin{align}\label{eq:diff}
&  \quad \quad \left\| \frac{\Phi_{t} \circ c (h) - \Phi_{t} \circ c (0)}{h} - g_x(\nabla_ x p (x,y,t),v)\right\|_{L^2}^{2} \nonumber \\
& = \int_{M^n} \left| \frac{p(c(h),y,t) - p(c(0),y,t)}{h} -  \frac{\di}{\di s}\Bigl |_{s=0}  p(c(s), y, t) \right|^2 \di \mathrm{vol}_g(y) \nonumber \\
& = \int_{M^n} \left| \int_{0}^{h} \frac{s}{h} \mathrm{Hess}_{p(c(s), \cdot, t)}\left(c'(s),c'(s)\right) \di s \right|^2 \di \mathrm{vol}_g \nonumber \\
& \le h\int_{M^n} \int_{0}^{h} \left|\mathrm{Hess}_{p(c(s), \cdot, t)}\left(c'(s),c'(s)\right) \right|^2 \di s \,  \di \mathrm{vol}_g,
\end{align}
where we applied the identity $f(h)=f(0) + f'(0)h -\int_0^{h}sf''(s)\di s$,
valid for any $f \in C^2(-\epsilon,\epsilon)$, to the family of functions $f_y(s):=p(c(s), y, t)$, $y\in M^n$.
Thus, letting $h \to 0$ in (\ref{eq:diff}) shows that $\Phi_t$ is differentiable at $x \in M^n$ and that 
(\ref{eq:derivative}) holds. The smoothness of $\Phi_t$ follows similarly.
\end{proof}

Let $g_{L^2}$ be the ``flat'' Riemannian metric on $L^2(M^n, \mathrm{vol}_g)$ 
given by the $L^2$ scalar product. Thanks to Proposition~\ref{prop:smooth embedding}, for any $t>0$ one can consider the pull-back metric $\Phi_t^*g_{L^2}$ which writes as follows:
\begin{equation}\label{eq:pullbackmetric}
\Phi_t^*g_{L^2} (v, w):=\int_{M^n} g_x(\nabla_xp(x, y, t), v) g_x(\nabla_xp(x, y, t), w) \di \mathrm{vol}_g(y), \quad \forall v, w \in T_xM^n.  
\end{equation}
The asymptotic behavior of $\Phi_t^*g_{L^2}$ was discussed in \cite[Th.~5]{BerardBessonGallot} where the authors showed
\begin{equation}\label{eq: BBG expansion}
c(n)t^{(n+2)/2}\Phi_t^*g_{L^2} =  g - \frac{2t}{3}\left(\mathrm{Ric}_g - \frac{1}{2}\mathrm{Scal}_g\,  g\right) + O(t^2), \quad t \downarrow 0,
\end{equation}
in the sense of pointwise convergence, where $c(n)$ is a positive dimensional constant and $\mathrm{Ric}_g, \, \mathrm{Scal}_g$ denote the Ricci and the scalar curvature of $(M^n,g)$ respectively. Note that B\'erard, Besson and Gallot considered the maps $\Psi_{t} : M^n \to \ell^{2}(\setN)$ defined by:
$$
\Psi_t(x):=c(n)^{1/2} t^{(n+2)/4} (e^{-\lambda_i t /2} \phi_i(x))_{i \ge 1} \qquad \forall x \in M.
$$
A direct computation based on the formula $ \nabla_xp(x, y, t) = \sum_{i} e^{-\lambda_{i}t} \phi_{i}(y)\nabla_x \phi_i(x)$ yields to
$$\Psi_{2t}^{*}g_{\ell^2}= c(n) t^{(n+2)/2}\Phi_t^* g_{L^2}.$$
Therefore, to work with $\Psi_{2t}^{*}$ or $\Phi_{t}^{*}$ is equivalent in this context.

\subsection{$\RCD$-setting}\label{secorm}

We replace now the Riemannian manifold $(M^n,g)$ by a compact $\RCD^*(K,N)$ space $(X,\dist,\meas)$. It is immediate that even in this case 
the maps $\Phi_t$ are continuous embeddings. Indeed, since (\ref{eq:expansion1}) holds in the $\RCD^*(K,N)$ setting too, we can carry out the proof of Proposition~\ref{prop:smooth embedding} to get that $\Phi_t$ is an embedding for any $t>0$. Continuity is obvious as we consider the 
continuous representative of the heat kernel, see the Appendix for more details.

Let us turn now to an analog of the expansion (\ref{eq: BBG expansion}) in the $\RCD^*(K,N)$ setting. As there is presently no pointwise notion of Ricci and scalar curvature in this context, it is unlikely to get such a precise expansion. One might however be interested in the convergence statement:
\begin{equation}\label{eq: convergence}
c(n)t^{(n+2)/2}\Phi_t^*g_{L^2} \to  g  \qquad  t \downarrow 0.
\end{equation}
In order to give a meaning to this statement on $(X,\dist,\meas)$, let us first introduce a nonsmooth notion of Riemannian metric
(recall that $L^0(X,\meas)$ denotes the space of $\meas$-measurable functions on $(X,\meas)$).

\begin{definition}[Riemannian metrics]\label{def:Riemannian metric}
We say that a symmetric bilinear form $$\bar g:L^2(T(X,\dist,\meas))\times L^2(T(X,\dist,\meas))\to L^0(X,\meas)$$ is a 
\textit{Riemannian (semi, resp.) metric} on $(X, \dist, \meas)$ if the following two properties hold:
\begin{itemize}
\item[(1)] ($L^\infty$-linearity) $\bar g(\chi V,W)=\chi \bar g(V,W)$ $\meas$-a.e. on $X$ for all $V,\,W\in L^2(T(X,\dist,\meas))$, $\chi\in L^\infty(X,\meas)$;
\item[(2)] (non (semi, resp.) degeneracy) for all $V\in L^2(T(X,\dist,\meas))$ one has
\begin{equation}\label{eq:positivity}
\bar g(V, V)>0\quad(\bar g(V, V) \ge 0, \,resp.) \qquad\text{$\meas$-a.e. on $\{|V|>0\}$.}
\end{equation}
\end{itemize}
In the sequel we denote by $g$ the canonical metric singled out by Proposition~\ref{prop:gcanonical}.
\end{definition}
As we did for the canonical metric $g$ in \eqref{def:bfg}, if we consider a Riemannian semi metric, we can also define an associated lifted metric ${\bf\bar g}$ acting on suitable elements of $L^2(T^{\otimes 2}(X,\dist,\meas))$ in the following way: for any $\chi_i \in L^{\infty}(X, \meas), f_i^j \in H^{1, 2}(X, \dist, \meas)$, we set
\begin{equation}\label{nnnnnb}
\langle{\bf \bar g},\sum_i \chi_i \nabla f_i^1\otimes\nabla f^2_i\rangle:=\sum_i \chi_i \bar g(\nabla f^1_i,\nabla f^2_i);
\end{equation}
 more generally, we shall also apply this construction to
$L^\infty(X, \meas)$-bilinear forms 
$\bar g:[L^2(T(X,\dist,\meas))]^2\to L^0(X,\meas)$ (for instance, differences of metrics).

In the class of Riemannian semi metrics a natural partial order, that we shall use, is induced by the relation
\begin{equation}\label{eq:partial_order}
g_1\leq g_2\quad\Longleftrightarrow\quad
g_1(V,V)\leq g_2(V,V)\qquad\text{$\meas$-a.e. in $X$, for all $V\in L^2(T(X,\dist,\meas))$.}
\end{equation}
It is also obvious that the class of Riemannian semi metrics is invariant under multiplication by 
positive $\meas$-a.e.~functions in $L^0(X,\meas)$. This motivates the following definition.

\begin{definition}[Local norm of a Riemannian semi metric]\label{def:extralabel}
For a $L^\infty(X, \meas)$-bilinear form $\bar g:[L^2(T(X,\dist,\meas))]^2\to L^0(X,\meas)$,
the smallest $\meas$-measurable function $h:X\to [0,\infty]$, up to $\meas$-measurable sets, satisfying
$$
|\langle{\bf \bar g},\sum_i\chi_i\nabla f_i^1\otimes\nabla f_i^2\rangle|\leq h\bigl|\sum_i\chi_i \nabla f_i^1\otimes\nabla f_i^2\bigr\vert_{HS}
\quad\text{$\meas$-a.e. in $X$}
$$
for all $\sum_i\chi_i \nabla f_i^1\otimes\nabla f_i^2\in \mathrm{Test}T^{\otimes 2}(X,\dist,\meas)$, is denoted $|{\bf \bar g}|_{HS}$ or $|{\bf \bar g}|$ for short.
\end{definition}

Whenever $|{\bf \bar g}|_{HS}\in L^0(X,\meas)$ we have a unique extension of ${\bf\bar g}$, still denoted ${\bf \bar g}$,
to the completion of $\mathrm{Test}(T^{\otimes 2}(X,\dist,\meas))$, namely $L^2(T^{\otimes 2}(X,\dist,\meas))$.

\begin{remark}\label{99jjn}
Any $T=\sum_i^k f_i^0\dist f_i^1 \otimes \dist f_i^2 \in \mathrm{Test}(T^*)^{\otimes 2}(X, \dist, \meas)$ 
(i.e. $f_i^j \in \mathrm{Test}F(X, \dist, \meas))$ induces the $L^\infty(X, \meas)$-bilinear form $b_T$ as follows:
$$
b_T(V, W) :=\sum_{i=1}^kf_i^0\langle \nabla f_i^1, V\rangle \langle \nabla f_i^2, W\rangle \in L^0(X, \meas) \qquad 
\forall V,\,W \in L^2(T(X, \dist, \meas))
$$
with the same Hilbert-Schmidt norm: $|{\bf b}_T|_{HS}(x)=|T|_{HS}(x)$ for $\meas$-a.e.~$x \in X$.
This observation can be extended to the case when $T \in L^2((T^*)^{\otimes 2}(X, \dist, \meas))$, i.e.~any  $T \in L^2((T^*)^{\otimes 2}(X, \dist, \meas))$ induces the bilinear form ${\bf b}_T$ with the same Hilbert-Schmidt norm.

Conversely, for any $L^\infty(X, \meas)$-bilinear form ${\bar g}$ with $|{\bf {\bar g}}|_{HS} \in L^2(X, \meas)$, ${\bar g}$ defines an element in $(L^2(T^{\otimes 2}(X, \dist, \meas)))^*$ by (\ref{nnnnnb}). In particular, since $L^2(T^{\otimes 2}(X, \dist, \meas))^* \cong L^2((T^*)^{\otimes 2}(X, \dist, \meas))$, there exists a unique $T \in L^2((T^*)^{\otimes 2}(X, \dist, \meas))$ such that $b_T={\bar g}$.

Therefore we will sometimes regard any Riemannian semi metric ${\bar g}$ with $|{\bf {\bar g}}|_{HS} \in L^2(X, \meas)$ 
as an element ${\bf {\bar g}}$ in $L^2((T^*)^{\otimes 2}(X, \dist, \meas))$, without making explict the distinction (e.g. in \eqref{eq:explicit expression}).
\end{remark}

Finally, we introduce a suitable notion of convergence of Riemannian semi metrics $\bar g_i$ on a fixed 
$\RCD^*(K, N)$ space $(X,\dist,\meas)$.

\begin{definition}[Convergence of Riemannian semi metrics]\label{defcong}
We say that Riemannian semi metrics $\bar g_i$ $L^2$-weakly converge to a Riemannian semi metric $\bar g$ if  $\sup_i\||{\bf \bar g_i}|_{HS}\|_{L^2}<\infty$ and
${\bar g}_i(V,V)$ $L^2$-weakly converges to $\bar g(V,V)$ for all $V \in L^{\infty}(T(X,\dist,\meas))$. We say that
$\bar g_i\to\bar g$ $L^2$-strongly if $|{\bf \bar g_i}-{\bf \bar g}|_{HS}\to 0$ in $L^2(X,\meas)$.
\end{definition}

In the previous definition, the adjective ``weakly'' refers also to the fact that convergence is required
in a pointwise sense, namely without any uniformity w.r.t. $V$, even though the convergence with
$V$ fixed might occur in the strong $L^2$ sense. Also, this terminology is justified by
the fact that this notion of convergence corresponds precisely to weak convergence in the reflexive 
space $L^2(T^{\otimes 2}(X,\dist,\meas))$, since ${\bf\bar g}$ is uniquely determined by its
value on tensor products $V\otimes V$, namely $\bar g(V,V)$. As a consequence, one has
$$
\liminf_{i\to\infty}\int_X|{\bf\bar g_i}|_{HS}^2\di\meas_i\geq
\int_X|{\bf\bar g}|_{HS}^2\di\meas
$$
whenever $\bar g_i$ $L^2$-weakly converge to $\bar g$.

Notice also that $L^2$-strong convergence of
$\bar g_i$ to $\bar g$ implies strong convergence in $L^2$ of $\bar g_i(V,V)$ to $\bar g(V,V)$ for all
$V\in L^{\infty}(T(X,\dist,\meas))$ because 
$$
|\bar g_i(V,V)-\bar g(V,V)|\leq\|V\|_{L^{\infty}}^2 |{\bf\bar g_i}-{\bf\bar g}|_{HS},
$$
so that by integration the $L^2$ convergence of $\bar g_i(V,V)$ to $\bar g(V,V)$ can be obtained.

Similarly, if $\bar g_i \le C\bar g$ for some $C\geq 0$ independent of $i$, then the $L^2$-strong convergence of ${\bar g}_i$ to ${\bar g}$ implies that ${\bar g}_i(V, V) \to {\bar g}(V, V)$ in $L^1(X, \meas)$ for all $V \in L^2(T(X, \dist, \meas))$.

The following convergence criterion will also be useful.

\begin{proposition}\label{prop:convcrigi}
Let $\bar g_i,\,\bar g$ be Riemannian semi metrics.
Then $\bar g_i$ $L^2$-strongly converge to $\bar g$ as $i\to\infty$ if and only if
\begin{equation}\label{eq:convcrigi}
\lim_{i\to\infty}\int_X \bar g_i(V,V)\di\meas=\int_X \bar g(V,V)\di\meas\qquad\forall V\in L^\infty(T(X,\dist,\meas))
\end{equation}
and
$$
\limsup_{i\to\infty}\int_X|{\bf\bar g}_i|_{HS}^2\di\meas\leq
\int_X|{\bf\bar g}|_{HS}^2\di\meas<\infty.
$$ 
\end{proposition}
\begin{proof} One implication is obvious. To prove the converse, by the reflexivity of
$L^2(T^{\otimes 2}(X,\dist,\meas))$ it is sufficient to check the weak convergence 
in that space of ${\bf\bar g}_i$ to ${\bf\bar g}$.

Then replacing $V$ by $1_AV$ in \eqref{eq:convcrigi} for all Borel subset $A$ of $X$ yields that $\int_A{\bar g}_i(V, V)\di \meas \to \int_A{\bar g}(V, V)\di \meas$ as $i \to \infty$, which implies the $L^2$-weak convergence of ${\bar g}_i$ to ${\bar g}$.
\end{proof}

For any $t>0$, a natural way to define a pull-back Riemannian semi metric $g_t$ on $(X,\dist,\meas)$ is based on an integral version of
\eqref{eq:pullbackmetric}, namely $g_t (V_1, V_2)$ satisfies:
\begin{align}\label{eq:rrcd pull back}
\int_Xg_t (V_1, V_2) (x)\di\meas(x) & = \int_X\Bigl(
\int_X\langle \nabla_xp(x, y, t), V_1(x)\rangle \langle \nabla_xp(x, y, t), V_2(x)\rangle\di\meas(y)\Bigr)\di\meas (x), \nonumber\\
&\quad \quad \quad \quad \quad  \forall V_1,\, V_2 \in L^2(T(X, \dist, \meas)).
\end{align}

To see that this is a good definition (see also the next subsection for another equivalent definition), 
notice that the integrand $G(x,y)$ in the right hand side
of \eqref{eq:rrcd pull back} is pointwise defined as a map $y\mapsto G(\cdot,y)$ with values in $L^2(X,\meas)$ ($L^2$ integrability
follows by the Gaussian estimate \eqref{eq:equi lip}). By Fubini's theorem also the map $x\mapsto \int_X G(x,y)\di\meas(y)$
is well defined, up to $\meas$-negligible sets, and this provides us with the pointwise definition, up
to $\meas$-negligible sets, of $g_t (V_1, V_2)$, namely
\begin{equation}\label{eq:rrcd pull back_bis} 
g_t (V_1, V_2) (x)  = \int_X\langle \nabla_xp(x, y, t), V_1(x)\rangle \langle \nabla_xp(x, y, t), V_2(x)\rangle\di\meas(y).
\end{equation}
As a matter of fact, since many objects of the theory are defined only up to $\meas$-measurable sets, we shall mostly work
with the equivalent integral formulation.

It is obvious that \eqref{eq:rrcd pull back_bis} defines a symmetric bilinear form on $L^2(T(X,\dist,\meas))$ with 
values in $L^0(X,\meas)$ and with the $L^\infty(X,\meas)$-linearity property.
The next proposition ensures that $g_t$ is indeed a Riemannian semi metric on $(X,\dist,\meas)$,
provides an estimate from above in terms of the canonical metric, and the representation of the lifted metric ${\bf g}_t$.

\begin{proposition}\label{prop:riemanexist}
Formula \eqref{eq:rrcd pull back_bis} defines a Riemannian metric $g_t$ on $L^2(T(X,\dist,\meas))$ with
\begin{eqnarray}\label{eq:normaPhi_t}
\int_X |{\bf g}_t|^2_{HS}\di\meas&=& \sum_i e^{-2\lambda_i t}\int_X g_t(\nabla\phi_i,\nabla\phi_i)\di\meas\\
&=&\sum_i e^{-2\lambda_it}\int_X\int_X|\langle\nabla_x p(x,y,t),\nabla\phi_i\rangle|^2\di\meas(y)\di\meas(x),\nonumber
\end{eqnarray}
\begin{equation}\label{eq:normaPhi_t_other}
|{\bf g}_t|_{HS}(x)=\biggl\vert\int_X\dist_xp(x,y,t)\otimes \dist_x p(x,y,t)\di\meas(y)\biggr\vert_{HS}\qquad\text{for $\meas$-a.e. $x\in X$}
\end{equation}
and representable as the convergent\footnote{in the sense that $
\int_X |g_t-\sum_{i=1}^le^{-2\lambda_it}\dist \phi_i \otimes \dist \phi_i|_{HS}^2 \di \meas \to 0$ as $l \to +\infty$} series
\begin{equation}\label{eq:explicit expression}
{\bf g}_t=\sum_{i=1}^{\infty}e^{-2\lambda_it}\dist \phi_i \otimes \dist \phi_i
\qquad\text{in $L^2((T^*)^{\otimes 2}(X, \dist, \meas))$.} 
\end{equation}
Moreover, the rescaled metric
$t\meas (B_{\sqrt{t}}(\cdot))g_t$ satisfies
\begin{equation}\label{eq:riem est}
t\meas (B_{\sqrt{t}}(\cdot)) g_t\le C(K,N)g\qquad\forall t\in (0,C_4^{-1}),
\end{equation}
where $C_4$ is the constant in (\ref{eq:equi lip}).
\end{proposition} 

\begin{proof} Let us prove (\ref{eq:riem est}), assuming $0<t<\min\{1,C_4^{-1}\}$.
For $V\in L^2(T(X,\dist,\meas))$ and $y\in X$, the Gaussian estimate (\ref{eq:equi lip}) with $\epsilon=1$ and the upper bound on $t$ yield
\begin{equation}\label{eq:1105}
\int_X|\langle \nabla_xp(x, y, t),V(x)\rangle|^2 \di\meas(x) \le 
\int_X\frac{C_3^2e^2}{t \meas(B_{\sqrt{t}}(x))^2}\exp \left( \frac{-2\dist (x, y)^2}{5t}\right) |V(x)|^2\di\meas(x).
\end{equation}
By integration with respect to $y$ and taking into account \eqref{lem:bound} with $\ell=0$ (applied to the rescaled space
$(X,\dist_t,\meas)$ with $\dist_t=\sqrt{t}^{-1}\dist$, whose constants $c_0,\,c_1,\,c_2$ can be estimated uniformly w.r.t. $t$, since 
$(X,\dist_t,\meas)$ is $\RCD^*(t^2 K,N)$), we recover \eqref{eq:riem est}.

Let us prove now the non-degeneracy condition \eqref{eq:positivity}, using the expansion \eqref{eq:expansion2} of $\nabla_x p$. 
For all $V \in L^2(T(X, \dist,\meas))$ we have
\begin{align}\label{eigen exp}
&\int_Xg_t(V, V)\di\meas \nonumber \\
&=\int_X\int_X \langle \nabla_x p(x, y, t), V(x) \rangle^2\di\meas (x) \di\meas (y) \nonumber \\
&=\int_X\int_X \left(\sum_ie^{-\lambda_it}\phi_i(y)\langle \nabla \phi_i, V\rangle (x) \right)^2 \di\meas (x) \di\meas (y) \nonumber \\
&=\int_X\int_X \sum_{i,\, j}e^{-(\lambda_i+\lambda_j)t}\phi_i(y)\phi_j(y)\langle \nabla \phi_i, V\rangle (x)\langle \nabla \phi_j, V\rangle (x) \di\meas (x) \di\meas (y)\nonumber \\
&=\sum_ie^{-2\lambda_it}\int_X\langle \nabla \phi_i, V\rangle^2 \di\meas.
\end{align}

By $L^\infty$-linearity, it suffices to check that $\|g_t(V, V)\|_{L^1}=0$ implies $|V|(x)=0$ for $\meas$-a.e. $x \in X$.
Thus assume $\|g_t(V, V)\|_{L^1}=0$.  Then (\ref{eigen exp}) yields that for all $i$,
\begin{equation}\label{eq:couple}
\langle \nabla \phi_i, V\rangle (x)=0 \quad \text{for $\meas$-a.e. $x \in X$}.
\end{equation}
Since $L^2(T(X, \dist, \meas))$ is generated, in the sense of $L^2$-modules, by 
$\{\nabla f :\ f \in H^{1, 2}(X, \dist, \meas)\}$ and since the vector space spanned
by $\phi_i$ is dense in $H^{1,2}(X,\dist,\meas)$, it is easily seen that 
$L^2(T(X, \dist, \meas))$ is generated, in the sense of $L^2$-modules, also by 
$\{\nabla\phi_i :\ i\geq 1\}$. In particular (\ref{eq:couple}) shows that $V=0$.

In order to prove \eqref{eq:normaPhi_t} and \eqref{eq:explicit expression}, fix an integer $N\geq 1$ and let 
$$
{\bf g}_t^N:=\sum_{i=1}^{N}e^{-2\lambda_it}\dist \phi_i \otimes \dist \phi_i.
$$
Then
\begin{align}\label{eq:disuno}
\int_X|{\bf g}_t^N|_{HS}^2\di\meas &=\sum_{i,\, j=1}^Ne^{-2(\lambda_i+\lambda_j)t}\int_X\langle \nabla \phi_i, \nabla \phi_j\rangle ^2 \di\meas \\
&=\sum_{i=1}^Ne^{-2\lambda_it}\left( \sum_{j=1}^Ne^{-2\lambda_jt}\int_X\langle \nabla \phi_i, \nabla \phi_j\rangle ^2\di\meas \right)\nonumber \\
&\le \sum_{i=1}^{\infty}e^{-2\lambda_it}\int_Xg_t(\nabla \phi_i, \nabla \phi_i)\di\meas \nonumber\\
&\le C\sum_{i=1}^{\infty}e^{-2\lambda_i t}\int_X|\nabla \phi_i|^2\di\meas\le C\sum_{i=1}^{\infty}e^{-2\lambda_i t}\lambda_i<\infty\nonumber,
\end{align}
where $C=C(K,N,t,\meas (X))$ and we used \eqref{eq:riem est} and \eqref{eigen exp}, together with a uniform lower bound on
$\meas(B_{\sqrt{t}}(x))$. By Proposition \ref{prop:lowerbound},
an analogous computation shows that $\||{\bf g}^N_t-{\bf g}^M_t|_{HS}\|_2\to 0$ as $N,\,M\to\infty$, hence ${\bf g}_t^N\to
{\bf\tilde g}_t$ in $L^2((T^*)^{\otimes 2}(X, \dist, \meas))$.

Passing to the limit in the identity
$$
\int_X\langle {\bf g}^N_t,\chi^2\nabla f\otimes\nabla f\rangle\di\meas
=\sum_{i=1}^Ne^{-2\lambda_it}\int_X\chi^2\langle \nabla \phi_i, \nabla f\rangle ^2 \di\meas
$$
with $\chi\in L^\infty(X,\meas)$, $f\in\mathrm{Test}F(X,\dist,\meas)$, 
we obtain from (\ref{eigen exp}) with $V=\chi\nabla f$ 
$$
\int_X\langle {\bf\tilde g}_t,\chi^2\nabla f\otimes\nabla f\rangle\di\meas
=\int_X\langle {\bf g}_t,\chi^2\nabla f\otimes\nabla f\rangle\di\meas.
$$
Hence ${\bf\tilde g}_t={\bf g}_t$ (in particular ${\bf g}_t$ has finite Hilbert-Schmidt norm and ${\bf g}_t$ can be extended to
$L^2(T^{\otimes 2}(X, \dist, \meas))$).

In order to prove \eqref{eq:normaPhi_t} it is sufficient to pass to the limit as $N\to\infty$ in
$$
\int_X |{\bf g}_t^N|_{HS}^2\di\meas=\int_X\sum_{i,\, j=1}^Ne^{-2(\lambda_i+\lambda_j)t}\langle \nabla \phi_i, \nabla \phi_j\rangle ^2\di\meas,
$$
taking \eqref{eigen exp} into account.

Finally, \eqref{eq:normaPhi_t_other} follows by the observation that ${\bf g}_t$ is induced by the scalar product,
w.r.t. the Hilbert-Schmidt norm, with the vector $\int_X\dist_x p(x,y,t)\otimes \dist_x p(x,y,t)\di\meas(y)$.
\end{proof}

\subsection{The pull-back semi metric of a Lipschitz map into a Hilbert space}\label{sec:pull}
In this subsection we discuss the pull-back Riemannian semi metric of a Lipschitz map from a compact $\RCD^*(K, N)$ space into a Hilbert space in order to introduce the finite dimensional reduction (Proposition \ref{projj}).

Let $(X, \dist, \meas)$ be a compact $\RCD^*(K, N)$ space with $\dim_{\dist, \meas}(X)=n$ and $\mathrm{supp}\,\meas=X$, let $(H, \langle \cdot, \cdot \rangle)$ be a (real) separable Hilbert space and let $F:X \to H$ be a $L$-Lipschitz map.
We fix an orthonormal basis $\{e_i\}_{i\geq 1}$ of $H$ 
and denote by $F_i:X \to \mathbb{R}$ the projection of $F$ to $\mathbb{R} \cong \mathbb{R}e_i \subset H$.
\begin{lemma}\label{230}
We have
\begin{equation}\label{21}
\sum_{i=1}^\infty|\nabla F_i|^2\le nL^2 \qquad\text{$\meas$-a.e. in $X$.}
\end{equation}
\end{lemma}
\begin{proof}
Fix $\epsilon>0$ and $\ell\in\setN \setminus \{0\}$. 
Let $C$ be the Borel domain of a $(1+\epsilon)$-biLipschitz embedding $\phi : C \hookrightarrow \setR^n$, and as a consequence of 
Theorem~\ref{th: RCD decomposition} we can assume that $(\phi)_{\sharp}(\meas\res C)$ and $\mathcal{H}^n \res \phi(C)$ are mutually absolutely continuous. Let $G=(G_1,\ldots,G_\ell):\mathbb{R}^n \to \mathbb{R}^\ell$ be a $(1+\epsilon)L$-Lipschitz extension of 
$(F_1\circ \varphi^{-1},\ldots,F_\ell\circ \varphi^{-1}) : \phi(C) \to \mathbb{R}^\ell$ (granted by Kirszbraun's theorem). 
Applying the chain rule, we get that $
\sum_{i=1}^\ell |\nabla F_i|^2 \le (1+\eps)^2 \sum_{i=1}^\ell |\nabla G_i|^2\circ \phi$ $\meas$-a.e.~on $C$.
Moreover Rademacher's theorem yields
\begin{equation}\label{boun}
\sum_{i=1}^\ell |\nabla F_i|^2 \le (1+\eps)^2 \sum_{i=1}^\ell |\nabla G_i|^2\circ \phi= 
(1+\epsilon)^2 \sum_{i=1}^\ell \sum_{k=1}^n |\partial_{x_k} G_i|^2 \circ \phi \qquad \text{$\meas$-a.e.~on $C$}.
\end{equation}
Since $G$ is $(1+\epsilon)L$-Lipschitz, for any $k=1,\ldots,n$ one has
$\sum_{i=1}^\ell |\partial_k G_i|^2 \le L^2 (1+\epsilon)^2$ $\mathcal{H}^n$-a.e. on $\mathbb{R}^n$. 
This with \eqref{boun} and \cite[Th. 1.1]{MondinoNaber} implies
$$\sum_{i=1}^\ell |\nabla F_i|^2 \le (1+ \epsilon)^4 nL^2  \qquad \text{$\meas$-a.e.~on $X$.}$$ 
The result follows by letting $\epsilon\downarrow 0$ and then letting $\ell\to\infty$.
\end{proof}

\begin{proposition}
The $L^2$-tensor
\begin{equation}\label{200}
\sum_{i=1}^\infty\dist F_i \otimes \dist F_i \in L^2((T^*)^{\otimes 2}(X, \dist, \meas))
\end{equation}
defines the Riemannian semi metric $F^*g_H$ (called the pull-back metric by $F$) with $|F^*{\bf g}_H|_{HS}(x)\le nL^2$ for $\meas$-a.e. $x \in X$, 
and it does not depend on the choice of the orthonormal basis $\{e_i\}_i$.
\end{proposition}
\begin{proof}
For all $\ell\geq i$, $|\sum_{k=i}^\ell\dist F_i \otimes \dist F_k|\le \sum_{k=i}^\ell |\nabla F_i|^2$. Hence, Lemma~\ref{230} yields
$$\sum_{k=1}^\ell\dist F_i \otimes \dist F_i \to \tilde{\bf g}\qquad\text{in $L^2((T^*)^{\otimes 2}(X, \dist, \meas))$}$$ 
as $\ell\to\infty$,  with $|\tilde{{\bf g}}|_{HS}(x)\le nL^2$ for $\meas$-a.e. $x \in X$.

In order to prove the independence of \eqref{200} with respect to the choice of the orthonormal basis $\{e_i\}_i$, let us fix another
orthonormal basis $\{v_j\}_j$ of $H$ and let us denote by $G_j$ the projection of $F$ to $\mathbb{R}v_j\subset H$.
Let $\{a_{ij}\}_{ij} \subset \mathbb{R}$ with $e_i=\sum_ja_{ij}v_i$, so that the orthogonality of
$e_i$ gives $\sum_j a_{ij}a_{kj}=\delta_{ik}$.

Then, since $G_j=\sum_ia_{ij}F_i$, for all $V_1,\,V_2\in L^2(T(X, \dist, \meas))$ we have
\begin{align*}
\sum_j\int_X\langle \nabla G_j, V_1\rangle \langle \nabla G_j, V_2\rangle \di \meas 
&=\sum_j\sum_{i,\,k}a_{ij}a_{kj}\int_X\langle \nabla F_i, V_1\rangle \langle \nabla F_k, V_2\rangle \di \meas \\
&=\sum_{i,\,k}\sum_{j}a_{ij}a_{kj}\int_X\langle \nabla F_i, V_1\rangle \langle \nabla F_k, V_2\rangle \di \meas \\
&=\sum_i\int_X\langle \nabla F_i, V_1\rangle \langle \nabla F_i, V_2\rangle \di \meas,
\end{align*}
which proves the desired independence.
\end{proof}

It is clear that in the Riemannian case $(X, \dist, \meas)=(M^n, \dist_g, \mathrm{vol}_g)$, if $F$ is smooth, then (\ref{200}) is equal to the 
standard pull-back metric. More generally one has the following:

\begin{proposition}
The Lipschitz embedding $\Phi_t: X \to L^2(X, \meas)$ in \eqref{eq:mancava} satisfies $\Phi_t^*g_{L^2}=g_t$, with $g_t$ in \eqref{eq:rrcd pull back_bis}.
\end{proposition}
\begin{proof}
Note that the corresponding projection $(\Phi_t)_i$ of $\Phi_t$ to $\mathbb{R}\phi_i \subset L^2(X, \meas)$ is 
$$
(\Phi_t)_i(x)=\int_Xp(x, y, t)\phi_i(y)\di \meas = e^{-\lambda_it}\phi_i(x)
$$
which implies that $\Phi_t$ is Lipschitz because of Propositions \ref{prop:Jiang} and \ref{prop:lowerbound}.
Thus for $V_1,\,V_2\in L^2(T(X, \dist,  \meas))$ we have
\begin{align*}
\int_X\Phi_t^*g_{L^2}(V_1, V_2)\di \meas&=\sum_i\int_X\langle \nabla (e^{-\lambda_it}\phi_i), V_1\rangle \langle \nabla (e^{-\lambda_it}\phi_i), V_2\rangle \di \meas \\
&=\sum_ie^{-2\lambda_it}\langle \nabla \phi_i, V_1\rangle \langle \nabla \phi_i, V_2\rangle \di \meas \\
&=\int_Xg_t(V_1, V_2)\di \meas
\end{align*}
which completes the proof.
\end{proof}

Similarly we get the following:
\begin{proposition}\label{projj}
For all $\ell\in\mathbb{N}$, let $\Phi_t^\ell$ be the projection of $\Phi_t$ to $H:=\bigoplus_{i=1}^\ell\mathbb{R}\phi_i \subset L^2(X, \meas)$.
Then $(\Phi_t^\ell)^*g_H=g_t^\ell$, where 
\begin{equation}\label{eq:finalfinal}
{\bf g}_t^\ell:=\sum_{i=1}^\ell e^{-2\lambda_it}\dist \phi_i \otimes \dist \phi_i.
\end{equation}
\end{proposition}

\section{Convergence results via blow-up}\label{sec:5}

In this section, we study the $L^2$-convergence of the rescaled metrics $\mathrm{sc}_t g_{t}$ to $g$ as $t\to 0^+$ 
on a given compact $\RCD^*(K,N)$ space $(X,\dist,\meas)$ with $\mathrm{supp}\,\meas =X$. Here the function $\mathrm{sc}_t : X \to \setR$ is a suitable
scaling function whose expression requires an immediate discussion. In the Riemannian case $(X,\dist,\meas) = (M^n, \dist_g,\vol_g)$, 
one knows by (\ref{eq: BBG expansion}) that $\mathrm{sc}_t \equiv c(n) t^{(n+2)/2}$ where $c(n)>0$ is a constant depending 
only on the dimension $n$. In the $\RCD$ setting:

\begin{itemize}
\item The analogy with the Riemannian setting suggests to take $\mathrm{sc}_{t} =  t^{(n+2)/2}$, where $n=\dim_{\dist,\meas}(X)$
(recall Theorem~\ref{th: RCD decomposition}).
\item On the other hand, since the $\RCD$ setting is closer to a weighted Riemannian setting, we can also
set $\mathrm{sc}_t = t \meas (B_{\sqrt{t}}(\cdot))$, to take also into account the effect of the weight $\theta$, namely
the density of $\meas$ with respect to $\mathcal{H}^n\res\mathcal{R}_n$.
\end{itemize}
In both cases, we prove that $\mathrm{sc}_t g_t$ converges to a rescaled version of the canonical Riemannian metric $g$ on $(X,\dist,\meas)$, 
where the rescaling reflects the choice of $\mathrm{sc}_t$. To be more precise, we prove in Theorem~\ref{th:convergence1} that 
$\hat{g}_t:=t \meas(B_{\sqrt{t}}(x))g_t$ converge to $\hat{g}=c_n g$, with $c_n$ as in \eqref{eq:defck} below.
Concerning the second scaling, as $t^{(n+2)/2} = \frac{\sqrt{t}^n}{\meas(B_{\sqrt{t}}(x))}t \meas(B_{\sqrt{t}}(x))$, 
we prove in Theorem~\ref{th:convergence2} that the limit of the newly rescaled metrics $\tilde{g}_t$ is 
$c_n(\omega_n\theta)^{-1}1_{\mathcal{R}_n^*}\hat{g}$ (notice that this is a good definition, since $\theta$
is well-defined up to $\mathcal{H}^n$-negligible sets and $\meas$ and $\mathcal{H}^n$ are mutually
absolutely continuous on $\mathcal{R}_n^*$).

We start with introducing a technical concept, namely harmonic points of vector fields. Those are points at which a vector field 
infinitesimally (meaning after blow-up of the metric measure space) looks like the gradient of a harmonic function.

\subsection{Harmonic points}

Let us first recall the definition of Lebesgue point.

\begin{definition}[Lebesgue point]
Let $f\in L_\loc^p(X, \meas)$ with $p\in [1,\infty)$. 
We say that $x \in X$ is a $p$-Lebesgue point of $f$ if there exists $a \in \setR$ such that
\[
\lim\limits_{r \to 0} \fint_{B_r (x)} |f(y)-a|^p\di \meas (y) = 0.
\]
The real number $a$ is uniquely determined by this condition and denoted by $f^*(x)$ (we omit the $p$-dependence). 
The set of $p$-Lebesgue points of $f$ is Borel and denoted by $\Leb_p(f)$.
\end{definition}

Note that the property of being a $p$-Lebesgue point and $f^*(x)$ do not depend on the choice of the
versions of $f$, and that $x\in\Leb_p(f)$ implies
$\fint_{B_r (x)} |f(y)|^p\di\meas\to |f^*(x)|^p$ as $r\downarrow 0$.
It is well-known (e.g. \cite{Heinonen}) that the doubling property ensures 
that $\meas(X \setminus \Leb_p(f))=0$, 
and that the set $\{ x\in\Leb_p(f) :\  f^*(x) = f(x) \}$ (which does depend on the choice of representative in 
the equivalence class) has full measure in $X$. When we apply these properties to a characteristic
function $f=1_A$ we obtain that $\meas$-a.e. $x\in A$ is a point of density 1 for $A$ and
$\meas$-a.e. $x\in X\setminus A$ is a point of density 0 for $A$.

\begin{definition}[Harmonic point of a function]\label{def:buhfunction}
Let  $x \in X$, $R>0$, $z \in B_R(x)$ and let $f \in H^{1, 2}(B_R(x), \dist, \meas)$. We say that \textit{$z$ is a harmonic point of $f$} if
$z\in\Leb_2(|\nabla f|)$ and for any $(Y, \dist_Y,\meas_Y,y) \in \mathrm{Tan}(X, \dist,\meas,z)$, mGH limit of 
$(X,t_{i}^{-1} \dist,\meas(B_{t_{i}}(z))^{-1} \meas,z)$, where $t_i \to 0^+$,
there exist a subsequence $(t_{i(j)})_j$ of $(t_i)_i$ and $\hat{f} \in \Lip(Y, \dist_Y) \cap \mathrm{Harm}(Y, \dist_Y, \meas_Y)$ such that the rescaled functions $f_{t_{i(j)},z}$ $H^{1, 2}_{\mathrm{loc}}$-strongly converge to $\hat{f}$ as $j\to\infty$, where $f_{z,t}$ is defined by
$$f_{t, z}:=\frac{1}{t}\left( f-\fint_{B_{t}(z)}f\di\meas \right)$$
on $(X,t^{-1} \dist,\meas(B_{t}(z))^{-1} \meas)$.
We denote by $H(f)$ the set of harmonic points of $f$.
\end{definition}

Note that being an harmonic point also does not depend on the choice of versions of $f$ and $|\nabla f|$ and that this notion is closely related to the differentiability of $f$ at $x$. For instance in the Riemannian case $(X, \dist, \meas)=(M^n, \dist_g, \mathrm{vol}_g)$ with $f \in C^1(M^n)$, every point $x \in M^n$ is a harmonic point of $f$, and the function $\hat{f}$ appearing by blow-up is unique and equals the differential of $f$ at $x$. On the other hand if $f(x)=|x|$ on $\mathbb{R}^n$, then $0_n$ is not an harmonic point of $f$. 

The definition of harmonic point can be extended to vector fields as follows.

\begin{definition}[Harmonic point of an $L^2$-vector field]\label{def:buhderivation}
Let $V \in L^2(T(X, \dist, \meas))$ and let $z \in X$.
We say that \textit{$z$ is a harmonic point of $V$} if there exists $f \in H^{1, 2}(X, \dist, \meas)$ such that
$z\in H(f)$ and
\begin{equation}\label{eq7}
\lim_{r \downarrow 0} \fint_{B_r(z)}\left|V-\nabla f \right|^2\di\meas=0.
\end{equation}
We denote by $H(V)$ the set of harmonic points of $V$. 
\end{definition}

Obviously, if $V=\nabla f$ for some $f \in H^{1, 2}(X, \dist, \meas)$, then 
Definition~\ref{def:buhderivation} is compatible with Definition~\ref{def:buhfunction}.
Notice also that, as a consequence of \eqref{eq7} and the condition $z\in\Leb_2(|\nabla f|)$, $\fint_{B_r(z)}|V|^2\di\meas$ 
converge as $r\downarrow 0$ to $(|\nabla f|^*)^2(z)$ and we shall denote this precise value by $|V|^{2*}(z)$. 
By the Lebesgue theorem, this limit coincides for $\meas$-a.e. $z\in H(V)$ 
with $|V|^2(z)$.
The statement and proof of the following result are very closely related to Cheeger's version \cite{Cheeger} of Rademacher
theorem in metric measure spaces; we simply adapt the proof and the statement to our needs.

\begin{theorem} For all $V \in L^2(T(X, \dist, \meas))$ one has $\meas(X\setminus H(V))=0$.
\end{theorem}
\begin{proof}  \textit{Step 1: the case of gradient vector fields $V=\nabla f$.} Recall that $\RCD^*(K,N)$ spaces are local doubling and satisfy a local 
(2,2)-Poincar\'e inequality, see the discussion after \eqref{eq:locPoincaré}. 
We fix $z\in\Leb_2(|\nabla f|)$ where ${\rm Dev}(f,B_r(z))$, as defined in \eqref{eq:odev} of Theorem~\ref{thm:Cheeger}, 
is infinitesimally faster than $\meas(B_r(z))$ as $r\downarrow 0$.
Let us prove that $z\in H(f)$. Let $(t_{i})_i$ and $(Y, \dist_Y, \meas_Y, y)$, $f_{t_i,z}$ be as in Definition~\ref{def:buhfunction}.
Take $R>1$, set $\dist_{t_i}=t_i^{-1}\dist$, $\meas_{t_i}=\meas(B_{t_i}(x))^{-1}\meas$ and write $H^{1,2}_{t_i}$ and $L^{2}_{t_i}$ for $H^{1, 2}(B_R^{\dist_{t_i}}(z), \dist_{t_i},\meas_{t_i})$ and $L^{2}(B_{R}^{\dist_{t_i}}(z),\meas_{t_i})$, respectively. 
Along with the existence of the limit $(|\nabla f|^*(x))^2$ of $ \fint_{B_r(z)}|\nabla f|^2\di\meas$ as $r\downarrow 0$, 
this provides, for $i$ large enough, a uniform control of the $H^{1, 2}_{t_i}$-norms of $f_{t_i, z}$ on $B_R^{\dist_{t_i}}(z)$;
\begin{align*}
\|f_{t_{i},z}\|_{H^{1,2}_{t_i}}^2 &= t_{i}^{-2} \|f - \fint_{B_{1}^{\dist_{t_i}}} f \di \meas\|_{L^{2}_{t_i}}^2 + \frac{\meas(B_{R}^{\dist_{t_i}}(z))}{\meas(B_{1}^{\dist_{t_i}}(z))} \fint_{B_{t_{i}R}^{\dist}(z)} |\nabla f|^{2} \di \meas \\
&\le C_0(K, N, R)\fint_{B_{t_{i}R}^{\dist}(z)} |\nabla f|^{2} \di \meas  +C_1(K, N, R)\fint_{B_{t_{i}R}^{\dist}(z)} |\nabla f|^{2} \di \meas \\
& \le C_2(K,N,R)((|\nabla f|^*(x))^2 + 1),
\end{align*}
where we used the Poincar\'e inequality. Thus, since $R>1$ is arbitrary, by 
Theorem~\ref{thm:compact loc sob} and a diagonal argument there exist a subsequence $(s_i)_i$ of $(t_i)_i$ 
and $\hat{f}\in H^{1,2}_{\mathrm{loc}}(Y,\dist_Y,\meas_Y)$ such that $f_{s_i, z}$ $H^{1, 2}_{\mathrm{loc}}$-weakly converge to $\hat{f}$.

Let us prove that $f_{s_i, z}$ is a $H^{1, 2}_{\mathrm{loc}}$-strong convergent sequence.
Let $R>0$ where \eqref{eq:countably} holds on $(Y, \dist_Y, \meas_Y, y)$ and let $h_{s_i, R}$ be the harmonic replacement of $f_{s_i, z}$ on $B_R^{\dist_{s_i}}(z)$.
Then applying Proposition~\ref{prop:harmconti} yields that $h_{s_i,R}$ $H^{1, 2}$-weakly converge to the harmonic replacement 
$h_R$ of $\hat{f}$ on $B_R(y)$. Since $h_{s_i,R}$ are harmonic, by 
Theorem~\ref{thm:stability lap}, $h_{s_i,r}$ $H^{1, 2}$-strongly converge to $h_R$ on $B_r(z)$ for any $r<R$.

Note that Proposition~\ref{prop: harmonic replacement} and the harmonicity of $h_{i, R}$ yield
\begin{align}\label{46}
\int_{B_R^{\dist_{s_i}}(z)} | \nabla (f_{s_i, z} - h_{s_i, R})|^2 \di\meas_{s_i} &= \int_{B_R^{\dist_{s_i}}(z)} | \nabla f_{s_i, z}|^2 \di\meas_{s_i}  - \int_{B_R^{d_{s_i}}(z)} | \nabla h_{s_i,R}|^2 \di\meas_{s_i} \nonumber ={\rm Dev}(f,B_{Rs_i}(z)).
\end{align}
Thus, since by our choice of $z$, ${\rm Dev}(f,B_{Rs_i}(z))$ goes to $0$ as $i \to \infty$, the 
Poincar\'e inequality gives $\|f_{s_i,z}-h_{s_i,R}\|_{L^2(B_R^{\dist_{s_i}}(z))}\to 0$, hence $f_{s_i,z}$ 
$H^{1,2}$-weakly converge to $h_R$ on $B_R(y)$, so that $\hat{f}=h_R$ on $B_R(y)$.
In addition, the $H^{1,2}$-strong convergence on balls $B_r(z)$, for all $r<R$,
of the functions $h_{s_i,R}$ shows that  $f_{s_i, z}$ $H^{1, 2}$-strongly converge to $\hat{f}$ on $B_r(z)$ for all $r<R$.
Since $R$ has been chosen subject to the only condition \eqref{eq:countably}, which holds with at most
countably many exceptions,  we see that $\hat{f} \in \mathrm{Harm}(Y, \dist_Y, \meas_Y)$ 
and that $f_{s_i, z}$ $H^{1, 2}_{\mathrm{loc}}$-strongly converge to $\hat{f}$.

Finally, let us show that $\hat{f}$ has a Lipschitz representative. It is easy to check that the condition $z\in\Leb_2(|\nabla f|)$, namely
$$
\lim_{r\downarrow 0}\fint_{B_r(z)}||\nabla f|-|\nabla f|^*(z)|^2\di\meas=0
$$
with the $H^{1,2}_{\mathrm{loc}}$-strong convergence of $f_{s_i, z}$ yield 
$|\nabla \hat{f}|(w)= |\nabla f|^*(z)$ for $\meas_Y$-a.e. $w \in Y$. Thus the Sobolev-to-Lipschitz property shows that $\hat{f}$ has
a Lipschitz representative.

\noindent
\textit{Step 2: the general case when $V\in L^2(T(X, \dist, \meas))$.} Let $A_j$, $M$, $k$, $F_i$ be given by Theorem~\ref{thm:Cheeger}. It is sufficient to prove the existence of $f$ as in Definition~\ref{def:buhderivation} for
$\meas$-a.e. $x\in A_j$. Since $\int_{B_r(x)\setminus A_j}|V|^2\di\meas=o(\meas(B_r(x)))$ for $\meas$-a.e.
$x\in A_j$, we can assume with no loss of generality, possibly replacing $V$ by $1_{X\setminus A_j}V$, that
$V=0$ on $X\setminus A_j$. As illustrated in \cite[Cor.~2.5.2]{Gigli} (by approximation of the $\chi_i$ by
simple functions) the expansion \eqref{eq:Cheexp} gives also
$$
1_{A_j}\left(\nabla f-\sum_{i=1}^k\alpha_i\nabla F_i\right)=0
$$
for all $f\in \Lip(X, \dist)\cap H^{1,2}(X,\dist,\meas)$, with $\sum_i\alpha_i^2\leq M|\nabla f|$ $\meas$-a.e. on $A_j$. 
By the approximation in Lusin's sense of Sobolev by 
Lipschitz functions and the locality of the pointwise norm, the same is true for Sobolev functions $f$.
Eventually, by linearity and density of gradients, we obtain the representation
 $$
V=\sum_{i=1}^k\alpha_i\nabla F_i
$$
for suitable coefficients $\alpha_i\in L^2(X,\meas)$, null on $X\setminus A_j$. It is now easily seen that
if $x$ is an harmonic point of all $F_i$ and a 2-Lebesgue point of all $\alpha_i$, then $x\in H(V)$ with
$$
f(y):=\sum_{i=1}^k\alpha_i^*(x)F_i(y).
$$

\end{proof}

\subsection{The behavior of $t\meas (B_{\sqrt{t}}(x))g_t$ as $t \downarrow 0$}

The main purpose of Subsections 5.2 and 5.3 is to prove Theorem~\ref{th:convergence1}, i.e.~the 
$L^2$-strong convergence of the metrics
\begin{equation}\label{def:hatgt}
\hat{g}_t:=t \meas (B_{\sqrt{t}}(\cdot))g_t \xrightarrow{t \downarrow 0} \hat{g},
\end{equation}
where $\hat{g}$ is the normalized Riemannian metric on $(X, \dist, \meas)$ defined by $c_n g$,
where $n=\dim_{\dist,\meas}(X)$ and the dimensional constant $c_n$ is given by
\begin{equation}\label{eq:defck}
c_n:= 
\frac{\omega_n}{(4\pi)^n}
 \int_{\setR^n}\bigl|\partial_{x_1}\bigl(e^{-|x|^2/4}\bigr)\bigr|^2\di x.
\end{equation}

Here is an important proposition whose proof contains the main technical ingredients that shall be used in the sequel.

\begin{proposition}\label{prop:blowupvector}
Let $V \in L^2(T(X, \dist, \meas))$ and $y \in \mathcal{R}_n \cap H(V)$. Then 
\begin{equation}\label{eq:blowup2}
\lim_{t \downarrow 0}\int_X t\meas (B_{\sqrt{t}}(x)) |\langle \nabla_x p(x, y, t), V(x)\rangle|^2 \di\meas (x) =
c_n |V|^{2*}(y).
\end{equation}
\end{proposition}
\begin{proof} As $y \in H(V)$, there exists $f \in H^{1,2}(X,\dist,\meas)$ such that $y\in H(f)$ and
$\fint_{B_r(x)}|V-\nabla f|^2\di\meas\to 0$ as $r\downarrow 0$. 
With $W=V-\nabla f$, let us first prove that
\begin{equation}\label{eq:10001}
\lim_{t \downarrow 0}\int_X t\meas (B_{\sqrt{t}}(x)) |\langle \nabla_x p(x, y, t), W(x)\rangle|^2 \di\meas (x) =0.
\end{equation}
Using the heat kernel estimate \eqref{eq:equi lip} with $\epsilon=1$ we need to estimate, for $0<t<C_4^{-1}$, 
$$
\int_X \frac{1}{\meas(B_{\sqrt{t}}(x))}\exp\left(-\frac{2\dist^2(x,y)}{5t}\right)|W(x)|^2\di\meas(x)
$$
and use \eqref{doubspec} to reduce the proof to the estimate of
$$
\frac{1}{\meas(B_{\sqrt{t}}(y))}\int_X\exp\left(-\frac{2\dist^2(x,y)}{5t}+c_1\frac{\dist(x,y)}{\sqrt{t}}\right)|W(x)|^2\di\meas(x).
$$
Using the identity $\int f(\dist(\cdot,y))\di \mu=-\int_0^\infty\mu(B_r(y))f'(r)\di r$ with 
$\mu_y:=\exp(c_1\dist(\cdot,y)/\sqrt{t})|W|^2\meas$ and $f_y(r)=\exp(-2r^2/(5t))$, we need to estimate
$$
- \frac{1}{\meas(B_{\sqrt{t}}(y))}\int_0^\infty\mu_y(B_r(x))f_y'(r)\di r.
$$
Now, write $\mu_y(B_r(y))\leq\omega(r)\exp(c_2r/\sqrt{t})\meas(B_r(y))$ with $\omega$ bounded and infinitesimal as $r\downarrow 0$ and
use the change of variables $r=s\sqrt{t}$ to see that it suffices to estimate
$$
\frac 45\int_0^\infty \left( \omega(s\sqrt{t})\frac{\meas(B_{s\sqrt{t}}(y))}{\meas(B_{\sqrt{t}}(y))}\right)\exp\left(c_1 s-\frac{2s^2}{5}\right) s\di s.
$$
Now we can split outer the integration in $(0,1)$ and in $(1,\infty)$; the former obviously gives an infinitesimal contribution as
$t\downarrow 0$; the latter can be estimated with the exponential growth condition \eqref{eq:BishopGromov} 
on $\meas(B_r(y))$ and gives an infinitesimal contribution as well.
This proves \eqref{eq:10001}. 

Now, setting $c_n(L)=\omega_n/(4\pi)^n\int_{B_L(0)}\bigl|\partial_{x_1}\bigl(e^{-|x|^2/4}\bigr)\bigr|^2\di x\uparrow c_n$ as $L\uparrow\infty$, 
we shall first prove that
\begin{equation}\label{eq:blowup222}
\lim_{t \downarrow 0}\int_{B_{L\sqrt{t}}(y)}t\meas (B_{\sqrt{t}}(x)) |\langle \nabla_x p(x, y, t), V(x)\rangle|^2 \di\meas (x) =
c_n(L) |V|^{2*}(y)
\end{equation}
for any $L<\infty$.
Taking \eqref{eq:10001} into account, it suffices to prove that
\begin{equation}\label{eq:1000}
\lim\limits_{t \downarrow 0} \int_{B_{L\sqrt{t}}(y)} t \meas(B_{\sqrt{t}}(x))|\langle \nabla_{x}p(x,y,t),\nabla f (x) \rangle|^2 \di \meas(x) = c_n (L) 
(|\nabla f|^*)^2(y)\quad\forall L\in [0,\infty).
\end{equation}
In order to prove \eqref{eq:1000}, for $t>0$ let us consider the rescaling $\dist \mapsto\dist_t:=\sqrt{t}^{-1}\dist$, $\meas\mapsto\meas_t:=\meas (B_{\sqrt{t}}(y))^{-1}\meas$. We denote by $p_t$ the heat kernel on the rescaled space $(X,\dist_t,\meas_t)$. 
Applying (\ref{eq:1001}) with $a:=\sqrt{t}^{-1}$, $b:=\frac{1}{\meas (B_{\sqrt{t}}(y))}$ and $s:=t$ yields
(notice that the factor $t=a^{-2}$ disappears by the scaling term in the definition of $f_{\sqrt{t},y}$ and the scaling of gradients)
\begin{align}\label{eq:2}
&\int_{B_{L\sqrt{t}}(y)}t\meas (B_{\sqrt{t}}(x))|\langle \nabla_x p(x, y, t), \nabla f(x)\rangle|^2 \di\meas (x) \nonumber\\
&=\int_{B_L^{\dist_t}(y)}\meas_t (B_1^{\dist_t}(x))|\langle \nabla_x p_t(x, y, 1),\nabla f_{\sqrt{t}, y}(x)\rangle|^2 \di\meas_t(x).
\end{align}
Take a sequence $t_i\to 0^+$, let $(s_i)_i$ be a subsequence of $(t_i)_i$ and $\hat{f}$ be a Lipschitz and harmonic function on $\mathbb{R}^n$ as in Definition~\ref{def:buhfunction} (i.e. $\hat{f}$ is the limit of $f_{\sqrt{s_i},y}$). Note that $\hat{f}$ has necessarily linear growth. Since linear growth harmonic functions on Euclidean spaces are actually linear or constant functions, we see that $\nabla \hat{f}= \sum_j a_j\frac{\partial}{\partial x_j}$ for some $a_j \in \mathbb{R}$. Then, by Theorem \ref{thm:local conv heat kernel}, letting $i\to\infty$ in the right hand side of (\ref{eq:2}) shows
\begin{align}\label{eq:3}
&\lim_{i\to\infty}\int_{B_L^{\dist_{s_i}}(y)}\meas_{s_i} (B_1^{\dist_{s_i}}(x))
|\langle \nabla_x p_{s_i}(x, y, 1), \nabla f_{\sqrt{s_i},y}(x)\rangle|^2 \di\meas_{s_i}(x) \nonumber \\
&=\int_{B_L(0_n)}\hat{\mathcal{H}}^n (B_1(x))|\langle\nabla_x q_n(x, 0_n, 1),\nabla\hat{f}(x)\rangle|^2 \di \hat{\mathcal{H}}^n(x),
\end{align}
where $\hat{\mathcal{H}}^n=\mathcal{H}^n/\omega_n$ (hence $\hat{\mathcal{H}}^n (B_1(x)) \equiv 1$) and 
$q_n$ denotes the heat kernel on $(\mathbb{R}^n, \dist_{\mathbb{R}^n}, \hat{\mathcal{H}}^n)$. Since
 \eqref{eq:pEuclidean} and \eqref{eq:1001} give
\begin{align*}
q_n(x, 0_n, 1) \equiv \frac{\omega_n}{\sqrt{4\pi}^{n}}e^{-|x|^2/4},
\end{align*}
a simple computation shows that the right hand side of (\ref{eq:3}) is equal to $c_n(L)(\sum_j|a_j|^2)$.
Finally, from
\begin{align}
(|\nabla f|^*(z))^2=\lim_{r \downarrow 0}\left( \frac{1}{\meas (B_r(y))}\int_{B_r(y)}|\nabla f|^2\di\meas \right)
&=\lim_{i\to\infty}\int_{B_1^{\dist_{s_i}}(y)}|\nabla f_{\sqrt{s_i}, y}|^2\di\meas_{s_i}\nonumber \\
&=\int_{B_1(0_n)}|\nabla \hat{f}|^2\di \hat{\mathcal{H}}^n= \sum_j|a_j|^2,
\end{align}
we have (\ref{eq:blowup222}) because $(t_i)_i$ is arbitrary.

In order to obtain (\ref{eq:blowup2}) it is sufficient to let $L\to\infty$ in (\ref{eq:blowup222}), taking into account that
$c_n(L)\uparrow c_n$ as $L\uparrow\infty$ and that, arguing as for \eqref{eq:10001}, one can prove that
$$
\lim_{L\to\infty}\sup_{0<t<C_4^{-1}}\int_{X\setminus B_{L\sqrt{t}}(y)}
|\langle \nabla_xp(x,y,t),W(y)\rangle|^2\di\meas(x)=0.
$$ 
\end{proof}

\begin{corollary}\label{cor:asymptotic}
Let $A$ be a Borel subset of $X$. Then for any $V \in L^2(T(X, \dist, \meas))$ and $y \in H(V) \cap \mathcal{R}_n$,
one has
\begin{enumerate}
\item[(1)] if $\int_{B_r(y)\cap A}|V|^2\di\meas=o(\meas(B_r(y)))$ as $r\downarrow 0$, we have
\begin{equation}\label{eq:109}
\lim_{t \downarrow 0}\int_A t\meas (B_{\sqrt{t}}(x)) |\langle \nabla_x p(x, y, t), V(x)\rangle|^2 \di\meas (x) = 0;
\end{equation}
\item[(2)] if $\int_{B_r(y)\setminus A}|V|^2\di\meas=o(\meas(B_r(y)))$ as $r\downarrow 0$, we have
\begin{equation}\label{eq:110}
\lim_{t \downarrow 0}\int_A t\meas (B_{\sqrt{t}}(x)) |\langle \nabla_x p(x, y, t), V(x)\rangle|^2 \di\meas (x)=c_n |V|^{2*}(y).
\end{equation}
\end{enumerate}
In particular, if $V \in L^p(T(X, \dist, \meas))$ for some $p>2$, \eqref{eq:109} holds if $A$ has density $0$ at $y$, and
\eqref{eq:110} holds if $A$ has density $1$ at $y$.
\end{corollary}
\begin{proof} (1) Let $W=1_A V$ and notice that our assumption gives that $y\in H(W)$, with $f\equiv 0$, so that
$|W|^{2*}(y)=0$. Therefore \eqref{eq:109} follows by applying Proposition~\ref{prop:blowupvector} to $W$. 
The proof of \eqref{eq:110} is analogous.
\end{proof}

\begin{remark}\label{eq:firstblue}
Thanks to the estimate \eqref{doubspec}, a similar argument provides also the following results
for all $y\in H(V)\cap\mathcal R_n$:
\begin{enumerate}
\item[(1)]if $\int_{B_r(y)\cap A}|V|^2\di\meas=o(\meas(B_r(y)))$ as $r\downarrow 0$, we have
\begin{equation}\label{eq:21}
\lim_{t \downarrow 0}\int_At\meas (B_{\sqrt{t}}(y)) |\langle \nabla_x p(x, y, t), V(x)\rangle|^2 \di\meas (x)=0;
\end{equation}
\item[(2)]if $\int_{B_r(y)\setminus A}|V|^2\di\meas=o(\meas(B_r(y)))$ as $r\downarrow 0$, we have
\begin{equation}\label{eq:20}
\lim_{t \downarrow 0}\int_A t\meas (B_{\sqrt{t}}(y)) |\langle \nabla_x p(x, y, t), V(x)\rangle|^2 \di\meas (x) = c_n |V|^{2*}(y).
\end{equation}
\end{enumerate}
\end{remark}

\begin{theorem}\label{thm:asymptotic1}
Let $V \in L^2(T(X, \dist, \meas))$.
Then for any Borel subsets $A_1,\, A_2$ of $X$ we have
\begin{equation}\label{eq:201}
\lim_{t \downarrow 0}\int_{A_1} \left(\int_{A_2}t \meas (B_{\sqrt{t}}(x))|\langle \nabla_xp(x, y, t), V(x)\rangle|^2 \di\meas (x)\right) 
\di\meas (y)=\int_{A_1\cap A_2} \hat{g}(V,V)\di\meas.
\end{equation}
\end{theorem}

\begin{proof} Taking the uniform $L^{\infty}$ estimate \eqref{eq:riem est} into account, it is enough to prove the result for
$V \in L^\infty(T(X, \dist, \meas))$, since this space is dense in $L^2(T(X,\dist,\meas))$. 
Take $y \in X$. By \eqref{eq:1105}, for $0<t<C_4^{-1}$, we get
\begin{equation}\label{eq:11055}
\int_Xt \meas (B_{\sqrt{t}}(x))|\langle \nabla_xp(x, y, t),V(x)\rangle|^2 \di\meas(x) \le 
\int_X\frac{C_3^2e^2\|V\|_{L^{\infty}}^2}{\meas(B_{\sqrt{t}}(x))}\exp \left( \frac{-2\dist (x, y)^2}{5t}\right)\di\meas(x)
\end{equation}
and, by applying (\ref{lem:bound}) to the rescaled space $(X, \sqrt{t}^{-1}\dist, \meas (B_{\sqrt{t}}(x)))^{-1}\meas)$, we obtain that the right hand side in \eqref{eq:11055} is uniformly bounded as function of $y$. 

Thus, denoting by $A_2^*$ the set of points of density 1 of $A_2$ and by
$A_2^{**}$ the set of points of density 0 of $A_2$ (so that $\meas(X\setminus (A_2^*\cup A_2^{**}))=0$), 
the dominated convergence theorem, Corollary~\ref{cor:asymptotic} and the definition of $\hat{g}$ imply
\begin{align}
&\int_{A_1} \left(\int_{A_2}t\meas (B_{\sqrt{t}}(x))|\langle \nabla_xp(x, y, t), V(x)\rangle|^2 \di\meas (x)\right) \di\meas (y) \nonumber \\
&=\int_{\mathcal{R}_n \cap {A_1} \cap A_2^*} \left(\int_{A_2}t\meas (B_{\sqrt{t}}(x))|\langle\nabla_xp(x, y, t), V(x)\rangle|^2 \di\meas (x)\right) \di\meas (y) \nonumber \\
&+\int_{\mathcal{R}_n \cap A_1 \cap A^{**}_2} \left(\int_{A_2}t\meas (B_{\sqrt{t}}(x))|\langle\nabla_xp(x, y, t), V(x)\rangle|^2 \di\meas (x)\right) \di\meas (y) \nonumber \\
&\to \int_{\mathcal{R}_n \cap A_1 \cap A_2^* }c_n|V|^{2*}(y)\di\meas(y)= \int_{A_1 \cap A_2}\hat{g}(V,V)\di\meas.
\end{align}
\end{proof}

\begin{remark} Building on Remark~\ref{eq:firstblue},
one can prove by a similar argument
\begin{equation}\label{eq:202}
\lim_{t \downarrow 0}\int_{A_1} \left(\int_{A_2}t\meas (B_{\sqrt{t}}(y))|\langle \nabla_xp(x, y, t), V(x)\rangle|^2
 \di\meas (x)\right) \di\meas (y)=\int_{A_1\cap A_2} \hat{g}(V,V)\di\meas.
\end{equation}
\end{remark}

In order to improve the convergence of $\hat{g}_t$ from weak to strong, a classical Hilbertian
strategy is to prove convergence of the Hilbert norms. In our case, at the level of ${\bf\hat{g}_t}$ 
(and taking \eqref{eq:righthsk} and \eqref{eq:normaPhi_t_other} into account), this translates into
\begin{equation}\label{eq:blowuphs2}
\limsup_{t\downarrow 0}\int_X
\left(t\meas (B_{\sqrt{t}}(x))\right)^2\biggl|\int_X \dist_x p(x, y, t)\otimes \dist_x p(x,y,t) \di \meas (y)\biggr|_{HS}^2
\di\meas(x)
\leq n c_n^2 \meas(X).
\end{equation}
The proof of this estimate requires a more delicate blow-up procedure, and to its proof we devoted the next subsection.
Notice that, by using the (non-sharp) estimate of the left hand side in \eqref{eq:blowuphs2} 
with $\int_X\bigl[t\meas (B_{\sqrt{t}}(\cdot))\int_X |\nabla_x p|^2 \di\meas\bigr]^2\di\meas$ one obtains 
$n^2 c^2_n\meas(X)$, but this upper bound is 
not sufficient to obtain the convergence of the Hilbert-Schmidt norms.

We are now in a position to prove the main theorem of this subsection.

\begin{theorem}\label{th:convergence1}
The family of Riemannian metrics $\hat{g}_t$ in \eqref{def:hatgt}
$L^2$-strongly converges to $\hat{g}$ as $t\downarrow 0$
according to Definition~\ref{defcong}. In particular one has $L^1$-strong convergence
of $\hat{g}_t(V,V)$ to $\hat{g}(V,V)$ as $t\downarrow 0$ for all $V\in L^2(T(X,\dist,\meas))$.
\end{theorem}
\begin{proof}
For all $V\in L^2(T(X,\dist,\meas))$, the $L^1$-weak convergence of $\hat{g}_t(V,V)$ to $\hat{g}(V,V)$ follows easily from 
Theorem~\ref{thm:asymptotic1}:  indeed, choosing $A_1=X$, we obtain that $\int_{A_2}\hat{g}_t(V,V)\di\meas$ converge as
$t\downarrow 0$ to $\int_{A_2}\hat{g}(V,V)\di\meas$ for any Borel set $A_2\subset X$. The Vitali-Hahn-Saks theorem
then grants convergence in the weak topology of $L^1$.

By combining \eqref{eq:normaPhi_t_other}, \eqref{eq:blowuphs2} and \eqref{eq:righthsk} we have
\begin{eqnarray}\label{eq:noenergyloss}
&&\limsup_{t\downarrow 0}\int_X|{\bf\hat{g}}_t|^2_{HS}\di\meas\nonumber\\ &=&
\limsup_{t\downarrow 0}\int_X \bigl(t\meas (B_{\sqrt{t}}(x))\bigr)^2\biggl\vert\int_X\dist_x p(x, y, t)\otimes
\dist_x p(x,y,t)\di\meas (y)\biggr\vert_{HS}^2 \di\meas (x)\nonumber \\
&=& n c_n^2\meas (\mathcal{R}_n)=\int_X |{\bf\hat{g}}|^2_{HS}\di\meas.
\end{eqnarray}
The $L^2$-strong convergence now comes from Proposition~\ref{prop:convcrigi}.
\end{proof}
\begin{remark}\label{counterex}
With a uniform $L^\infty$-bound (\ref{eq:riem est}), the $L^2$-strong convergence $\hat{g}_t \to \hat{g}$ implies $L^p$-strong convergence for all $1 \le p<\infty$. However, this result cannot be improved to $L^\infty$-strong convergence because of the following example (note that higher dimensional analogous examples can be obtained by taking cartesian products):

Let us consider a $\RCD^*(0,1)$ space $(X, \dist, \meas):=([0,\pi],\dist,\mathcal{H}^1/\pi)$ where $\dist$ is the Euclidean distance. In this context, the canonical Riemannian metric $g$ is ${\bf g}=\di s \otimes \di s$, the operator $\Delta$ is the Laplacian with Neumann boundary condition, the corresponding orthonormal basis of $L^2([0,\pi],\mathcal{H}^1/\pi)$ made of eigenfunctions $(\phi_i)_{i\ge0}$ is given by $\phi_0\equiv 1$ and $\phi_i(s) = \sqrt{2}\cos(is) (i \ge 1)$, and by \eqref{eq:explicit expression} for any $t>0$ the pull-back metric $g_t$ is
\begin{equation}\label{eq:expressionexample}
{\bf g}_t = 2 \sum_{i \ge 1} e^{-2 i^2 t} i^2 \sin^2(is)\dist s \otimes \dist s.
\end{equation}
Since $|{\bf g}_t|_{HS}$ is infinitesimal around $s=0, \,\pi$, we have $\||\hat{\bf g}_t - \hat{\bf g}|_{HS}\|_{L^\infty} \ge c_1$ for any $t>0$.
Moreover, $\Phi_t$ is not biLipschitz, that is, $(\Phi_t)^{-1}: \Phi_t(X) \to X$ is not Lipschitz because if it were Lipschitz, then by an argument similar to the proof of Proposition~\ref{prop:smooth embedding}, there would be $\hat{c}_t>0$ such that $g_t \ge \hat{c}_tg$, which contradicts that $|{\bf g}_t|_{HS}$ is infinitesimal around $s=0,\,\pi$. Similarly $\Phi_t^\ell$ is not biLipschitz for all $t,\,\ell$ (recall Proposition~\ref{projj} for the definition of $\Phi_t^\ell$).
\end{remark}

\subsection{Proof of \eqref{eq:blowuphs2}}\label{proof}
We set
$$
F(x, t):=\left(t\meas (B_{\sqrt{t}}(x))\right)^2\biggl|\int_X \dist_x p(x, y, t)\otimes \dist_x p(x,y,t) \di\meas (y)\biggr|_{HS}^2
$$
and \eqref{eq:riem est} provides a uniform upper bound on the $L^\infty$ norm
of $F(\cdot,t)$, for $0<t \leq1$.
Now, we claim that \eqref{eq:blowuphs2} follows by Proposition~\ref{rem:blowuptensor} below; indeed, by integration of both
sides we get
$$
\lim_{t\downarrow 0}\int_X \frac{1}{\meas (B_{\sqrt{t}}(\bar x))}\int_{B_{\sqrt{t}}(\bar x)} F(x,t)\di\meas(x)\di\meas(\bar x)=nc_n^2\meas(X)
$$
and, thanks to Fubini's theorem, the left hand side can be represented as
$$
\lim_{t\downarrow 0}\int_X F(x,t)\biggl(\int_{{B_{\sqrt{t}}(x)}}\frac{1}{\meas (B_{\sqrt{t}}(\bar x))}\di\meas(\bar x)\biggr)
\di\meas(x)
$$
Since it is easily seen that $\int_{{B_{\sqrt{t}}(x)}}\frac{1}{\meas (B_{\sqrt{t}}(\bar x))}\di\meas(\bar x)$ are uniformly
bounded and converge to $1$ as $t\downarrow 0$ for all $x \in \mathcal{R}_n$
(in particular for $\meas$-a.e. $x$), we have
\begin{align*}
&\left| \int_XF(x,t)\biggl(\int_{{B_{\sqrt{t}}(x)}}\frac{1}{\meas (B_{\sqrt{t}}(\bar x))}\di\meas(\bar x)\biggr)\di \meas (x)-\int_XF(x, t)\di \meas (x) \right| \\
&\le C\int_X\left| 1-\int_{{B_{\sqrt{t}}(x)}}\frac{1}{\meas (B_{\sqrt{t}}(\bar x))}\di\meas(\bar x)\right|\di \meas(x) \to 0,
\end{align*}
where we used the dominated convergence theorem.
Thus 
$$
\lim_{t \downarrow 0}\int_XF(x, t)\di \meas (x)=\lim_{t\downarrow 0}\int_X \frac{1}{\meas (B_{\sqrt{t}}(\bar x))}\int_{B_{\sqrt{t}}(\bar x)} F(x,t)\di\meas(x)\di\meas(\bar x)=nc_n^2\meas(X)
$$
which proves \eqref{eq:blowuphs2}.

Hence, we devote the rest of the subsection to the proof of the proposition.

\begin{proposition}\label{rem:blowuptensor}
For all $\bar x \in\mathcal{R}_n$ one has
\begin{equation}\label{eq:blowuphs}
\lim_{t\downarrow 0} \frac{1}{\meas (B_{\sqrt{t}}(\bar x))}\int_{B_{\sqrt{t}}(\bar x)} F(x,t)\di\meas(x)=n c_n^2,
\end{equation}
with $c_n$ defined as in \eqref{eq:defck}.
\end{proposition}
\begin{proof} Let us fix $t_j \to 0^+$ and consider the mGH convergent sequence 
\begin{equation}\label{eq:scal}
(X, \dist_j, \meas_j, \bar x):=\left(X, \sqrt{t_j}^{-1}\dist, \meas (B_{\sqrt{t_j}}(\bar x))^{-1}\meas, \bar x\right) 
\stackrel{mGH}{\to} \left(\mathbb{R}^n, \dist_{\mathbb{R}^n}, \hat{\mathcal{H}}^n, 0_n\right),
\end{equation}
where $\hat{\mathcal{H}}^n:=\mathcal{H}^n/\omega_n$. 

Setting (note that the center in the first factor is $\bar x$, unlike $F(x,t)$)
$$
\bar F(x,t):=\bigl(t\meas (B_{\sqrt{t}}(\bar x))\bigr)^2\biggl|\int_X \dist_x p(x, y, t)\otimes \dist_x p(x,y,t) \di\meas (y)\biggr|_{HS}^2,
$$
we claim that, in order to get \eqref{eq:blowuphs}, it is sufficient to prove that
\begin{equation}\label{eq:blowuphs_modified}
\lim_{j\to\infty} \frac{1}{\meas (B_{\sqrt{t_j}}(\bar x))}\int_{B_{\sqrt{t_j}}(\bar x)} F(x,t_j)\di\meas(x)=n c_n^2.
\end{equation}
Indeed, letting
\begin{equation}
H_j(x):=\left|\int_X\dist_xp(x, y, t_j) \otimes \dist_xp(x, y, t_j)\di\meas (y)\right|_{HS}^2,
\end{equation}
so that $\bar F(x,t_j)=\bigl(t_j\meas (B_{\sqrt{t_j}}(\bar x))\bigr)^2 H_j(x)$, one has
\begin{align}
&\frac{1}{\meas (B_{\sqrt{t_j}}(\bar x))}\int_{B_{\sqrt{t_j}}(\bar x)}\left|\bigl( t_j\meas (B_{\sqrt{t_j}}(\bar x))\bigr)^2H_j(x)-
\bigl(t_j\meas (B_{\sqrt{t_j}}(x))\bigr)^2H_j(x)\right|\di \meas (x) \nonumber\\
&= \int_{B_1^{\dist_j}(\bar x)}\left| 1- \bigl(\meas_j(B_1^{\dist_j}(x))\bigr)^2\right| \left|\int_X\dist_xp_j(x, y, 1) \otimes \dist_xp_j(x, y, 1)\di \meas_j (y)\right|_{HS}^2\di \meas_j(x) \nonumber\\
&\le C \int_{B_1^{\dist_j}(\bar x)}\left| 1- \bigl(\meas_j(B_1^{\dist_j}(x))\bigr)^2\right| \di \meas_j(x)
\to C\int_{B_1(0_n)}\left| 1-\bigl(\hat{\mathcal{H}}^n(B_1(x))\bigr)^2\right|\di\hat{\mathcal{H}}^n(x)=0,\nonumber
\end{align}
where $C$ comes from the Gaussian estimate \eqref{eq:equi lip} 
and we used the uniform convergence of $\meas_j(B_1^{\dist_j}(x))$ to $\hat{\mathcal{H}}^n(B_1(x))$.

Applying Proposition~\ref{prop:harmconti} with the good cut-off functions constructed in \cite{MondinoNaber}
for the standard coordinate functions 
$h_i: \mathbb{R}^n \to \mathbb{R}$ yields that (possibly extracting a subsequence) the existence of Lipschitz functions
$h_{i,j}\in D(\Delta^j)$, harmonic in $B_3^{\dist_j}(\bar x)$, such that $h_{i, j}$ $H^{1, 2}$-strongly converge to $h_i$ on $B_3(0_n)$ 
with respect to the convergence \eqref{eq:scal}. Here and in the sequel we are denoting $\Delta^j$ the Laplacian 
of $(X,\dist_j,\meas_j)$. Note that gradient estimates for solutions of Poisson's equations given in \cite{Jiang} show 
\begin{equation}\label{eq:ikk}
C:=\sup_{i, j}\||\nabla h_{i, j}|_j\|_{L^{\infty}(B_2^{\dist_j}(\bar x))}<\infty.
\end{equation}

On the other hand Bochner's inequality (we use here and in the sequel the notation $\mathrm{Hess}^j$ for the
Hessian in the rescaled space. Recall that in Subsection~\ref{se:rcdstarkn}) it is shown that
\begin{equation}
\frac{1}{2}\int_X \Delta^j \phi |\nabla h_{i, j}|_j^2\di \meas_j \ge \int_X \phi \left(  |\mathrm{Hess}^j_{h_{i, j}}|^2 +
t_jK|\nabla h_{i, j}|_j^2\right) \di \meas_j
\end{equation}
for all $\phi \in D(\Delta^j)$ with $\Delta^j \phi \in L^\infty(X, \meas_j)$ and $\supp \phi \subset B_3^{\dist_j}(\bar x)$.
In particular, taking as $\phi=\phi_j$ the good cut-off functions constructed in \cite{MondinoNaber} we obtain
\begin{equation}\label{eq:hesstozero}
\lim_{j \to \infty}\int_{B_2^{\dist_j}(\bar x)}|\mathrm{Hess}^j_{h_{i, j}}|^2\di \meas_j =0.
\end{equation}
Let us define functions $a_j^{\ell, m}:B_2^{\dist_j}(\bar x) \to \mathbb{R}$, $a^{\ell, m}:B_2(0_n) \to \mathbb{R}$ by
$$
a_j^{\ell, m}(x):=\int_X\langle \nabla_x p_j(x, y, 1), \nabla h_{\ell, j}(x)\rangle_j 
\langle \nabla_x p_j(x, y, 1), \nabla h_{m, j}(x)\rangle_j \di \meas_j(y), 
$$
$$
a^{\ell, m}(x):=\int_{\mathbb{R}^n} \langle \nabla_x q_n(x, y, 1), \nabla h_\ell(x)\rangle \langle \nabla_x
q_n(x, y, 1), \nabla h_m(x)\rangle \di \hat{\mathcal{H}}^n(y), 
$$
respectively, where $p_j(x, y, t)$ is the heat kernel of $(X, \dist_j, \meas_j)$ and $q_n(x,y,t)$ is the heat kernel of $(\mathbb{R}^n,\dist_{\mathbb{R}^n}, \hat{\mathcal{H}}^n)$ 
(we also use the $\langle\cdot,\cdot\rangle_j$ notation
to emphasize the dependence of these objects on the rescaled metric). Notice that the explicit expression
\eqref{eq:1001} of $q_n(x,y,t)$ provides the identity $a^{\ell,m}=c_n\delta_{\ell,m}$.

Now let us prove that $a_j^{\ell, m}$ $L^p$-strongly converge to $a^{\ell, m}$ on $B_1(0_n)$ for all $p \in [1, \infty)$.
It is easy to check the uniform $L^\infty$ boundedness by the Gaussian estimate \eqref{eq:equi lip} and 
\eqref{eq:ikk}, and the $L^p$-weak convergence by 
Theorem~\ref{thm:local conv heat kernel}. To improve the convergence from weak to strong, 
thanks to the compactness result stated in Theorem~\ref{thm:last_compactness}, it suffices to prove that
$a_j^{\ell, m} \in H^{1,2}(B_2^{\dist_j}(\bar x), \dist_j, \meas_j)$ for all $j$, and that
\begin{equation}\label{eq:essen1}
\sup_j\int_{B_2^{\dist_j}(\bar x)}|\nabla a_j^{\ell, m}|_j\di\meas_j<\infty.
\end{equation} 
Thus, let us check that \eqref{eq:essen1} holds as follows.
For any $y\in X$, the Leibniz rule and \eqref{eq:Hessian2} give the following equality in $L^2(T^*(B_2^{\dist_j}({\bar x}), \dist_j, \meas_j))$:
\begin{eqnarray}\label{eq:hessexp}
&&\dist_x \left(\langle \nabla_x p_j(x, y, 1), \nabla h_{\ell, j}(x)\rangle_j\langle \nabla_xp_j (x, y, 1), \nabla h_{m, j}(x)\rangle_j\right)
\\&=&\langle \nabla_x p_j(x, y, 1), \nabla h_{\ell, j}(x)\rangle_j\bigl(\mathrm{Hess}^j_{p_j(\cdot,y,1)}(\nabla h_{m,j},\cdot)+
\mathrm{Hess}^j_{h_{m,j}}(\nabla_x p_j(x,y,1),\cdot)\bigr)\nonumber
\\&+& \langle \nabla_x p_j(x, y, 1), \nabla h_{m, j}(x)\rangle_j\bigl(\mathrm{Hess}^j_{p_j(\cdot,y,1)}(\nabla h_{\ell,j},\cdot)+
\mathrm{Hess}^j_{h_{\ell,j}}(\nabla_x p_j(x,y,1),\cdot)\bigr).\nonumber
\end{eqnarray}

Now, recalling that $(X,\dist_j,\meas_j)$ arises from the rescaling of a fixed compact space, the Gaussian estimate
\eqref{eq:equi lip} yields that 
$\langle\nabla_x p_j(x, y, 1), \nabla h_{\ell, j}(x)\rangle_j\langle \nabla_xp_j (x, y, 1), \nabla h_{m, j}(x)\rangle_j$
belong to $H^{1,2}(B_2^{\dist_j}(\bar x), \dist_j, \meas_j)$, with norm for $j$ fixed uniformly bounded w.r.t.~$y$.
Hence, we can commute differentiation w.r.t.~$x$ and integration w.r.t.~$y$ to obtain that
$a_j^{\ell, m} \in H^{1,2}(B_2^{\dist_j}(\bar x), \dist_j, \meas_j)$ with
\begin{equation}\label{eq:eeeqq}
\nabla a_j^{\ell, m}(x)=\int_X
\nabla_x \left(\langle \nabla_x p_j(x, y, 1), \nabla h_{\ell, j}(x)\rangle_j\langle \nabla_xp_j (x, y, 1), \nabla h_{m, j}(x)\rangle_j\right)
\di\meas_j(y)
\end{equation}
in $L^2(T(B_2^{\dist_j}(\bar x), \dist_j, \meas_j))$. According to the referee's suggestion, let us clarify (\ref{eq:eeeqq}) as follows. It easily follows from (\ref{eq:expansion2}) that $a_j^{l, m} \in H^{1,2}(B_2^{\dist_j}(\bar x), \dist_j, \meas_j)$. On the other hand for any $f_i \in D(\Delta_j, B_2^{\dist_j}(\bar x)) (i=1, 2)$ with $\supp f_i \subset B_2^{\dist_j}(\bar x)$, denoting $a_j^{l, m}(x)=\int_XH_j^{l, m}(x, y)\di \meas_j(y)$, we have
\begin{align}
&\int_{B_2^{\dist_j}(\bar x)}\langle \nabla a_j^{l, m}, f_1\nabla f_2\rangle \di \meas_j \nonumber \\
&=\int_{B_2^{\dist_j}(\bar x)}a_j^{l, m}(x)\left( -\langle \nabla f_1, \nabla f_2\rangle (x) -f_1(x)\Delta_j f_2(x)\right)\di \meas_j(x) \nonumber \\
&=\int_X\left(\int_{B_2^{\dist_j}(\bar x)} H_j^{l, m}(x, y) \cdot \left( -\langle \nabla f_1, \nabla f_2\rangle (x) -f_1(x)\Delta_j f_2(x)\right) \di \meas_j(x)\right)\di \meas_j(y) \nonumber \\
&=\int_{B_2^{\dist_j}(\bar x)}\left(\int_X\left \langle \nabla_x H_j^{l, m}(x, y), f_1(x)\nabla f_2(x)\right \rangle \di \meas_j(y)\right)\di \meas_j(x)
\end{align}
which proves (\ref{eq:eeeqq}) because the set $\{\sum_{\alpha =1}^{\beta}f_{1, \alpha}\nabla f_{2, \alpha}; \beta \in \mathbb{N}, f_{i, \alpha} \in D(\Delta_j, B_2^{\dist_j}(\bar x)), \supp f_{i, \alpha} \subset B_2^{\dist_j}(\bar x)\}$ is dense in $L^2(T(B_2^{\dist_j}(\bar x), \dist_j, \meas_j))$.  

 {From} \eqref{eq:hessexp} we then get
\begin{eqnarray*}
|\nabla a_j^{\ell, m}|_j&\leq &C\bigl(|\mathrm{Hess}^j_{p_j(\cdot,y,1)}(\nabla h_{\ell,j},\nabla h_{\ell,j})|+
|\mathrm{Hess}^j_{p_j(\cdot,y,1)}(\nabla h_{m,j},\nabla h_{m,j})|\bigr)|\nabla p_j(\cdot,y,1)|_j
\\&+&
C\bigl(|\mathrm{Hess}^j_{h_{\ell,j}}(\nabla p(\cdot,y,1),\nabla p(\cdot,y,1))|+
|\mathrm{Hess}^j_{h_{m,j}}(\nabla p(\cdot,y,1),\nabla p(\cdot,y,1))|\bigr)|\nabla p_j(\cdot,y,1)|_j
\end{eqnarray*}
where $C$ is the constant in \eqref{eq:ikk}, so that using \eqref{eq:ikk} once more we get
\begin{align}\label{eq:eqqqqq}
&\||\nabla a_j^{\ell, m}|_j\|_{L^1(B^{\dist_j}_2(\bar x))} \nonumber \\
&\le \tilde{C} \left( \int_X \int_{B^{\dist_j}_2(\bar x)} |\mathrm{Hess}^j_{p_j( \cdot, y, 1)}|^2\di \meas_j(x) \di \meas_j(y) \right)^{1/2} \left( 
\int_X\int_{B^{\dist_j}_2(\bar x)}|\nabla_x p_j(x, y, 1))|_j^2\di \meas_j(x) \di \meas_j(y)\right)^{1/2} \nonumber \\
& +\tilde{C}\int_X\int_{B_2^{\dist_j}(\bar x)}\left(|\mathrm{Hess}^j_{h_{\ell, j}}|(x)|\nabla_xp_j(x, y, 1)|_j^2
+|\mathrm{Hess}^j_{h_{m, j}}|(x) |\nabla_xp_j(x, y, 1)|_j^2\right)\di \meas_j(x)\di \meas_j(y)
\end{align}
for some positive constant $\tilde{C}$ (recall that the Hessian norm is the Hilbert-Schmidt norm).
Note that the second term of the right hand side of (\ref{eq:eqqqqq}) is uniformly bounded with respect to $j$ because of the Gaussian estimate \eqref{eq:equi lip}, \eqref{eq:hesstozero}
and Cavalieri's formula (see Lemma \ref{lem:volume}).
Note that \eqref{eq:equi lip} and \eqref{eq:lapheatbd} with Lemma \ref{lem:volume} show
$$
\sup_j\left( \int_X \int_{B^{\dist_j}_2(\bar x)} |\Delta^j_x p_j(x, y, 1)|^2\di \meas_j(x) \di \meas_j(y) + 
\int_X\int_{B^{\dist_j}_2(\bar x)}|\nabla_x p_j(x, y, 1))|_j^2\di \meas_j(x) \di \meas_j(y) \right) <\infty.
$$
In particular by applying \eqref{eq:Hessian1} to the scaled spaces, with a sequence of
good cut-off functions constructed in \cite{MondinoNaber}, we obtain
$$
\sup_j \int_X \int_{B^{\dist_j}_2(\bar x)} |\mathrm{Hess}^j_{p_j( \cdot, y, 1)}|^2(x) \di \meas_j(x) \di \meas_j(y) <\infty.
$$
Thus \eqref{eq:eqqqqq} yields (\ref{eq:essen1}), which completes the proof of the $L^p$-strong convergence of $a_j^{\ell, m}$ to 
$a^{\ell, m}$ for all $p\in [1,\infty)$.

Then, since $a^{\ell,m}=c_n\delta_{\ell m}$ we get
\begin{equation}\label{eq:sharpblowup}
\lim_{j \to \infty}\int_{B_1^{\dist_j}(\bar x)}\sum_{\ell, m} |a_j^{\ell, m}|^2 \di \meas_j 
=\int_{B_1(0_n)}\sum_{\ell, m} |a^{\ell, m}|^2 \di\hat{\mathcal{H}}^n=n c_n^2.
\end{equation}
Hence, to finish the proof of \eqref{eq:blowuphs_modified}, and then of the proposition, it suffices to check that
\begin{equation}\label{eq:asymptoticequal}
\fint_{B_{\sqrt{t_j}}(\bar x)}\left|
\sum_{\ell, m} |a_j^{\ell, m}|^2 - \bigl(t_j \meas (B_{\sqrt{t_j}}(\bar x))\bigr)^2 
\left| \int_X \dist_x p(x, y, t_j)\otimes \dist_x p(x,y,t_j) \di\meas (y) \right|^2_{HS} \right| \di\meas (x)
\end{equation}
is infinitesimal as $j \to \infty$. 

To prove this fact, we first state an elementary property of Hilbert spaces whose proof is quite standard, and therefore
omitted: for any $r$-dimensional Hilbert space $(V, \langle \cdot, \cdot \rangle)$, $\epsilon>0$, $\{e_i\}_{i=1}^r \subset V$ one has the 
implication
\begin{equation}\label{eq:almost}
|\langle e_i, e_j \rangle - \delta_{ij}|<\epsilon\quad\forall i,\,j
\quad\Rightarrow\quad
\left| |v|^2- \sum_{i=1}^r|\langle v, e_i\rangle|^2 \right|\le C(r) \epsilon^2 |v|^2\qquad\forall v\in V.
\end{equation}
Note that the scaling property \eqref{eq:1001} of the heat kernel gives
\begin{align}\label{eq:asymptoticequal1}
(\ref{eq:asymptoticequal})=\int_{B_{1}^{\dist_j}(\bar x)}\left|\sum_{\ell,m} |a_j^{\ell, m}|^2 - G_j \right|\di\meas_j,
\end{align}
where 
$$
G_j(x):=\left| \int_X \dist_x p_j(x, y, 1)\otimes \dist_x p_j(x,y,1) \di\meas_j (y) \right|^2_{HS}.
$$
Let 
$$
\epsilon_j:=\max_{\ell, m}\int_{B_1^{\dist_j}(\bar x)}\left| \langle \nabla h_{\ell, j}, \nabla h_{m, j}\rangle_j -\delta_{\ell m}\right|\di \meas_j.
$$
Then notice that for all $\ell,\,m$, one has
$$
\int_{B_1^{\dist_j}(\bar x)}\left| \langle \nabla h_{\ell, j}, \nabla h_{m, j}\rangle_j -\delta_{\ell m}\right|\di \meas_j \to 
\int_{B_1(0_n)}\left| \langle \nabla h_\ell, \nabla h_m\rangle -\delta_{\ell m}\right|\di\hat{\mathcal{H}}^n=0
$$
as $j \to \infty$. In particular $\epsilon_j \to 0$.

Let
$$
K_j^{\ell, m}:=\left\{w \in B_1^{\dist_j}(\bar x):\ |\langle \nabla h_{\ell, j}, \nabla h_{m, j}\rangle_j(w) -\delta_{lm}|>\sqrt{\epsilon_j}\right\}.
$$
Then the Markov inequality and the definition of $\epsilon_j$ give $\meas_j(K_j^{\ell, m}) \le \sqrt{\epsilon_j}$, so that
$K_j:= \bigcup_{\ell, m}K_j^{\ell, m}$ satisfy $\meas_j(K_j) \to 0$ as $j \to \infty$.

On the other hand, \eqref{eq:almost} with $r=n^2$ yields
\begin{equation}
\int_{B_1^{\dist_j}(\bar x)\setminus K_j}
\left|\sum_{\ell, m} |a_j^{\ell, m}|^2 - G_j \right|\di\meas_j \le C(n^2)\epsilon_j \int_{B_1^{\dist_j}(\bar x)}|G_j|^2\di \meas_j \to 0, 
\end{equation}
where we used $\sup_j\|G_j\|_{L^\infty(B^{\dist_j}_1(\bar x))}<\infty$, as a consequence of the Gaussian estimate
\eqref{eq:equi lip}.

Then since
$$
\int_{K_j}\left|\sum_{\ell, m} |a_j^{\ell, m}|^2 - G_j \right|\di\meas_j \le \sqrt{\meas_j(K_j)} 
\left( \int_{B_1^{\dist_j}(\bar x)}\left| \sum_{\ell, m} |a_j^{\ell, m}|^2 - G_j\right|^2\di\meas_j \right)^{1/2} \to 0, 
$$
where we used the uniform $L^{\infty}$-bounds on $a_j^{\ell, m}$, we have 
$$
\eqref{eq:asymptoticequal1} = \int_{K_j}\left|\sum_{\ell, m} |a_j^{\ell, m}|^2 - G_j \right|\di \meas_j + 
\int_{B_1^{\dist_j}(\bar x) \setminus K_j}\left|\sum_{\ell, m} |a_j^{\ell, m}|^2 - G_j \right|\di \meas_j  \to 0.
$$
Thus we have that the expression in \eqref{eq:asymptoticequal} is infinitesimal as $j\to\infty$, which completes the proof of 
Proposition~\ref{rem:blowuptensor}.
\end{proof}

\subsection{The behavior of $t^{(n+2)/2}g_t$ as $t \downarrow 0$}

Let us now consider the convergence result
\[
\tilde{g}_t := t^{(n+2)/2}g_t \to \tilde{g},
\]
where $n=\dim_{\dist,\meas}(X)$ and, with our notation $\meas=\theta\mathcal{H}^n$, where $\theta$ is the density of $\meas$ w.r.t. $\mathcal{H}^n$. The normalized metric $\tilde{g}$ is defined by
\[
\tilde{g} = \frac{c_n}{\omega_n\theta} 1_{\mathcal{R}_n^*} g.
\]
We shall need the following well-known lemma, already used in \cite{AmbrosioHondaTewodrose}, and whose simple proof is omitted here.

\begin{lemma}\label{lem:dct}
Let $f_i, \,g_i,\, f,\, g \in L^1(X, \meas)$.
Assume that $f_i\to f$ and  $g_i \to g$ $\meas$-a.e., that $|f_i|\le g_i$ $\meas$-a.e., 
and that $\lim_{i \to \infty}\|g_i\|_{L^1(X, \meas)}=\|g\|_{L^1(X, \meas)}$.
Then $f_i \to f$ in $L^1(X, \meas)$. 
\end{lemma}

Let us start with the analog of Theorem~\ref{thm:asymptotic1} in this setting.
\begin{theorem}\label{thm:asymptotic4}
Let $V \in L^\infty(T(X, \dist, \meas))$ and $A_1\subset\mathcal{R}^*_n$ Borel. If 
\begin{equation}\label{eq:25}
\lim_{r \downarrow 0} \int_{A_1}\frac{r^n}{\meas (B_r(y))}\di\meas(y) 
=\int_{A_1}\lim_{r \downarrow 0}\frac{r^n}{\meas (B_r(y))}\di\meas(y)<\infty,
\end{equation}
then for any Borel set $A_2\subset X$ one has
\begin{equation}\label{eq:23}
\lim_{t \downarrow 0}\int_{A_1} \left(\int_{A_2}t^{(n+2)/2} |\langle \nabla_xp(x, y, t), V(x)\rangle|^2 \di\meas (x)\right) \di\meas (y)
=\frac{c_n}{\omega_n}\int_{A_1\cap A_2} |V|^2\di \mathcal{H}^n. 
\end{equation}
\end{theorem}
\begin{proof} Recall that \eqref{eq:goodlimitRN} of Theorem~\ref{thm:RN} gives that $r^n/\meas(B_r(y))$ converges as
$r\to 0$ to $1/(\omega_n\theta(y))$ for $\meas$-a.e. $y\in A_1$, 
By an argument similar to the proof of Theorem~\ref{th:convergence1},  using also \eqref{eq:exp}, we obtain
\begin{equation}\label{eq:22}
\phi_t(y):=t\meas (B_{\sqrt{t}}(y)))\int_{A_2}|\langle \nabla_xp(x,y,t), V(x)\rangle|^2 \di\meas (x)\le C(K, N)\|V\|_{L^{\infty}}^2 
\end{equation}
for all $y\in X$ and $t\in (0,C_4^{-1})$. Let  
\begin{equation}
f_t(y):=\frac{\sqrt{t}^{n}}{\meas (B_{\sqrt{t}}(y))}1_{A_1}(y)\phi_t(y),\qquad
g_t(y):=C(K, N)\|V\|_{L^{\infty}}^21_{A_1}(y)\frac{\sqrt{t}^{n}}{\meas (B_{\sqrt{t}}(y))},
\end{equation}
so that \eqref{eq:22} gives $f_t(y)\leq g_t(y)$.
Note that (\ref{eq:21}) and (\ref{eq:20})  yield
\begin{equation}\label{eq:222}
\lim_{t \downarrow 0}f_t(y) = \frac{c_n}{\omega_n} 1_{A_2}(y)
\frac{1}{\theta(y)} |V|^{2}(y)\quad\text{for $\meas$-a.e. $y\in A_1$.}
\end{equation}
Applying Lemma~\ref{lem:dct} with $g(y)=C(K, N)\|V\|_{L^{\infty}}^21_{A_1}(y)/(\omega_n\theta(y))$
and taking \eqref{eq:25} into account we get
\begin{equation}
\lim_{t \downarrow 0}\int_Xf_t\di\meas=\int_X\lim_{t \downarrow 0}f_t\di\meas=
\frac{c_n}{\omega_n}\int_{A_1\cap A_2} |V|^2\di \mathcal{H}^n,
\end{equation}
which proves (\ref{eq:23}).
\end{proof}

We are now in a position to prove the main result of this subsection.

\begin{theorem}\label{th:convergence2}
Assume that as $r \downarrow 0$
\begin{equation}\label{eq:goodlimit22}
\frac{r^n}{\meas(B_r(x))} \to \lim\limits_{r \downarrow 0} \frac{r^n}{\meas(B_r(x))} \quad \mathrm{in}\, L^2(X, \meas).
\end{equation}
Then $\tilde{g}_t$ $L^2$-strongly converge to $\tilde{g}$ as $t \downarrow 0$.
\end{theorem}

\begin{proof}
Let $A_2 \subset X$ be a Borel set and $V \in L^\infty(T(X,\dist,\meas))$. Then Fubini's theorem leads to 
\[
\int_{A_2} \tilde{g}_t(V,V) \di \meas  =  \int_X  \int_{\mathcal{R}_n^* \cap A_2} \scal{\nabla_x p(x,y,t)}{V(x)}^2 \di \meas(x) \di \meas (y)
\]
Then, we can apply Theorem~\ref{thm:asymptotic4} to get 
\begin{equation}\label{hhhhgggbb}
\int_{A_2} \tilde{g}_t(V,V) \di \meas  \to \frac{c_n}{\omega_n} \int_{\mathcal{R}_n^* \cap A_2} |V|^2 \di \mathcal{H}^n = \int_{A_2}\tilde{g}(V,V) \di \meas.
\end{equation}
Let us prove now the $L^2$-strong convergence of ${\bf \tilde{g}}_t$ to ${\bf \tilde{g}}$ as $t \downarrow 0$ using
Proposition~\ref{prop:convcrigi} with (\ref{hhhhgggbb}). Note that
\[
|{\bf \tilde{g}}|_{HS} = \frac{c_n}{\omega_n\theta}1_{\mathcal{R}_n^*} |{\bf g}|_{HS},
\qquad
|{\bf \tilde{g}}_t|_{HS} = t^{(n+2)/2}  |{\bf g}_t|_{HS}.
\]
Let us write for clarity $\tilde{F}(x,t)=\left| \int_X \dist_x p(x,y,t) \otimes \dist_xp(x,y,t) \di \meas(y) \right|_{HS}$. Applying  (\ref{eq:normaPhi_t_other}), 
\eqref{eq:blowuphs2} and \eqref{eq:righthsk} we get
\begin{align}\label{b0}
\limsup\limits_{t \downarrow 0} \int_X |{\bf \tilde{g}}_t|_{HS}^2 \di \meas 
& = \limsup\limits_{t \downarrow 0} \int_{\mathcal{R}_n} t^{n+2} |{\bf g}_t|_{HS}^2 \di \meas \nonumber\\
& = \limsup\limits_{t \downarrow 0} \int_{\mathcal{R}_n} t^{n+2} \tilde{F}^2(x,t) \di \meas(x)\nonumber \\
& \le \int_{\mathcal{R}_n} \limsup\limits_{t \downarrow 0} \left( \frac{t^{(n+2)/2}}{t \meas(B_{\sqrt{t}}(x)))} \right)^2 t^2\meas(B_{\sqrt{t}}(x))^2 \tilde{F}^2(x,t) \di \meas(x)\nonumber \\
& = \int_{\mathcal{R}_n} \frac{1}{\omega_n^2\theta^2}  n c_n^2 \di \meas(x) = \int_X|{\bf \tilde{g}}|_{HS}^2 \di \meas <\infty.
\end{align}
Notice that we are enabled to pass to the limit under the integral sign thanks to \eqref{eq:goodlimit22} and Lemma~\ref{lem:dct}, since the convergence
in \eqref{eq:blowuphs2} is dominated.
\end{proof}

We obtain in particular the following corollary when the metric measure space $(X,\dist,\meas)$ is Ahlfors $n$-regular: indeed, in this
case obviously one has $n=\dim_{\dist,\meas}(X)$, $\meas$ and $\mathcal{H}^n$ are mutually absolutely continuous 
and the existence of the limits in \eqref{eq:goodlimit22}, as well as the validity of the
equality, are granted by the rectifiability of $\mathcal{R}_n$ and by the dominated convergence theorem.

\begin{corollary}\label{cor:ahl}
Assume that $\meas$ is Ahlfors $n$-regular, i.e. there exists $C\geq 1$ such that 
$$
C^{-1}\le \frac{\meas (B_r(x))}{r^n} \le C\quad \forall r\in (0,1], \, \forall x \in X.
$$
Then $\tilde{g}_t$ $L^2$-strongly converge to $\tilde{g}$ as $t \downarrow 0$.
\end{corollary}

\subsection{Behavior with respect to the mGH-convergence}

Fix a mGH-convergent sequence of compact $\RCD^*(K, N)$-spaces, with $(X,\dist,\meas)$ compact as well:
$$
(X_j, \dist_j, \meas_j) \stackrel{mGH}{\to} (X, \dist, \meas).
$$
In this section we can adopt the extrinsic point of view of Subsection~\ref{secopt}, viewing when necessary all metric measure spaces as
isometric subsets of a compact metric space $(Y,\dist)$, with $X_j$ convergent to $X$ w.r.t. the Hausdorff distance and 
$\meas_j$ weakly convergent to $\meas$.

Let us denote by $\lambda_{i, j},\, \lambda_i,\, \phi_{i, j},\, \phi_i$ the corresponding eigenvalues and eigenfunctions of $-\Delta_j$, $-\Delta$, 
respectively, listed taking into account their multiplicities (we will also use a similar notation below), recall that $\{\phi_{i,j}\}_{i\geq 0}$ are orthonormal bases
of $L^2(X_j,\meas_j)$ and that, according to \cite{GigliMondinoSavare13}, for any $i$ one has
$\lambda_{i,j}\to\lambda_i$ as $j\to\infty$, so-called spectral convergence. 
In addition, by the uniform bound on the diameters of the spaces,
we know from Proposition~\ref{prop:Jiang} (see also \cite{Jiang}) that uniform Lipschitz continuity of 
eigenfunctions holds, i.e.
\begin{equation}\label{1}
\sup_j\||\nabla \phi_{i, j}|_j\|_{L^{\infty}}<\infty \qquad \forall i\geq 0.
\end{equation}
With no loss of generality, we can also assume that the $\phi_{i,j}$ are restrictions of Lipschitz functions defined on
$Y$, with Lipschitz constant equal to $\|\nabla \phi_{i, j}\|_{L^\infty(X_j,\meas_j)}$.

Although the following lemma was already discussed in the proof of \cite[Th. 7.8]{GigliMondinoSavare13}, we give the proof for 
the reader's convenience.

\begin{lemma}\label{lem:conveigenfunct}
Under the same setting as above, there exist $j(k)$ and an $L^2$-orthonomal basis $\{\psi_i\}_{i\geq 0}$ of $L^2(X, \meas)$ such that 
$\phi_{i, j(k)}$ $H^{1, 2}$-strongly converge as $k\to\infty$ to $\psi_i$ for all $i$. In addition, the convergence is also
uniform in this sense: for all $\epsilon>0$ there exist $\delta>0$ and $k_0$ such that $k\geq k_0$, $x_k\in\supp\meas_{j(k)}$
and $x\in\supp\meas$ with $\dist(x_k,x)<\delta$ imply $|\phi_{i,j(k)}(x_k)-\psi_i(x)|<\epsilon$.
\end{lemma}
\begin{proof}
Since $\| |\nabla \phi_{i, j}|_j \|_{L^2}^2=\lambda_{i, j}$, by Theorem~\ref{thm:stability lap} and a diagonal argument there exist a
subsequence $j(k)$ and $\psi_i \in L^2(X, \meas)$ such that $\phi_{i, j(k)}$ $H^{1, 2}$-strongly converge as $k\to\infty$ to $\psi_i$ 
for all $i\geq 0$, with $L^2$-weak convergence of $\Delta_{j(k)} \phi_{i, j(k)}$ to $\Delta \psi_i$. 
In particular we obtain that $\Delta \psi_i=\lambda_i\psi_i$ for all $i$ and that 
$$
\int_X\psi_\ell\psi_m\di\meas=\lim_{k \to \infty}\int_{X_{j(k)}}\phi_{\ell, j(k)}\phi_{m, j(k)}\di\meas_{j(k)}=\delta_{\ell m}.
$$
Thus, as written above, $\{\psi_i\}_{i\geq 0}$ is an $L^2$-orthonormal basis of $L^2(X, \meas)$.
Finally the uniform convergence is justified by the $L^2$-strong convergence of $\phi_i$ with (\ref{1}).
\end{proof}

Taking Lemma~\ref{lem:conveigenfunct} into account, with no loss of generality in the sequel we can assume that 
$\phi_{i, j}$ $H^{1, 2}$-strongly converge to $\phi_i$ for all $i\geq 0$, in addition with uniform convergence in $Y$.

Let us discuss the $L^2$-convergence of Riemannian semi metrics ${\bar g}_i$ with respect to mGH convergence. 
It is easy to check that the following definition is compatible with Definition~\ref{defcong}, dealing with metrics in
a fixed metric measure structure.

\begin{definition}\label{def:conriemaspaces}
We say that Riemannian semi metrics ${\bar g}_j$ on $(X_j, \dist_j, \meas_j)$ $L^2$-weakly converge to a Riemannian semi metric 
${\bar g}$ on $(X, \dist, \meas)$ if  
$\sup_j\int_{X_j}|{\bf {\bar g}}_j|_{HS}^2\di\meas_j<\infty$ and 
${\bar g}_j(\nabla h_j, \nabla h_j)$ $L^2$-weakly converge to ${\bar g}(\nabla h, \nabla h)$, whenever $h_j$ $H^{1, 2}$-strongly converge to 
$h$ with $\sup_j\||\nabla h_j|_j \|_{L^{\infty}}<\infty$.
$L^2$-strong convergence is defined by requiring, in addition, that $\lim_j\int_{X_j}|{\bf {\bar g}}_j|_{HS}^2\di\meas_j=
\int_{X}|{\bf {\bar g}}|_{HS}^2\di\meas$.
\end{definition}

It is not difficult to show several fundamental properties of $L^2$-strong/weak convergence of semi metrics, including $L^2$-weak compactness
(not needed in this paper) and lower semicontinuity of $L^2$-norms with respect to $L^2$-weak convergence, 
as discussed in Definition~\ref{prop:convcrigi}; in particular, the convergence can be improved from weak to strong if 
$$
\limsup_j\int_{X_j}|{\bf {\bar g}}_j|_{HS}^2\di\meas_j\leq\int_{X}|{\bf {\bar g}}|_{HS}^2\di\meas.
$$

\begin{theorem}\label{th:convergence3}
Let $t_j \to t \in (0, \infty)$, let $\Phi^{X_j}_{t_j}:X_j\to L^2(X_j,\meas_j)$ be the corresponding embeddings 
and let $g_{t_j}^{X_j}$ be the corresponding pull-back metrics of $(X_j,\dist_j,\meas_j)$.
 Then $g_{t_j}^{X_j}$ $L^2$-strongly converge to $g_t^X$ and $\Phi_{t_j}^{X_j}(X_j)$,
endowed with the $L^2(X_j,\meas_j)$ distance, GH-converge to
$\Phi_t^X(X)$ endowed with the $L^2(X,\meas)$ distance.
\end{theorem}
\begin{proof}  By rescaling with no loss of generality we can assume that $t_j \equiv t=1$.

Let us prove first the convergence of metrics. Note that (\ref{eq:riem est}) yields $\sup_j\||{\bf g}_{1}^{X_j}|_{HS}\|_{L^{\infty}}<\infty$. 
For all $k\geq 1$, recalling the representation formula \eqref{eq:explicit expression} for the metrics, we define  
$$
{\bf G}_j^k:=\sum_{i \ge k}^\infty e^{-2\lambda_{i, j}}\dist \phi_{i, j} \otimes \dist \phi_{i, j} \left( ={\bf g}_{1}^{X_j}-({\bf g}_{1}^{X_j})^{k-1}\right)
$$
and define ${\bf G}^k$ analogously. Then, arguing as in \eqref{eq:disuno}, we get
\begin{equation}\label{2}
\int_{X_j}|{\bf G}_{j}^k|_{HS}^2\di\meas_j= \sum_{\ell, \,m \ge k}^{\infty}e^{-2(\lambda_{\ell, j}+\lambda_{m, j})}\int_{X_j}\langle \nabla \phi_{\ell, j}, \nabla \phi_{m, j}\rangle^2_j \di\meas_j
\leq C\sum_{\ell \ge k}^{\infty}\lambda_{\ell, j}e^{-2\lambda_{\ell, j}}
\end{equation}
with $C=C(K,N, \meas (X))$, and a similar estimate holds for $\int_{X}|{\bf G}^k|_{HS}^2\di\meas$.
On the other hand, since 
$$
\int_{X_j}\int_{X_j}|\nabla_xp_j(x, y,1)|^2_j\di\meas_j(x) \di\meas_j(y)=\sum_{\ell=1}^{\infty}\lambda_{\ell, j}e^{-2\lambda_{\ell, j}}
$$
and
$$
\int_{X_j}\int_{X_j}|\nabla_xp_j(x, y, 1)|^2_j\di\meas_j(x) \di\meas_j(y) \to \int_{X}\int_{X}|\nabla_xp (x, y,1)|^2\di\meas(x) \di\meas (y),
$$
taking also the spectral convergence into account we get
\begin{equation}\label{3}
\sum_{\ell\ge k}^{\infty}\lambda_{\ell, j}e^{-2\lambda_{\ell, j}} \to \sum_{\ell \ge k}^{\infty}\lambda_\ell e^{-2\lambda_\ell}\qquad\forall k. 
\end{equation}
In particular for any $\epsilon >0$ there exists $k$ such that for all sufficiently large $j$
$$
\sum_{\ell\ge k}^{\infty}\lambda_{\ell,\, j}e^{-2\lambda_{\ell, j}}+\sum_{\ell\ge k}^{\infty}\lambda_\ell e^{-2\lambda_\ell}<\epsilon.
$$
Thus, for sufficiently large $j$ one has
\begin{equation}\label{4}
\int_{X_j}|{\bf G}_{ j}^k|_{HS}^2\di\meas_j+\int_X|{\bf G}^k|_{HS}^2\di\meas <2C\epsilon.
\end{equation}
On the other hand, since $\phi_{\ell, j}$ $H^{1,2}$-strongly converge to $\phi_\ell$, (\ref{1}) yields that 
$\langle \nabla \phi_{\ell, j}, \nabla \phi_{m, j}\rangle$ $L^p$-strongly converge to $\langle \nabla \phi_\ell, \nabla \phi_m\rangle$ for all $p \in [1, \infty)$.
In particular, as $j \to \infty$ we  get
\begin{align}\label{5}
\int_{X_j}|({\bf g}_{1}^{X_j})^k|_{HS}^2\di\meas_j &=\sum_{\ell, \,m=1}^ke^{-2(\lambda_{\ell, j}+\lambda_{m, j})}
\int_{X_j}\langle \nabla \phi_{\ell, j}, \nabla \phi_{m, j}\rangle^2_j\di\meas_j \nonumber \\
&\to \sum_{\ell,\, m=1}^ke^{-2(\lambda_\ell+\lambda_m)}\int_{X}\langle \nabla \phi_\ell, \nabla \phi_{m}\rangle^2\di\meas 
= \int_X|({\bf g}_1^X)^k|_{HS}^2\di\meas.
\end{align}
Since $\epsilon$ is arbitrary, combining \eqref{4} with \eqref{5} yields 
\begin{equation}\label{L2normconv}
\int_{X_j}|{\bf g}_{1}^{X_j}|_{HS}^2\di\meas_j \to \int_X|{\bf g}_1^X|^2_{HS}\di\meas.
\end{equation}
Since it is easy to check that Lemma \ref{lem:conveigenfunct} yields that $({\bf g}_{1}^{X_j})^{k-1}$ $L^2$-weakly converge to $({\bf g}_{1}^X)^{k-1}$, combining (\ref{4}) with (\ref{L2normconv}) completes the proof of the $L^2$-strong convergence of metrics.

Now we prove the second part of the statement.
Using the eigenfunctions $\varphi_{i,j}$ we can embed isometrically all $\Phi_t^{X_j}(X_j)\subset L^2(X_j,\meas_j)$ into $\ell_2$,
and then we need only to prove the Hausdorff convergence inside $\ell_2$ of the sets $W_j$ to $W$, where
$$
W_j=\left\{\bigl(e^{-\lambda_{i,j}}\varphi_{i,j}(x)\bigr)_{i\geq 1}:\ x\in X_j\right\},\qquad
W=\left\{\bigl(e^{-\lambda_{i}}\varphi_{i}(x)\bigr)_{i\geq 1}:\ x\in X\right\}.
$$
By Propositions \ref{prop:lowerbound} and \ref{prop:Jiang}, for all $\epsilon>0$ there exists $k \in \mathbb{N}$ such that 
for all $j$
$$
\sum_{i \ge k+1}e^{-2\lambda_{i, j}} \|\phi_{i, j}\|_{L^{\infty}}^2 <\epsilon^2 \qquad \sum_{i \ge k+1}e^{-2\lambda_i}\|\phi_i\|_{L^{\infty}}^2<\epsilon^2.
$$
Denoting $\pi^k: \ell_2 \to \ell_2$ the projection defined by $\pi^k((x)_i):=(x_1, \ldots, x_k, 0, \ldots)$,
from this it is easy to get
$$
\dist_H^{\ell_2}(W_j, W_j^k)<\epsilon, \qquad \dist_H^{\ell_2}(W, W^k)<\epsilon,
$$
where $W_j^k:=\pi^k(W_i)$, $W^k:=\pi^k(W)$ and $\dist_H^{\ell_2}$ denotes the Hausdorff distance. Hence, by the triangle inequality,
 it suffices to check that $\dist_H^{\ell_2}(W_j^k, W^k) \to 0$ for fixed $k$. Since 
$$
W^k_j=\left\{\bigl(
e^{-\lambda_{1,j}}\varphi_{1,j}(x),e^{-\lambda_{2,j}}\varphi_{2,j}(x),\ldots,e^{-\lambda_{k,j}}\varphi_{k,j}(x),0,0,\ldots\bigr):\ x\in X_j\right\},
$$
and an analogous formula holds for $W^k$, from the uniform convergence of the $\phi_{i,j}$ to $\phi_i$ 
we immediately get that $\dist_H^{\ell_2}(W_j^k, W^k) \to 0$. 
\end{proof}

\begin{remark}\label{6666000}
The canonical Riemannian metrics $g^{X_j}$ $L^2$-weakly converge to $g^X$, as a direct consequence of \cite[Th.~5.7]{AmbrosioHonda}.
In particular the lower semicontinuity of the $L^2$-norms of $g^{X_j}$, namely
\begin{equation}
\liminf_{j \to \infty}\int_{X_j}|\mathbf{g}^{X_j}|_{HS}^2\di \meas_j \ge \int_X|\mathbf{g}^X|^2_{HS}\di \meas
\end{equation}
yields
\begin{equation}\label{lowersemidim}
\liminf_{j \to \infty}\dim_{\dist_j,\meas_j}(X_j)\ge\dim_{\dist, \meas}(X).
\end{equation}
Indeed, setting $n_j=\dim_{\dist_j,\meas_j}(X_j)$, $n=\dim_{\dist,\meas}(X)$,
Lemma~\ref{riemdim} shows that $\int_{X_j}|\mathbf{g}^{X_j}|^2_{HS}\di \meas_j=n_j\meas_j(X_j)$
and $\int_X|\mathbf{g}^X|^2_{HS}\di \meas=n\meas(X)$.

This allows us to define the notion that \textit{$\{(X_j, \dist_j, \meas_j)\}_j$ is a noncollapsed convergent sequence to $(X, \dist, \meas)$} 
if $\lim_jn_j=n$ (see also \cite{Kita}). Moreover, convergence occurs without collapse if and only if
$$
\lim_{j \to \infty}\int_{X_j}|\mathbf{g}^{X_j}|_{HS}^2\di \meas_j=\lim_{j \to \infty}n_j\meas_j(X_j)=n\meas (X) =\int_X|\mathbf{g}^X|^2_{HS}\di \meas
$$
that is, if and only if $g^{X_j}$ $L^2$-strongly converge to $g^X$
(these observation are justified even for the noncompact case if we replace $X_j, X$ by $B_1(x_j), B_1(x)$, where $x_j \to x$).
One of the important points in Theorem~\ref{th:convergence3} is that the Riemannian metrics $g_{t_j}^{X_j}$ are $L^2$-strongly convergent 
even without the noncollapsed assumption, if $t_j\to t>0$. Compare with the next section.
\end{remark}

\section{Quantitative $L^2$-convergence for noncollapsed spaces}\label{sec:6}

Let us start this section with the following three questions, related to each other:

\noindent (1) Can Theorem~\ref{th:convergence3} be improved as follows:
\begin{equation}\label{ddrt}
(c_{n_j})^{-1}t_j\meas_j(B_{\sqrt{t_j}}(x))g_{t_j}^{X_j} \to g^X 
\end{equation}
in the sense of $L^2$-strong convergence, whenever $t_j \to 0^+$ and one has a mGH convergent sequence of compact $\RCD^*(K, N)$ spaces $(X_j, \dist_j, \meas_j) \stackrel{mGH}{\to} (X, \dist, \meas)$ with uniformly bounded diameter and $n_j:=\mathrm{dim}_{\dist_j, \meas_j} (X_j)$?

\noindent (2) Does a quantitative version of Theorem~\ref{th:convergence1} hold? Namely, for all $\epsilon>0$, $d \ge 1$ does there exist 
$t_0:=t_0(K, N, \epsilon, d)>0$ such that
\begin{equation}\label{tthgff}
\sup\limits_{0<t<t_0} \||\hat{\bf g}_t - \hat{\bf g}|_{HS}\|_{L^2(X,\meas)} \le \epsilon
\end{equation}
holds for any $\RCD^*(K,N)$ space $(X,\dist,\meas)$ with $d^{-1} \le \mathrm{diam}(X, \dist) \le d$, $\mathrm{supp}\,\meas=X$ and $\meas(X)=1$?

\noindent (3) Recall a result proved in \cite{Portegies16}: for all $\epsilon>0$, $\tau>0$, $d>0$, $n \in \mathbb{N}$, there exists 
$t_0:=t_0(K, n, \epsilon, \tau, d)>0$ such that for all $0<t\le t_0$ there exists $N_0:=N_0(K, n, \epsilon, \tau, d, t)\ge 1$ such 
that if a closed Riemannian manifold $(M^n, g)$ satisfies $\mathrm{Ric}_g \ge K$, $\mathrm{diam} (M^n, \dist_g)\le d$, $\mathrm{inj}_g \ge \tau$, 
then for all $x \in M^n$ and $N \ge N_0$,
\begin{equation}\label{bbp2}
1-\epsilon \le |g-c(n)t^{(n+2)/2}(\Phi_t^N)^*g_{\mathbb{R}^N}|_{HS}(x) \le 1+\epsilon
\end{equation}
where $\mathrm{inj}_g$ denotes the injectivity radius and
$\Phi_t^N$ is the truncated embedding map of Proposition~\ref{projj}.
What happens if we replace the assumption ``$\mathrm{inj}_g \ge \tau$'' by the weaker one ``$\mathcal{H}^n(M^n)\ge \tau$''?

Let us give a simple example where \eqref{ddrt} is not satisfied. As a consequence,
also the second question has no positive answer in general, because (\ref{tthgff}) easily implies (\ref{ddrt}).
 
\begin{example}
We consider a sequence of collapsing flat tori:
$$
(X^r, \dist^r, \meas^r):=\left(\mathbb{S}^1(1) \times \mathbb{S}^1(r), \dist_{g^{1}\times g^r}, \frac{\mathcal{H}^2}{4\pi^2r}\right) \stackrel{mGH}{\to} \left(\mathbb{S}^1(1), \dist_{g^1}, \frac{\mathcal{H}^1}{2\pi}\right)=:(X, \dist, \meas) \quad (r \downarrow 0),
$$
where $\mathbb{S}^1(r):=\{x \in \mathbb{R}^2; |x|=r\}$ with the standard Riemannian metric $g^r$.

Then choosing sufficiently small $t_r$ with $\||(c_2)^{-1}t_r\meas^r(B_{\sqrt{t_r}}(x))({\bf g^1\times g^r})-({\bf g^1 \times g^r})|_{HS}\|_{L^2}<r$ yields that $\||(c_2)^{-1}t_r\meas^r(B_{\sqrt{t_r}}(x))({\bf g^1\times g^r})|_{HS}\|_{L^2}^2 \to 2 \neq 1=\||{\bf g^1}|_{HS}\|_{L^2}^2$ as $r \downarrow 0$ which shows that (\ref{ddrt}) is not satisfied.
\end{example}

In this section we give positive answers to all above questions for \textit{noncollapsed} $\RCD^*(K, N)$ spaces 
(Theorems~\ref{thm:quantitativ} and \ref{label}), except for the embeddedness of $\Phi_t^N$.
For that, we introduce two useful notations to simplify our arguments (for the latter one, see also \cite{CheegerColding}), 
\begin{enumerate}
\item for $a, \,b \in \mathbb{R}$ and $\epsilon \in (0, \infty)$, we write $a=b \pm \epsilon$ if $|a-b| \le \epsilon$,
\item any function $f:(\mathbb{R}_{>0})^{k+m} \to \mathbb{R}_{\ge 0}$, satisfying that
$$
\lim_{\epsilon_1, \ldots, \epsilon_k \to 0}f(\epsilon_1, \ldots, \epsilon_k, c_1, \ldots, c_m)=0
$$
for all fixed $c_1, \ldots, c_m \in \mathbb{R}$, is denoted by $\Psi(\epsilon_1, \ldots, \epsilon_k; c_1, \ldots, c_m)$ for simplicity.
\end{enumerate} 

The following three convergence results are valid in general compact $\RCD^*(K, N)$ spaces, and they will play key roles in the proofs of the main theorems.

Let us recall that for pointed (not necessary compact) $\RCD^*(K, N)$ spaces, the pointed mGH-convergence topology is metrizable. For example, the $p\mathbb{G}_W$-distance introduced in \cite{GigliMondinoSavare13} gives such a distance.
\begin{proposition}\label{q97j}
Let $(X, \dist, \meas)$ be a compact $\RCD^*(K, N)$ space and 
$f \in D(\Delta)$ with $\||\nabla f|\|_{L^{\infty}} +\|\Delta f\|_{L^{\infty}} \le L$.
Assume that the $p\mathbb{G}_W$-distance between $(X, \sqrt{t}^{-1}\dist, \meas (B_{\sqrt{t}}(x))^{-1}\meas, x)$ and $(\mathbb{R}^n, \dist_{\mathbb{R}^n}, \omega_n^{-1}\mathcal{H}^n,  0_n)$ is at most $\epsilon$ for some $t>0$, $x\in X$. Then 
\begin{equation}\label{eq:quant heat conv}
\int_Xt\meas (B_{\sqrt{t}}(z))\langle \nabla_zp(z, x, t), \nabla f(z) \rangle^2\di \meas (z) =c_n\fint_{B_{\sqrt{t}}(x)}|\nabla f|^2\di \meas \pm \Psi(\epsilon, t; K, N, L).  
\end{equation}
\end{proposition}
\begin{proof}
The proof is achieved by contradiction.
Assume that \eqref{eq:quant heat conv} does not hold for some $K$, $N$ and $L$. Then there exist $\tau >0$ and sequences as follows;
\begin{enumerate}
\item $(X_i, \dist_i, \meas_i)$ are compact $\RCD^*(K, N)$ spaces,
\item $t_i \in (0, 1)$ with $t_i \to 0^+$, 
\item $f_i \in D(\Delta_i)$ with $\||\nabla f_i|_i\|_{L^{\infty}}+\|\Delta^i f_i\|_{L^{\infty}}\le L$
\item $x_i \in X_i$ with $(X_i, \sqrt{t_i}^{-1}\dist_i, \meas (B_{\sqrt{t_i}}(x_i))^{-1}\meas, x_i) \stackrel{mGH}{\to} (\mathbb{R}^n, \dist_{\mathbb{R}^n}, \omega_n^{-1}\mathcal{H}^n, 0_n)$ and
\begin{equation}\label{6t}
\left|\int_{X_i}t_i\meas (B_{\sqrt{t_i}}(z))\langle \nabla_zp_i(z, x_i, t), \nabla f_i(z) \rangle^2_i\di \meas_i (z) -c_n\fint_{B_{\sqrt{t_i}}(x_i)}|\nabla f_i|_i^2\di \meas_i\right| \ge \tau.
\end{equation}
\end{enumerate} 
Since $|\Delta^{t_i} f_{\sqrt{t_i}, x_i}| \le t_i L$, where $\Delta^{t_i}$ is the Laplacian on $(X_i, \sqrt{t_i}^{-1}\dist_i, \meas_i(B_{\sqrt{t_i}}(x_i))^{-1}\meas_i)$, by combining Theorem~\ref{thm:compact loc sob} and Theorem \ref{thm:stability lap},
with no loss of generality we can assume that $f_{\sqrt{t_i}, x_i}$ $H^{1, 2}_{\mathrm{loc}}$-converge to a linear growth harmonic function $f$ on $\mathbb{R}^n$.
Note that $|\nabla f|^{*2} \equiv |\nabla f|^{*2}(0_n)$. 

By an argument similar to the proof of Proposition~\ref{prop:blowupvector}, we have
\begin{align}\label{aw}
&\int_{X_i}t_i\meas (B_{\sqrt{t_i}}(z))\langle \nabla_zp_i(z, x_i, t), \nabla f_i(z) \rangle^2_i\di \meas_i (z) \nonumber \\
&=\int_{B_{R\sqrt{t_i}}(x_i)}t_i\meas (B_{\sqrt{t_i}}(z))\langle \nabla_zp_i(z, x_i, t), \nabla f_i(z) \rangle^2_i\di \meas_i (z) \pm \Psi (R^{-1}; K, N, L).
\end{align}
Then taking the limit $i \to \infty$ shows
\begin{align}\label{eq:aw}
&\int_{B_{R\sqrt{t_i}}(x_i)}t_i\meas (B_{\sqrt{t_i}}(z))\langle \nabla_zp_i(z, x_i, t), \nabla f_i(z) \rangle^2_i\di \meas_i (z) \nonumber \\
&\to\int_{B_R(0_n)}\hat{\mathcal{H}}^n(B_1(z))\langle \nabla_zq_n(z, 0_n, 1), \nabla f(z)\rangle^2\di \hat{\mathcal{H}}^n=c_n(R)|\nabla f|^{*2}(0_n),
\end{align} 
where recall $\hat{\mathcal{H}}^n=\omega_n^{-1}\mathcal{H}^n$.
Thus letting $R \uparrow \infty$ and taking \eqref{aw}, \eqref{6t} and \eqref{eq:aw} into account yields
\begin{equation}\label{x}
\left|c_n|\nabla f|^{*2}(0_n)-c_n\fint_{B_1(0_n)}|\nabla f|^2\di \mathcal{H}^n\right| \ge \tau
\end{equation}
which contradicts the fact that $|\nabla f|^{*2} \equiv |\nabla f|^{*2}(0_n)$. 
Thus the proof is completed.
\end{proof}

\begin{corollary}\label{applied}
Under the same assumptions as in Proposition~\ref{q97j}, let $h \in \mathrm{Lip}(X, \dist)$ with $\|h\|_{L^{\infty}}+\||\nabla h|\|_{L^{\infty}}\le L$. 
Then 
\begin{align}\label{w3x}
&\int_Xt\meas (B_{\sqrt{t}}(z))\langle \nabla_zp(z, x, t), h(z)\nabla f(z) \rangle^2\di \meas (z) =c_nh(x)^2\fint_{B_{\sqrt{t}}(x)}|\nabla f|^2\di \meas\pm 
\Psi,   
\end{align}
with $\Psi(\epsilon, t; K, N, L)$.
\end{corollary}
\begin{proof}
For the sake of brevity, let us write $H(z,x,t):=t\meas(B_{\sqrt{t}}(z))\langle \nabla_zp(z, x, t), h(z)\nabla f(z) \rangle^2,$ $\bar{H}(z,x,t) = t\meas(B_{\sqrt{t}}(z))\langle \nabla_zp(z, x, t), h(x)\nabla f(z) \rangle^2$ and $C:=\fint_{B_{\sqrt{t}}(x)}|\nabla f|^2\di \meas$. 
As discussed in the proof of Proposition~\ref{prop:blowupvector}, we know that
\begin{equation}\label{kk2}
\int_X H(z,x,t) \di \meas (z) =\int_{B_{R\sqrt{t}}(x)}H(z,x,t) \di \meas (z) \pm \Psi(R^{-1}; K, N, L)
\end{equation}
and
\begin{equation}\label{kk2'}
\int_X \bar{H}(z,x,t) \di \meas (z) =\int_{B_{R\sqrt{t}}(x)}\bar{H}(z,x,t) \di \meas (z) \pm \Psi(R^{-1}; K, N, L).
\end{equation}
{From} \eqref{eq:2} and the fact that $|h\nabla f_{\sqrt{t}, x}-h(x)\nabla f_{\sqrt{t}, x}| \le L^2R\sqrt{t}$ holds on $B_R^{\dist_t}(x)$, where $\dist_t=\sqrt{t}^{-1}\dist$, \eqref{kk2} and \eqref{kk2'} imply
$$
\int_X H(z,x,t) \di \meas (z) = \int_X \bar{H}(z,x,t) \di \meas (z) \pm (\Psi(R^{-1}; K, N, L) + \Psi(R\sqrt{t}; K, N, L)).
$$
Proposition \ref{q97j} applied to $h(x)f$ gives $\int_{X}\bar{H}(z,x,t)\di \meas (z) = c_nh(x)^2C\pm \Psi(t, \epsilon; K, N, L)$, which yields to
\begin{align}\label{kk9}
\int_X H(z,x,t) \di \meas (z) = c_nh(x)^2C \pm (\Psi(R^{-1}; K, N, L) + \Psi(R\sqrt{t}; K, N, L)+\Psi(t, \epsilon; K, N, L)).
\end{align}
Thus taking $R=t^{-1/4}$ in  \eqref{kk9} completes the proof.
\end{proof}

\begin{lemma}\label{lem:suf}
Let $(X_j, \dist_j, \meas_j) \stackrel{mGH}{\to} (X, \dist, \meas)$ be a mGH-convergent sequence of compact $\RCD^*(K, N)$ spaces with uniformly bounded diameter.
Then a sequence of Riemannian semi metrics ${\bar g}_j$ on $(X_j, \dist_j, \meas_j)$, with $\sup_j\int_{X_j}|{\bf {\bar g}}_j|_{HS}^2\di\meas_j<\infty$, $L^2$-weakly converge to a Riemannian semi metric ${\bar g}$ on $(X, \dist, \meas)$ according to Definition~\ref{def:conriemaspaces}
if and only if
\begin{equation}\label{11q}
\int_{X_j}{\bar g}_j(h_j^1\nabla h_j^2, h_j^1\nabla h_j^2)\di \meas_j \to \int_X{\bar g}(h^1\nabla h^2, h^1\nabla h^2)\di \meas
\end{equation}
whenever $h_j^1 \in \mathrm{LIP}(X_j, \dist_j), h_j^2 \in \mathrm{Test}F(X_j, \dist_j, \meas_j)$ $L^2$-strongly converge to $h^1 \in \mathrm{Lip}(X, \dist), h^2 \in \mathrm{Test}F(X, \dist, \meas)$, respectively, with $\sup_{j} (\||\nabla h_j^1|_j\|_{L^{\infty}}+\||\nabla h_j^2|_j\|_{L^{\infty}} +\|\Delta^j h_j^2\|_{L^{\infty}})<\infty$.
\end{lemma}
\begin{proof}
It is enough to check the `` if '' part.

By an argument similar to the proof of \cite[Th.10.3]{AmbrosioHonda}, we see that $\||{\bf {\bar g}}|_{HS}\|_{L^2}\le \liminf_{i \to \infty}\||{\bf {\bar g}}_i|_{HS}\|_{L^2}$, in particular, ${\bf {\bar g}} \in L^2((T^*)^{\otimes 2}(X, \dist, \meas))$.

First let us remark that if a sequence $\phi_j \in L^2(X_j, \meas_j)$ with $\phi \in L^2(X, \meas)$ satisfies $\sup_j\|\phi_j\|_{L^2}<\infty$ and
\begin{equation}\label{ff4}
\int_{X_j}\psi_j\phi_j\di \meas_j \to \int_X\psi\phi\di \meas
\end{equation}
for every $L^2$-strongly convergent sequence $\psi_j \in \mathrm{Lip}(X_j, \dist_j) \to \psi \in \mathrm{Lip}(X, \dist)$ with  $\sup_j\||\nabla \psi_j|_j\|_{L^{\infty}}<\infty$, then $\phi_j$ $L^2$-weakly converge to $\phi$ because for every uniformly convergent sequence $\eta_j \in C^0(X_j) \to \eta \in C^0(X)$
and all $\epsilon >0$, we can find a uniformly convergent sequence $\hat{\eta}_j \in \mathrm{Lip}(X_j, \dist_j) \to \hat{\eta} \in \mathrm{Lip}(X, \dist)$ with $\sup_j\||\nabla \eta_j|_j\|_{L^{\infty}}<\infty$ and $\sup_j\|\eta_j-\hat{\eta}_j\|_{L^{\infty}}\le \epsilon$.

Replacing $h_j^1, h^1$ by $\sqrt{|h_j^1|+1}, \sqrt{|h^1|+1}$ respectively in (\ref{11q}) yields
\begin{align}\label{1126}
&\int_{X_j}|h_j|{\bar g}_j(\nabla h_j^2, \nabla h_j^2)\di \meas_j +\int_{X_j}{\bar g}_j(\nabla h_j^2, \nabla h_j^2)\di \meas_j \nonumber \\
&\to \int_X|h_1|{\bar g}(\nabla h^2, \nabla h^2)\di \meas + \int_X{\bar g}(\nabla h^2, \nabla h^2)\di \meas.
\end{align}
Since letting $h^1_j \equiv 1$ in (\ref{11q}) shows
$$
\int_{X_j}{\bar g}_j(\nabla h_j^2, \nabla h_j^2)\di \meas_j \to  \int_X{\bar g}(\nabla h^2, \nabla h^2)\di \meas,
$$
by (\ref{1126}) we have
$$
\int_{X_j}|h_j|{\bar g}_j(\nabla h_j^2, \nabla h_j^2)\di \meas_j \to  \int_X|h_1|{\bar g}(\nabla h^2, \nabla h^2)\di \meas,
$$
which easily yields (after truncation for $h_j$)
$$
\int_{X_j}h_j{\bar g}_j(\nabla h_j^2, \nabla h_j^2)\di \meas_j \to  \int_Xh_1{\bar g}(\nabla h^2, \nabla h^2)\di \meas,
$$
By (\ref{ff4}), we see that ${\bar g}_j(\nabla h_j^2, \nabla h_j^2)$ $L^2$-weakly converge to ${\bar g}(\nabla h^2, \nabla h^2)$ on $X$.

Let $f_j \in \mathrm{Lip}(X_j, \dist_j)$ with $\sup_j\||\nabla f_j|_j\|_{L^{\infty}}<\infty$, and let $f \in \mathrm{Lip}(X, \dist)$ be the $H^{1, 2}$-strong limit of $f_j$.
Then our goal is to prove that ${\bar g}_j(\nabla f_{j}, \nabla f_{j})$ $L^2$-weakly converge to ${\bar g}(\nabla f^2, \nabla f^2)$.

By using a mollified heat flow (c.f. \cite{AmbrosioHonda, Gigli}), we can find a sequence $h_i \in \mathrm{Test}F(X, \dist, \meas)$ with $\Delta h_i \in L^{\infty}$ for any fixed $i$, $\sup_{i}\| |\nabla h_i| \|_{L^{\infty}}<\infty$ and $h_i \to f$ in $H^{1, 2}(X, \dist, \meas)$.
Moreover, for any $i$, there exists a sequence $h_{i, j} \in \mathrm{Test}F(X_j, \dist_j, \meas_j)$ such that $\sup_j (\||\nabla h_{i, j}|_j\|_{L^{\infty}}+\|\Delta^j h_{i, j}\|_{L^{\infty}})<\infty$ and that $h_{i, j}$ $H^{1, 2}$-strongly converge to $h_i$.
Note that the argument above shows that  ${\bar g}_j(\nabla h_{i, j}^2, \nabla h_{i, j}^2)$ $L^2$-weakly converge to ${\bar g}(\nabla h^2_i, \nabla h^2_i)$ on $X$.

Letting $L_{i, j}:=\||\nabla h_{i, j}|_j\|_{L^{\infty}}+\||\nabla f_j|_j\|_{L^{\infty}}$ shows
\begin{align}\label{prg}
&\int_{X_j}\left|{\bar g}_j(\nabla h_{i, j}, \nabla h_{i, j})-{\bar g}_j(\nabla f_{j}, \nabla f_{j})\right|\di \meas_j \nonumber \\
&\le \int_{X_j}|{\bar g}_j|_{HS}  |\nabla h_{i, j} \otimes \nabla h_{i, j}-\nabla f_j \otimes \nabla f_j|_{HS}\di \meas_j \nonumber \\
&=\int_{X_j}|{\bar g}_j|_{HS}  |\nabla h_{i, j} \otimes \nabla h_{i, j}-\nabla f_j \otimes \nabla h_{i, j}+\nabla f_j \otimes \nabla h_{i, j}-\nabla f_j \otimes \nabla f_j|_{HS}\di \meas_j \nonumber \\
&\le \int_{X_j}||{\bf {\bar g}}_j|_{HS}|_{HS} \left( |\nabla h_{i, j} -\nabla f_j|_j |\nabla h_{i, j}|_j+|\nabla f_j|_j |\nabla h_{i, j}-\nabla f_j|_j\right)\di \meas_j \nonumber \\
&\le 2L_{i, j}\||{\bf {\bar g}}_j|_{HS}\|_{L^2} \||\nabla (h_{i, j}-f_j)|_j\|_{L^2}.
\end{align}
Then the right hand side of (\ref{prg}) converges to $0$ when letting $j \to \infty$ and then letting $i \to \infty$.
Combining these observations with the $L^2$-weak convergence of ${\bar g}_j(\nabla h_{i, j}^2, \nabla h_{i, j}^2)$ to ${\bar g}(\nabla h_i^2, \nabla h_i^2)$ yields that ${\bar g}_j(\nabla f^2_{j}, \nabla f_{j}^2)$ $L^2$-weakly converge to ${\bar g}(\nabla f^2, \nabla f^2)$, which completes the proof.
\end{proof}

{From} now on we focus on 
noncollapsed $\RCD^*(K, N)$ spaces,
namely $\RCD^*(K, N)$ spaces $(X, \dist, \meas)$ with $\mathrm{supp}\,\meas =X$ and $\meas=\mathcal{H}^N$. Such spaces were introduced and studied in \cite{DePhillippisGigli} where they proved the following facts which generalize important properties of noncollapsed Ricci limit spaces \cite{CheegerColding1} to the $\RCD$ setting.

\begin{theorem}\label{aaa4ff}
If $(X, \dist, \mathcal{H}^N)$ is a noncollapsed $\RCD^*(K, N)$ space, then $N=\dim_{\dist, \mathcal{H}^N}(X)$. Moreover
\begin{equation}\label{eqrigi}
\lim_{r \downarrow 0}\frac{\mathcal{H}^N(B_r(x))}{\omega_Nr^N} \le 1\qquad\forall x\in X.
\end{equation}
The equality in \eqref{eqrigi} holds if and only if $x \in \mathcal{R}_N$.

Finally, for all $v >0$,
the $p\mathbb{G}_W$-distance and the pointed Gromov-Hausdorff distance induces the same compact topology
on the set $\mathcal{M}(K, N, v)$ of all isometry classes of pointed noncollapsed $\RCD^*(K, N)$ spaces $(Y, \dist, \mathcal{H}^N, y)$ with $\mathcal{H}^N(B_1(y)) \ge v$, 
\end{theorem}

The almost rigidity of \eqref{eqrigi} shown in the next proposition is a direct consequence of \cite[Th.1.3 and 1.5]{DePhillippisGigli}.

\begin{proposition}\label{almostrigidity}
Let $(X, \dist, \mathcal{H}^N)$ be a noncollapsed $\RCD^*(K, N)$ space, $x \in X$ and $\epsilon >0$.
Assume that $\frac{\mathcal{H}^N(B_r(x))}{\omega_Nr^N}\ge 1-\epsilon$ for some $r\le 1$. 
Then, for all $t<1$, the $p\mathbb{G}_W$-distance between $(X, t^{-1}\dist, \mathcal{H}^N(B_t(x))^{-1}\mathcal{H}^N, x)$ and $(\mathbb{R}^N, \dist_{\mathbb{R}^N}, \omega_N^{-1}\mathcal{H}^N, 0_N)$ is at most $\Psi (\epsilon, t/r, r; K, N)$.
\end{proposition}
\begin{proof}
The proof is achieved by contradiction.
If the above statement does not hold, there exist pointed noncollapsed $\RCD^*(K, N)$ spaces $(X_i, \dist_i, \mathcal{H}^N, x_i)$, positive numbers $t_i \downarrow 0, r_i \downarrow 0$ and $\tau>0$ such that $\frac{\mathcal{H}^N(B_{r_i}(x_i))}{\omega_Nr_i^N} \to 1$, $t_i/r_i \to 0$, and that the $p\mathbb{G}_W$-distance (denoted by $D_i$ for short) between $(X_i, t_i^{-1}\dist_i, \mathcal{H}^N(B_{t_i}(x_i))^{-1}\mathcal{H}^N, x_i)$ and $(\mathbb{R}^N, \dist_{\mathbb{R}^N}, \omega_N^{-1}\mathcal{H}^N, 0_N)$ is at least $\tau$.
Then for all $R>0$ by the Bishop-Gromov theorem (\ref{eq:BishopGromov}) we have
\begin{align*}
\lim_{i \to \infty}\frac{\mathcal{H}^N(B_{2R}^{t_i^{-1}\dist_i}(x_i))}{\omega_N(2R)^N}=\lim_{i \to \infty}\frac{\mathcal{H}^N(B_{2Rt_i}^{\dist_i}(x_i))}{\omega_N(2Rt_i)^N}
&=\lim_{i \to \infty}\frac{\mathcal{H}^N(B_{2Rt_i}^{\dist_i}(x_i))}{\mathrm{Vol}_{K, N} (2Rt_i)} \cdot \frac{\mathrm{Vol}_{K, N} (2Rt_i)}{\omega_N(2Rt_i)^N} \\
&\ge \lim_{i \to \infty}\frac{\mathcal{H}^N(B^{\dist_i}_{r_i}(x_i))}{\omega_Nr_i^N} = 1
\end{align*}
because of $\lim_{r \downarrow 0}\frac{\mathrm{Vol}_{K, N} (r)}{\omega_Nr^N}=1$.
Thus by \cite[Th.1.5]{DePhillippisGigli}, $({\bar B}_R^{t_i^{-1}\dist_i}(x_i), t_i^{-1}\dist_i, x_i)$ pointed Gromov-Hausdorff converge to $({\bar B}_R(0_N), \dist_{\mathbb{R}^N}, 0_N)$. This implies that $(X_i, t_i^{-1}\dist_i, x_i)$ pointed Gromov-Hausdorff converge to $(\mathbb{R}^N, \dist_{\mathbb{R}^N}, 0_N)$. 
By Theorem~\ref{aaa4ff}, $D_i$ is infinitesimal, which contradicts $D_i \ge \tau$.
\end{proof}

We are now in a position to improve Theorem~\ref{th:convergence3}, including the case when $t=0$, 
thus giving a positive answer to our first question \eqref{ddrt} in the setting of noncollapsed spaces.
For the proof, let us recall the maximal function theorem for a compact $\RCD (K, N)$ space $(X, \dist, \meas)$:
\begin{equation}\label{maxfunc}
\meas (\{M(f)>t\}) \le \frac{C(K, N, d)}{t}\int_X|f|\di \meas \qquad \forall t>0,\,\forall f \in L^1(X, \meas),
\end{equation}
where $d$ is a constant with $\mathrm{diam}\,X\le d$, and $M(f)(x):=\sup_{r>0}\fint_{B_r(x)}|f|\di \meas$ (see for instance \cite{Heinonen} for the proof).

\begin{theorem}\label{thm: mgh l2 strong}
Let $(X_j, \dist_j, \mathcal{H}^N) \stackrel{mGH}{\to} (X, \dist, \mathcal{H}^N)$ be a sequence of compact noncollapsed $\RCD^*(K, N)$ spaces
with uniformly bounded diameter.
Then $t_j\mathcal{H}^N(B_{\sqrt{t_j}}(x))g_{t_j}^{X_j}$ $L^2$-strongly converge to $c_Ng^X$ for all $t_j \to 0^+$.
In particular $\omega_Nt_j^{(N+2)/2}g_{t_j}^{X_j}$ $L^2$-strongly converge to $c_Ng^X$.
\end{theorem}
\begin{proof}
First let us check the $L^2$-weak convergence.
For that, by Lemma \ref{lem:suf}, it is enough to prove that for all $L^2$-strong convergent sequences $h_j \in \mathrm{LIP} (X_j, \dist_j), f_j \in \mathrm{Test}F(X_j, \dist_j, \mathcal{H}^N)$ to $h \in \mathrm{LIP} (X, \dist), f \in \mathrm{Test}F(X, \dist, \mathcal{H}^N)$, respectively, with 
$$L:=\sup_j(\|h_j\|_{L^{\infty}}+\||\nabla h_j|_j\|_{L^{\infty}}+\|f_j\|_{L^{\infty}}+\||\nabla f_j|_j\|_{L^{\infty}}+\|\Delta^j f_j\|_{L^{\infty}})<\infty,$$ 
it holds that as $j \to \infty$
\begin{align}\label{goal}
\int_{X_j}\int_{X_j}t_j\mathcal{H}^N(B_{\sqrt{t_j}}(x))\langle \nabla_xp_j(x, y, t_j), h_j(x)\nabla f_j(x)\rangle^2_j\di \mathcal{H}^N(x)\di \mathcal{H}^N(y) \to c_N\int_X|h\nabla f|^2\di \mathcal{H}^N.
\end{align}
Fix $0<\epsilon <1$.
For all $x \in \mathcal{R}_N \cap H(f)$, there exists $0<r(x)<\epsilon$ such that $(\overline{B}_t(x), \dist, x)$ is $(\epsilon t)$-Gromov-Hausdorff close to $(\overline{B}_t(0_N), \dist_{\mathbb{R}^N}, 0_N)$ for all $0<t \le 2r(x)$ and 
\begin{equation}\label{fd}
\fint_{B_t(x)}||\nabla f|^2-|\nabla f|^{*2}(x)|\di \mathcal{H}^N \le \epsilon \qquad \forall t\in (0, 2r(x)].
\end{equation}
Then, applying Vitali's theorem to the cover $\mathcal{B}:=\{\overline{B}_t(x)\}_{x \in \mathcal{R}_N \cap H(f), t<r(x)}$ of $\mathcal{R}_N \cap H(f)$  we obtain a disjoint subcover $\hat{\mathcal{B}}:=\{\overline{B}_{r_i}(x_i)\}_{i \in \mathbb{N}} \subset \mathcal{B}$ such that  
$$
\mathcal{R}_N \cap H(f) \setminus \bigcup_{i=1}^l \overline{B}_{r_i}(x_i) \subset \bigcup_{i=\ell+1}^{\infty}\overline{B}_{5r_i}(x_i) \qquad \forall\ell\in \mathbb{N}.
$$
Take $\ell=N_0$ with $\sum_{i=N_0+1}\mathcal{H}^N(\overline{B}_{r_i}(x_i)) \le \epsilon$. In particular 
\begin{align}\label{alin}
\mathcal{H}^N(X)&=\sum_{i=1}^{N_0}\mathcal{H}^N(\overline{B}_{r_i}(x_i)) \pm \sum_{i=N_0+1}^{\infty}\mathcal{H}^N(\overline{B}_{5r_i}(x_i)) =\sum_{i=1}^{N_0}\mathcal{H}^N(\overline{B}_{r_i}(x_i)) \pm \Psi (\epsilon; K, N).
\end{align}
For all $i=1, 2, \ldots, N_0$, fix a convergent sequence $(x_{i, j})\subset X_j $ with $x_{i,j}\to x_i \in X$.
Note that for all such $i$, \eqref{fd} yields
\begin{equation}\label{8y8}
\fint_{B_{2r_i}(x_{i, j})}||\nabla f_j|_j^2-|\nabla f|^{*2}(x_i)|\di \mathcal{H}^N\le 2\epsilon
\end{equation}
for any sufficiently large $j$.

For any sufficiently large $j$, since $(\overline{B}_{2r_i}(x_{i, j}), \dist_j, x_{i, j})$ is $(2\epsilon r_i)$-Gromov-Hausdorff close to $(\overline{B}_{2r_i}(0_N), \dist_{\mathbb{R}^N}, 0_N)$, we see that $(\overline{B}_{r_i}(y), \dist_j, y)$ is $(3\epsilon r_i)$-Gromov-Hausdorff close to $(\overline{B}_{r_i}(0_N), \dist_{\mathbb{R}^N}, 0_N)$ for all $y \in \overline{B}_{r_i}(x_{i, j})$.
In particular, Theorem~\ref{aaa4ff} (after rescaling $r_i^{-1}\dist_j$) yields
$$
\frac{\mathcal{H}^N(B_{r_i}(y))}{\mathcal{H}^N(B_{r_i}(0_N))} =1 \pm \Psi (\epsilon; K, N).
$$
Applying the Bishop-Gromov inequality with (\ref{eqrigi}) yields
\begin{equation}
\frac{\mathcal{H}^N(B_{t}(y))}{\mathcal{H}^N(B_{t}(0_N))} =1 \pm \Psi (\epsilon; K, N) \quad \forall t\in (0, r_i].
\end{equation}
Thus Proposition \ref{almostrigidity} yields that the $p\mathbb{G}_W$-distance between $(X_j, t^{-1}\dist_j,  \mathcal{H}^N(B_t(y))^{-1}\mathcal{H}^N, y)$ and $(\mathbb{R}^N, \dist_{\mathbb{R}^N},  \omega_N^{-1}\mathcal{H}^N, 0_N)$ is at most  $\Psi (t/r_i, \epsilon; K, N)$.
Combining this (as $t:=\sqrt{s}$) with Corollary \ref{applied} yields that for all $i$, for all sufficiently large $j$, 
\begin{align*}
&\int_{X_j}s\mathcal{H}^N(B_{\sqrt{s}}(x))\langle \nabla_xp_j(x, y, s), h_j(x)\nabla f_j(x)\rangle^2_j\di \mathcal{H}^N(x)\\
&=c_Nh_j(y)^2\fint_{B_{\sqrt{s}}(y)}|\nabla f_j|_j^2(x)\di \mathcal{H}^N(x)\pm \Psi(\sqrt{s}/r_i, \epsilon; K, N, L) \quad \forall y \in B_{r_i}(x_i), 
\forall s\in (0,r_i^2].
\end{align*}
In particular
\begin{align}\label{key quant}
&\int_{X_j}t_j\mathcal{H}^N(B_{\sqrt{t_j}}(x))\langle \nabla_xp_j(x, y, t_j), h_j(x)\nabla f_j(x)\rangle^2_j\di \mathcal{H}^N(x) \nonumber \\
&=c_Nh_j(x_{i, j})^2\fint_{B_{\sqrt{t_j}}(y)}|\nabla f_j|_j^{2}(x)\di \mathcal{H}^N(x)\pm \Psi(\epsilon; K, N, L) \quad \forall y \in B_{r_i}(x_{i, j}).
\end{align}
On the other hand, letting $d:=\sup_j\mathrm{diam}\,(X_j, \dist_j)$ with (\ref{maxfunc}) and (\ref{8y8}) shows
\begin{align*}
&\mathcal{H}^N\left(\left\{y \in B_{r_i}(x_{i, j}); \left|\fint_{B_{\sqrt{t_j}}(y)}|\nabla f_j|^2(x)\di \mathcal{H}^N(x)-|\nabla f|^{*2}(x_i)\right|>\sqrt{\epsilon}\right\}\right) \\
&\le C_1(K, N, d)\sqrt{\epsilon}\mathcal{H}^N(B_{r_i}(x_{i, j}))
\end{align*}
which easily yields
\begin{equation}\label{key quant1}
\int_{B_{r_i}(x_{i, j})}\fint_{B_{\sqrt{t_j}}(y)}|\nabla f_j|_j^{2}(x)\di \mathcal{H}^N(x)\di \mathcal{H}^N(y) =(|\nabla f|^{*2}(x_i) \pm C_2(K, N, d)\sqrt{\epsilon})\mathcal{H}^N(B_{r_i}(x_{i, j})).
\end{equation}

Note that by the gradient heat kernel estimate (\ref{eq:equi lip}), the $L^2$ norm of the function 
\begin{equation}\label{ddx}
y \mapsto \int_{X_j}t_j\mathcal{H}^N(B_{\sqrt{t_j}}(x))\langle \nabla_xp_j(x, y, t_j), h_j\nabla f_j\rangle^2_j\di \mathcal{H}^N(x)
\end{equation}
is bounded from above by a constant depending only on $K, N$ and $L$.
Moreover it is clear that for all $\phi \in L^2(X_j, \dist_j, \mathcal{H}^N)$ with $\|\phi \|_{L^2} \le D<\infty$, if a Borel subset $A$ of $X_j$ satisfies $\mathcal{H}^N(X_j \setminus A) \le \delta$, then 
\begin{equation}\label{cr}
\int_{X_j}\phi\di \mathcal{H}^N=\int_{A}\phi \di \mathcal{H}^N \pm \delta^{1/2}D
\end{equation}
because of the Cauchy-Schwartz inequality.
Applying (\ref{cr}) for $A:=\bigcup_i^{N_0}B_{r_i}(x_{i, j})$ and the function defined in (\ref{ddx}) with (\ref{key quant}) and (\ref{key quant1}) shows that for any sufficiently large $j$
\begin{align}
&\int_{X_j}\int_{X_j}t_j\mathcal{H}^N(B_{\sqrt{t_j}}(x))\langle \nabla_xp_j(x, y, t_j), h_j(x)\nabla f_j(x)\rangle^2_j\di \mathcal{H}^N(x)\di \mathcal{H}^N(y) \nonumber \\
&=\sum_{i=1}^{N_0}\int_{B_{r_i}(x_{i, j})}\int_{X_j}t_j\mathcal{H}^N(B_{\sqrt{t_j}}(x))\langle \nabla_xp_j(x, y, t_j), h_j(x)\nabla f_j(x)\rangle^2_j\di \mathcal{H}^N(x)\di \mathcal{H}^N(y) \pm \Psi (\epsilon; K, N, L) \nonumber \\
&=\sum_{i=1}^{N_0}c_Nh_j(x_{i, j})^2|\nabla f|^{*2}(x_i)\mathcal{H}^N(B_{r_i}(x_{i, j})) \pm \Psi(\epsilon; d, K, N, L, V) \nonumber \\
&=(1 \pm \Psi (\epsilon; K, N, L))\sum_{i=1}^{N_0}c_Nh(x_i)^2|\nabla f|^{*2}(x_i)\mathcal{H}^N (B_{r_i}(x_i)) \pm \Psi(\epsilon; d, K, N, L, V) \nonumber \\
&=(1 \pm \Psi (\epsilon; K, N, L))\int_Xc_N|h\nabla f|^2\di \mathcal{H}^N \pm \Psi (\epsilon; d, K, N, L, V),
\end{align}
where $V:=\sup_j\mathcal{H}^N(X_j)<\infty$.
Since $\epsilon$ is arbitrary, we have \eqref{goal} and then the claimed $L^2$-weak convergence.

In order to improve this to the $L^2$-strong convergence, let us remark that under the same notation as above, by an argument similar to 
the proof of Proposition~\ref{rem:blowuptensor}, we can prove that for all $z \in B_{r_i}(x_i)$ and all $z_j \in B_{r_i}(x_{i, j})$ with $z_j \to z$,
\begin{equation}
\fint_{B_{\sqrt{t_j}}(z_j)}F_j(x, t_j)\di \mathcal{H}^N(x)=N (c_N)^2 \pm \Psi (\epsilon; K, N)
\end{equation}
for any sufficiently large $j$, where 
$$
F_j(x, t):=\left(t\mathcal{H}^N(B_{\sqrt{t}}(x))\right)^2\left|\int_{X_j}\dist_xp_j(x, y, t)\otimes \dist_xp_j(x, y, t) \di \mathcal{H}^N(y)\right|_{HS}^2.
$$
Therefore we have
\begin{align}
&\int_{X_j}\fint_{B_{\sqrt{t_j}}(z)}F_j(x, t_j)\di \mathcal{H}^N(x) \di \mathcal{H}^N(z) \nonumber \\
&=\sum_{i=1}^{N_0}\int_{B_{r_i}(x_{i, j})}\fint_{B_{\sqrt{t_j}}(z)}F_j(x, t_j)\di \mathcal{H}^N(x) \di \mathcal{H}^N(z) \pm \Psi (\epsilon; K, N) \nonumber \\
&=\sum_{i=1}^{N_0}N(c_N)^2\mathcal{H}^N(B_{r_i}(x_{i, j})) \pm \Psi (\epsilon; K, N) \nonumber \\
&=(1 \pm \Psi (\epsilon; K, N))N(c_N)^2\mathcal{H}^N(X) \pm \Psi (\epsilon; K, N) \nonumber \\
&=(1 \pm \Psi (\epsilon; K, N))\int_X|c_N{\bf g}^X|_{HS}^2\di \mathcal{H}^N \pm \Psi(\epsilon; K, N)
\end{align}
which yields
\begin{equation}
\int_{X_j}\fint_{B_{\sqrt{t_j}}(z)}F_j(x, t_j)\di \mathcal{H}^N(x) \di \mathcal{H}^N(z) \to \int_X|c_N{\bf g}^X|_{HS}^2\di \mathcal{H}^N 
\end{equation}
because $\epsilon$ is arbitrary.
Note that it is easy to check by Proposition \ref{almostrigidity} that for all $i=1, 2, \ldots, N_0$,
$$
\left| 1-\int_{B_{\sqrt{t_j}}(z)}\frac{1}{\mathcal{H}^N(B_{\sqrt{t_j}}(x))}\di \mathcal{H}^N(x)\right| \le \Psi (\epsilon; K, N) \qquad \forall z \in B_{r_i}(x_{i, j})
$$
for any sufficiently large $j$. Thus an argument similar to that in the begining of Subsection~\ref{proof} shows 
\begin{align*}
\lim_{j \to \infty}\int_{X_j}|t_j\mathcal{H}^N(B_{\sqrt{t_j}}(x)){\bf g}_{t_j}^{X_j}|_{HS}^2\di \mathcal{H}^N(x)
&=\lim_{j\to \infty}\int_{X_j}F_j(x, t_j)\di \mathcal{H}^N(x) \\
&= \lim_{j \to \infty}\int_{X_j}\fint_{B_{\sqrt{t_j}}(z)}F_j(x, t_j)\di \mathcal{H}^N(x) \di \mathcal{H}^N(z) \\
&= \int_X|c_N{\bf g}^X|_{HS}^2\di \mathcal{H}^N
\end{align*}
which completes the proof of the desired $L^2$-strong convergence. 

Finally, the remaining convergence result comes from Corollary~\ref{cor:ahl}.
\end{proof}

Let us give positive answers to the remaining questions. 

\begin{theorem}\label{thm:quantitativ}
For all $\epsilon>0$ and $1\le p <\infty$, there exists $t_0:=t_0(K, N, v, d, \epsilon, p)>0$ such that any compact noncollapsed 
$\RCD^*(K, N)$ space $(X, \dist, \mathcal{H}^N)$ with $\mathcal{H}^N(X) \ge v$ and $\mathrm{diam}\,(X, \dist) \le d$ satisfies
\begin{equation}\label{eq:quant lp}
\||\omega_Nt^{(N+2)/2}{\bf g}_t-c_N{\bf g}|_{HS}\|_{L^p} \le \epsilon \qquad \forall t\in (0,t_0].
\end{equation}
\end{theorem}
\begin{proof} Note that thanks to \eqref{eq:riem est}, it is enough to check the statement in the case when $p=2$ only.
Assume it does not hold. Then there exist $\epsilon_0 >0$, $t_i \to 0^+$ and a sequence of compact noncollapsed $\RCD^*(K, N)$ 
spaces $(X_i, \dist_i, \mathcal{H}^N)$ with $\mathrm{diam}\,(X_i, \dist_i) \le d$, $\mathcal{H}^N(X_i)\ge v$ satisfying
\begin{equation}\label{uy8}
\||\omega_Nt_i^{(N+2)/2}{\bf g}_{t_i}^{X_i}-c_N{\bf g}^{X_i}|_{HS}\|_{L^2} \ge \epsilon_0 \quad \forall i.
\end{equation}
Thanks to Theorem~\ref{aaa4ff} we know that, up to a subsequence, $(X_i, \dist_i, \mathcal{H}^N)$ converge in the
measured Gromov-Hausdorff sense to a compact noncollapsed $\RCD^*(K, N)$ space $(X, \dist, \mathcal{H}^N)$.
Then, Theorem~\ref{thm: mgh l2 strong} yields a contradiction.
\end{proof}

\begin{theorem}\label{label}
For all $d>0$, $v>0$, $\epsilon>0$ and $1\le p <\infty$ let $t_0:=t_0(K, N, v, d, \epsilon, p) >0$ be given by
Theorem~\ref{thm:quantitativ}. Then 
for all $0<t \le t_0$ and any compact noncollapsed 
$\RCD^*(K, N)$ space $(X, \dist, \mathcal{H}^N)$ with  $\mathcal{H}^N(X) \ge v$ and $\mathrm{diam} (X, \dist) \le d$, we have
\begin{equation}\label{pppp0}
\||\omega_Nt^{(N+2)/2}{\bf g}_t^\ell-c_N{\bf g}|_{HS}\|_{L^p}\le \epsilon \qquad \forall\ell\ge N_0=N_0(K, N, v, d, \epsilon, p, t),
\end{equation}
where ${\bf g}_t^\ell$ is the finite-dimensional approximation in \eqref{eq:finalfinal}.
\end{theorem}
\begin{proof} Again, thanks to \eqref{eq:riem est}, it is enough to check the statement in the case when $p=2$ only.
It suffices to check that for all $t >0$ and $\epsilon >0$, there exists $N_0:=N_0(K, N, v, d, t, \epsilon) \ge 1$ such that for all $\ell\ge N_0$ and any compact noncollapsed $\RCD^*(K, N)$ space $(X, \dist, \mathcal{H}^N)$ with $\mathcal{H}^N(X) \ge v$ and $\mathrm{diam}\,(X, \dist) \le d$, we have
\begin{equation}\label{n558}
\||\omega_Nt^{(N+2)/2}({\bf g}_t-{\bf g}_t^\ell)|_{HS}\|_{L^2} \le \epsilon
\end{equation}
because applying \eqref{n558} for $t\le t_0$ yields \eqref{pppp0}.

Assume that \eqref{n558} is not satisfied.
Then, as in the proof of Theorem~\ref{thm:quantitativ}, there exist $\epsilon_0 >0$, $N_j \to \infty$ and a mGH-convergent sequence $(X_j, \dist_j, \mathcal{H}^N) \stackrel{mGH}{\to} (X, \dist, \mathcal{H}^N)$ of compact noncollapsed $\RCD^*(K, N)$ spaces such that 
\begin{equation}\label{wdp}
\||\omega_Nt^{(N+2)/2}({\bf g}_{t}^{X_j}-({\bf g}_{t}^{X_j})^{N_j})|_{HS}\|_{L^2}\ge \epsilon_0.
\end{equation}
Theorem~\ref{th:convergence3} with Lemma~\ref{lem:conveigenfunct} yields 
\begin{align*}
\||\omega_Nt^{(N+2)/2}({\bf g}_t^X-({\bf g}_t^X)^\ell)|_{HS}\|_{L^2}&=\lim_{j \to \infty}\||\omega_Nt^{(N+2)/2}({\bf g}_t^{X_j}-({\bf g}_t^{X_j})^\ell)|_{HS}\|_{L^2} \\
&\ge \limsup_{j \to \infty}\||\omega_Nt^{(N+2)/2}({\bf g}_{t}^{X_j}-({\bf g}_{t}^{X_j})^{N_j})|_{HS}\|_{L^2}\ge \epsilon_0,\nonumber
\end{align*}
for all $\ell$, which is a contradiction because the left hand side converges to $0$ as $\ell\to \infty$,
\end{proof}

Theorems~\ref{thm:quantitativ} and \ref{label} are new even for smooth Riemannian manifolds and Alexandrov spaces. 
Moreover, recall that these convergence results are sharp because of Remark~\ref{counterex}.

\section{Appendix: expansion of the heat kernel}\label{sec:appendix}

Throughout this section we assume that $(X,\dist,\meas)$ is a compact metric measure space with $\meas(X)=1$ (this is not restrictive, up to a normalization), $\mathrm{diam}\,(X, \dist)>0$ and $\supp\meas=X$. 

The main aim of this section is to provide a complete proof of the expansions
\begin{equation}\label{eq:expansion1'}
p(x,y,t) = \sum_{i \ge 0} e^{- \lambda_i t} \phi_i(x) \phi_i (y) \qquad \text{in $C(X \times X)$}
\end{equation}
for any $t>0$ and
\begin{equation}\label{eq:expansion2'}
p(\cdot,y,t) = \sum_{i \ge 0} e^{- \lambda_i t} \phi_i(y) \phi_i \qquad \text{in $H^{1,2}(X,\dist,\meas)$}
\end{equation}
for any $y \in X$ and $t>0$, where $p$ denotes the locally H\"older representative of the heat kernel in the case when, in addition,
$(X,\dist,\meas)$ is a $\RCD^{*}(K,N)$ space. Our goal is to justify the convergence of the series in \eqref{eq:expansion1'} and
\eqref{eq:expansion2'}: as soon as this is secured, a standard argument shows that they provide good representatives of the
heat kernel. Here and in the sequel $0=\lambda_0<\lambda_1 \le \lambda_2 \le \cdots \to +\infty$ are the eigenvalues of $-\Delta$, and $\phi_0, \phi_1, \phi_2, \ldots$ are corresponding eigenfunctions forming an orthonormal basis of $L^2(X,\meas)$, with $\phi_0\equiv 1$.

In the following proposition we obtain an explicit estimate on the $L^\infty$ norm and the Lipschitz constant of
eigenfunctions of $-\Delta$ in terms of the size of eigenvalues. Recall that, under our assumptions, we can use
the continuous version of the $\phi_i$  which are even Lipschitz \cite{Jiang}.
It is worth pointing out that a local $(2. 2)$-Poincar\'e inequality for $\RCD^*(K, N)$ spaces (recall just after \eqref{eq:locPoincaré}) 
yields $\lambda_1\ge C(K, N, d)>0$ if $\mathrm{diam} (X) \le d$.

\begin{proposition} \label{prop:Jiang} 
Assuming that $(X, \dist, \meas)$ is a compact $\RCD^*(K, N)$ space, and that $D>0$ is such that $\mathrm{diam} (X, \dist) \le D$ and $\lambda_i \geq D^{-2}$, one has for some $C=C(K, N, D)>0$
\[
\|\phi_i\|_{L^\infty} \leq C \lambda_i^{N/4}, \qquad \| \nabla \phi_i \|_{L^\infty} \leq C \lambda_i^{(N+2)/4}.
\]
\end{proposition}
\begin{proof}
Without loss of generality we assume that $K \leq 0$. 
Since $\phi_i$ is an eigenfunction with eigenvalue $\lambda_i$, for all $t > 0$ one has $h_t \phi_i = e^{-\lambda_i t} \phi_i$, where $h_t$ denotes the heat flow, so that
\[
\phi_i(x) = e^{\lambda_i t} \int_X p(x,y,t) \phi_i(y) \di \meas(y), \quad \forall x \in X.
\]
Now by \eqref{eq:gaussian}
\[
\begin{split}
|\phi_i(x)| 
&\leq e^{\lambda_i t} \int_X p(x,y,t) |\phi_i(y)| \di \meas(y) \\
&\leq e^{\lambda_i t} \| \phi_i \|_{L^2} \left( \int_X p(x,y,t)^2 \di \meas(y) \right)^{1/2} \\
&\leq e^{\lambda_i t}  \frac{C_1}{\meas(B_{\sqrt{t}}(x))}\exp(C_2 t) 
\left( \int_X \exp\left( - \frac{2 \dist^2(x,y)}{5 t} \right) \di \meas(y) \right)^{1/2},
\end{split}
\]
where in the last line we used the normalization $\|\phi_i\|_{L^2} = 1$ and constants $C_i$ depending on $K$ and $N$. 
Now we use a scaled version of Lemma~\ref{lem:volume}, and get for $t \leq D^2$
\[
\begin{split}
|\phi_i(x)|& \leq C_1 C\exp(\lambda_i t + C_2 t) \frac{1}{\sqrt{\meas(B_{\sqrt{t}}(x))}}\\
&= C_1 C \exp(\lambda_i t + C_2 t) \sqrt{\frac{\meas(B_{D}(x))}{\meas(B_{\sqrt{t}}(x))}},
\end{split}
\]
where the constant $C$ depends on $D, K$ and $N$. The last equality follows by the assumption that 
$\mathrm{diam} (X, \dist) \leq D$ and $\meas(X) = 1$. By the Bishop-Gromov inequality \eqref{eq:BishopGromov}, we find
\[
\frac{\meas(B_{\sqrt{t}}(x))}{\mathrm{Vol}_{K,N}(\sqrt{t})}\geq \frac{\meas(B_D(x))}{\mathrm{Vol}_{K,N}(D)}.
\]
Therefore,
\[
|\phi_i(x)| \leq C_1 C \exp(\lambda_i t + C_2 t) \sqrt{\frac{\mathrm{Vol}_{K,N}(D)}{\mathrm{Vol}_{K,N}(\sqrt{t})}}.
\]
We choose $t = 1/\lambda_i$ to conclude the proof. 

Let us now prove the second inequality. We start from
\[
\phi_i(x) = e^{\lambda_i t} \int_X p(x,y,t) \phi_i(y) \di \meas(y), \quad \forall x \in X,
\]
to derive for $\meas$-almost all $x \in X$,
\[
\begin{split}
|\nabla \phi_i (x)| &\leq e^{\lambda_i t}\int_X |\nabla_x p(x,y,t)| |\phi_i(y)| \di \meas(y) \\
&\leq e^{\lambda_i t} \| \phi_i \|_{L^2} \left(\int_X |\nabla_x p(x,y,t) |^2 \di \meas(y) \right)^{1/2}.
\end{split}
\]
By the gradient bound in \eqref{eq:equi lip} and again Lemma~\ref{lem:volume} we get
\[
\begin{split}
|\nabla \phi_i(x)| &\leq e^{\lambda_i t} \frac{C_3 }{\sqrt{t} \meas(B_{\sqrt{t}}(x))} \exp(C_4 t) \left(\int_X \exp\left( - \frac{\dist^2(x,y)}{5t}\right)\di \meas(y) \right)^{1/2}\\
&\leq C_3 C \exp(\lambda_i t + C_4 t ) \frac{1}{\sqrt{t \meas(B_{\sqrt{t}}(x))}}\\
&\leq C_3 C \exp\left( \lambda_i t +C_4 t\right) \sqrt{\frac{\meas(B_D(x))}{t \meas(B_{\sqrt{t}}(x))}}.
\end{split}
\]
We use the Bishop-Gromov inequality once more to get
\[
|\nabla \phi_i(x) | \leq C_3 C \exp(\lambda_i t + C_4 t)\sqrt{\frac{\mathrm{Vol}_{K,N}(D)}{t \mathrm{Vol}_{K,N}(\sqrt{t})}}.
\]
Again, we pick $t = 1/\lambda_i$ to conclude the proof.
\end{proof}

The following result, well-known for compact Riemannian manifolds,  
provides a polynomial lower bound for the eigenvalues of $- \Delta$. The 
estimate we provide is not sharp, but sufficient for our purposes.

\begin{proposition}\label{prop:lowerbound}
Let $D>0$.
Assuming that $(X, \dist, \meas)$ is a $\RCD^*(K, N)$ space with $\mathrm{diam}(X, \dist) \leq D$ and
$\lambda_1 \ge D^{-2}$, there exists a constant $C_0=C_0(D, K, N)>0$ such that
$$
\lambda_i\ge C_0 i^{2/N}  \qquad \forall i \ge 1.
$$
\end{proposition}
\begin{proof}
Take $i \ge 1$, write $E_i = \Span(\phi_1, \ldots, \phi_i)$.
We claim that there exists $f_o \in E_i$ such that $\sup f_o^2\ge i$ and 
$\|f_o\|_2 = 1$. Let us define the continuous function $F= \sum_{j=1}^{i} \phi_j^2$ and let $p\in X$ be a maximum point of $F$.
Then 
$$
f_o(x):=\frac{1}{\sqrt{F(p)}}\sum_{j=1}^i \phi_j(p)\phi_j(x)
$$
satisfies $\|f_o\|_2=1$ and $f_o(p)=\sqrt{F(p)}$, so that 
\begin{equation}\label{eq:lowerboundeigenvalues1}
i = \dim E_i= \int_X F\di\meas \le F(p) \leq  \sup f_o^2.
\end{equation}

We claim now that there exists $C_1>0$ depending only on $K$ and $N$ such that 
\begin{equation}\label{eq:claim1}
\sup |f| \le C_1 \lambda_i^{N/4} \| f \|_{L^2}\qquad\forall f\in E_i.
\end{equation}
Using this claim with $f=f_o$ together with \eqref{eq:lowerboundeigenvalues1}, we obtain the stated lower bound on $\lambda_i$.

Proposition~\ref{prop:Jiang} yields that for all $a_j \in \mathbb{R}$ we have
\begin{align*}
\left|\sum_{j=1}^ia_j\phi_j\right|^2 \le \sum_{j,\,k=1}^i|a_j||a_k||\phi_j||\phi_k| &\le 
C(D, K, N)\sum_{j,\,k=1}^i\lambda_j^{N/4}\lambda_k^{N/4}|a_j||a_k| \\
&\le C(D, K, N)\lambda_i^{N/2}\sum_{j=1}^i(a_j)^2
\end{align*}
which proves (\ref{eq:claim1}).
\end{proof}
We are now in a position to conclude. The first expansion ~\eqref{eq:expansion1'} is a direct consequence of 
Propositions \ref{prop:Jiang} and \ref{prop:lowerbound}. The second expansion \eqref{eq:expansion2'} 
follows, thanks to the simple observation that $\|\nabla \phi_i\|_2^2 = \lambda_i$.

\end{document}